\documentclass[12pt, reqno]{amsart}
\usepackage{amsmath}
\usepackage{amssymb}
\usepackage{enumerate}
\usepackage{mathrsfs}
\usepackage{amsfonts}
\usepackage{color}
\usepackage{vmargin}
\usepackage{amsthm}
\usepackage{graphicx} 
\usepackage{esint}
\usepackage{mathtools}
\usepackage{MnSymbol}
\usepackage[shortcuts]{extdash}
\usepackage{hyperref}

\hypersetup{
	colorlinks=true,
	linkcolor=blue,
	citecolor=blue,
	urlcolor=blue
}

\allowdisplaybreaks
\usepackage[T1]{fontenc}
\usepackage{setspace}
\DeclareMathOperator*{\esssup}{ess\,sup}
\DeclareMathOperator*{\essinf}{ess\,inf}
\DeclareMathOperator*{\diam}{diam}

\def\XXint#1#2#3{{\setbox0=\hbox{$#1{#2#3}{\int}$ }
		\vcenter{\hbox{$#2#3$ }}\kern-.6\wd0}}

\long\def\symbolfootnote[#1]#2{\begingroup%
	\def\thefootnote{\fnsymbol{footnote}}\footnote[#1]{#2}\endgroup}

\setmarginsrb{20mm}{20mm}{20mm}{20mm}{10mm}{10mm}{10mm}{10mm}

\newtheoremstyle{remark}
{}{}{}{}{\bfseries}{.}{.5em}{{\thmname{#1 }}{\thmnumber{#2}}{\thmnote{ (#3)}}}
\theoremstyle{remboldstyle}
\newtheorem{remark}{Remark}[section]

\RequirePackage{amsthm}
\newcommand{\measurerestr}{%
	\,\raisebox{-.127ex}{\reflectbox{\rotatebox[origin=br]{-90}{$\lnot$}}}\,%
}
\newtheorem{defi}{Definition}[section]

\newtheorem{tw}{Theorem}[section]
\newtheorem{corollary}{Corollary}[section]
\newtheorem{lem}[tw]{Lemma}

\newtheorem{prop}[tw]{Proposition}

\newtheorem{ex}{Example}

\numberwithin{equation}{section}

\makeatletter
\def\@tocline#1#2#3#4#5#6#7{\relax
	\ifnum #1>\c@tocdepth 
	\else
	\par \addpenalty\@secpenalty\addvspace{#2}%
	\begingroup \hyphenpenalty\@M
	\@ifempty{#4}{%
		\@tempdima\csname r@tocindent\number#1\endcsname\relax
	}{%
		\@tempdima#4\relax
	}%
	\parindent\z@ \leftskip#3\relax \advance\leftskip\@tempdima\relax
	\rightskip\@pnumwidth plus4em \parfillskip-\@pnumwidth
	#5\leavevmode\hskip-\@tempdima
	\ifcase #1
	\or\or \hskip 1em \or \hskip 2em \else \hskip 3em \fi%
	#6\nobreak\relax
	\hfill\hbox to\@pnumwidth{\@tocpagenum{#7}}\par
	\nobreak
	\endgroup
	\fi}
\makeatother

\def\={\hspace{-3mm}&=&\hspace{-3mm}}

\begin{document}
	
	\date{}

	\title[\bf Compact embeddings of Sobolev, Besov, and Triebel-Lizorkin spaces]{\bf Compact embeddings of variable exponent Sobolev, Besov, and Triebel-Lizorkin spaces on metric measure spaces}
	
	\author{Micha{\l} Dymek}
	\maketitle
	
	\begin{abstract}
	We study compact embeddings of Sobolev, Besov, and Triebel-Lizorkin spaces with variable exponents on both bounded and unbounded metric measure spaces. We establish sufficient conditions for compactness, and under additional assumptions, we show that they are also necessary. Moreover, we investigate the influence of isometry group actions on the compactness of embeddings. In particular, we answer the open question posed by P. Górka in \cite{Górka}, proving a Berestycki-Lions type theorem.	\end{abstract}
	\bigskip
	
	\medskip
	
	
	\tableofcontents
	
	\section{Introduction}
	
	Compact embeddings of certain function spaces, such as Sobolev spaces, play a crucial role in modern analysis, particularly in the theory of nonlinear partial differential equations. One of the most fundamental results concerning compact embeddings is the Rellich-Kondrachov theorem. It states that if $\Omega \subseteq \mathbb R^n$ is a bounded domain with sufficiently regular boundary, then for every $p\in [1,n)$ there holds the compact embedding $W^{1,p}(\Omega) \hookrightarrow \hookrightarrow L^q(\Omega)$, where $1\leq q < p^*$ and $p^*=np/(n-p)$ (see, for instance \cite{Adams}). It is well-known that the assumption of boundedness of $\Omega$ cannot be omitted.
	
	However, in the case of unbounded domains, an interesting result was proven by Berestycki and Lions. Specifically, let $p\in [1,\infty]$ and $n\in \mathbb N$ with $n\geq 2$. Define the subspace $W^{1,p}_r(\mathbb R^n)$ of $W^{1,p}(\mathbb R^n)$ consisting of radially symmetric functions. Berestycki and Lions \cite{Berestycki, Lions} showed that if $p<n$, then the space $W_r^{1,p}(\mathbb R^n)$ is compactly embedded into $L^q(\mathbb R^n)$ for all $q$ satisfying $p<q<p^*$.

	The study of compact embeddings of Sobolev spaces has been extended to much more general settings than those mentioned above. For instance, many results have been obtained on Riemannian manifolds \cite{farkas, Gaczkowski, krytyczny, hebey}, locally compact Abelian groups \cite{kostrzewa}, metric measure spaces \cite{samko, zwartestale, Ambrosio, gacz, Górka, nonlinear, hajlaszkoskela, zmienneq, kalamajska} and for variable exponents \cite{fan, macios}. In particular, Rellich-Kondrachov type theorems have been established for Haj{\l}asz-Sobolev and Poincar\'e spaces defined on metric measure spaces by G\'orka and S{\l}abuszewski in \cite{nonlinear}, Haj{\l}asz and Koskela \cite{hajlaszkoskela} and \cite{kalamajska}. Recently, Alvarado, G\'orka and S{\l}abuszewski \cite{zwartestale} identified sufficient and necessary conditions guaranteeing compactness of embeddings for the fractional Haj{\l}asz-Sobolev spaces $M^{s,p}$, the Haj{\l}asz-Besov spaces $N^s_{p,q}$ and the Haj{\l}asz-Triebel-Lizorkin spaces $M^s_{p,q}$ (which were firstly introduced in \cite{hajlaszpierwszy}, \cite{koskela} and \cite{yang}). The Berestycki-Lions theorem was extended by G\'orka in \cite{Górka} to Sobolev spaces on metric measure spaces. Namely, if $(X,d,\mu)$ is a metric measure space and $H$ is a subgroup of the measure-preserving isometries of $(X,d)$, then, under suitable assumptions on $(X,d,\mu)$ and the orbits under the action of $H$, the embedding $M^{1,p}_H(X,d,\mu) \hookrightarrow L^q(X,\mu)$ is compact, where $M^{1,p}_H(X,d,\mu)$ is the subspace of $M^{1,p}(X,d,\mu)$ consisting of $H$-invariant functions. Moreover, the Berestycki-Lions theorem has been established on Riemannian manifolds by Hebey and Vaugon \cite{hebey} and Skrzypczak and Tintarev \cite{skrzypczak}, on the Heisenberg group by Balogh and Kristály \cite{baloch}. It has also been generalized to the variable exponent setting by Fan, Zhao, and Zhao \cite{fan} on Euclidean spaces and by Gaczkowski, Górka, and Pons \cite{Gaczkowski} on complete Riemannian manifolds. Compact embeddings in a presence of isometry group action was also investigated by Skrzypczak in \cite{skrzypczak2} for Besov and Triebel-Lizorkin spaces.
	
	The primary goal of this paper is to extend the Rellich-Kondrachov and Berestycki-Lions theorems to the variable exponent Haj{\l}asz-Sobolev spaces $M^{s(\cdot),p(\cdot)}$, the Haj{\l}asz-Besov $N^{s(\cdot)}_{p(\cdot),q(\cdot)}$ and the Haj{\l}asz-Triebel-Lizorkin $M^{s(\cdot)}_{p(\cdot),q(\cdot)}$ spaces defined on metric measure spaces. Additionally, we investigate necessary conditions for the compactness of embeddings, thus providing a partial characterization of compact embeddings for Sobolev, Besov, and Triebel-Lizorkin spaces in this broad setting.
	
	The remainder of this paper is organized as follows. In the second section, we introduce the notation and some measure-theoretic concepts used in the statements of our main results. We also recall the definitions of variable exponent function spaces, which are essential for our considerations, including Lebesgue, Sobolev, Triebel-Lizorkin, Besov, and Hölder spaces on metric measure spaces. The third section is dedicated to our main results. We first analyze the compact embeddings on measure spaces with finite or integrable measure, characterizing the embeddings into Lebesgue spaces within certain classes of metric measure spaces and proving the Rellich-Kondrachov theorem. Next, we state and prove the Berestycki-Lions theorem and discuss necessary conditions for embeddings of group invariant spaces. The final subsection focuses on compact embeddings between Sobolev, Besov, and Triebel-Lizorkin spaces.
	\section{Preliminaries}
	
	\subsection{Metric measure spaces}
	
	Let $(X,d,\mu)$ be a metric measure space equipped with a metric $d$ and Borel measure~$\mu$. Denote by
	\begin{equation*}
	B(x,r)=\left\{y\in X: d(x,y)<r \right\}
	\end{equation*}
	the open ball centered at $x\in X$ with radius $r\in (0,\infty)$. Throughout this paper, we assume that for every ball $B(x,r)\subseteq X$
	\begin{equation*}
	0<\mu(B(x,r))<\infty.
	\end{equation*}
	It is well known (see \cite{borel}) that the existence of a measure $\mu$ satisfying this assumption is equivalent to the separability of the metric space. Therefore, all metric measure spaces considered in this paper are assumed to be separable.
	
	Next, we introduce the concept of metric measure spaces with variable dimension. Such spaces have been studied in \cite{zmienneq} and \cite{zanurzenia}. Suppose that $Q: X \rightarrow (c,+\infty)$ is a bounded function, where $c\in (0,\infty)$. We say that the measure $\mu$ is \texttt{lower Ahlfors $Q(\cdot)$-regular}, if there exists $b_1\in (0,\infty)$ such that for all $x\in X$ and $r\in (0,1]$
	\begin{equation*}
	\mu\left(B(x,r)\right) \geq b_1r^{Q(x)}.
	\end{equation*}
	Similarily, we say that $\mu$ is \texttt{upper Ahlfors $Q(\cdot)$-regular} if there is $b_2\in (0,\infty)$ such that
	\begin{equation*}
	b_2r^{Q(x)} \geq \mu\left(B(x,r)\right)
	\end{equation*}
	for all $x\in X$ and $r\in (0,1]$. We say that $\mu$ is \texttt{$Q(\cdot)$-Ahlfors regular} if it is both upper and lower Ahlfors $Q(\cdot)$-regular.
	
	Now, let us recall the notion of doubling metric spaces. If $(X,d,\mu)$ is a metric measure space, we say that the measure $\mu$ is \texttt{doubling} if and only if there exists $C_d\in [1,\infty)$ such that
	\begin{equation*}
	\mu\left(B(x,2r)\right)\leq C_d\mu\left(B(x,r)\right) \textnormal{ for every ball } B(x,r)\subseteq X.
	\end{equation*}
	Furthermore, we say that metric space $(X,d)$ is \texttt{geometrically doubling} if and only if there exists constant $M\in (0,\infty)$ such that every ball $B(x,r) \subseteq X$ can be covered by at most $M$ balls with radius $r/2$. It is known that metric space equipped with a doubling measure is geometrically doubling (\cite{Coifman}). Conversely, in \cite{lukkainen} it was proven that every complete geometrically doubling metric space carries a doubling measure.
	
	Next, we state the covering characterization of geometrically doubling metric spaces (see \cite{zanurzenia} for the proof), which will be useful in some proofs within this paper.
	\begin{lem}\label{pokryciowy}
		Let $(X,d)$ be a metric space and $r\in (0,\infty)$. The following statements are equivalent:
		\begin{enumerate}
			\item[(i)] $(X,d)$ is geometrically doubling;
			\item[(ii)] There exist $A, B\in (0,\infty)$, a sequence $\left\{x_n\right\}_{n=1}^{\infty}\subseteq X$ such that the collection of balls $\left\{B(x_n,r)\right\}_{n=1}^{\infty}$ covers $X$ and, for any $\delta\in (r,\infty)$ each point $x\in X$ belongs to at most $A\left(\delta/r\right)^{B}$ balls $B(x_n,\delta)$.
		\end{enumerate}
	\end{lem}	
	
	Throughout the paper, we will also use the following characterization of totally bounded metric measure spaces.
	
	\begin{lem}\cite{zwartestale}\label{charakteryzacja}
	Let $(X,d,\mu)$ be a metric measure space. Then, the following statements are equivalent:
	\begin{enumerate}
		\item[(i)] $(X,d)$ is totally bounded;
		\item[(ii)] $\mu(X)<\infty$ and for every $r\in (0,\infty)$ we have $\displaystyle h(r):=\inf_{x\in X} \mu(B(x,r))>0$;
	\end{enumerate}
	\end{lem}
	
	Next, we highlight an integrability condition for measures, first introduced in \cite{zwartestale}. Similar to the case of constant exponents, it will turn out that this condition is closely related to the compact embeddings of Sobolev, Besov, and Triebel-Lizorkin spaces into Lebesgue spaces.
	
	\begin{defi}
	Let $(X,d,\mu)$ be a metric measure space. We say that measure $\mu$ is \texttt{integrable} if for every $r\in (0,\infty)$
	\begin{equation*}
		\int_X \frac{1}{\mu(B(x,r))}\mbox{d}\mu(x) <\infty.
	\end{equation*}
	\end{defi}
	
	\begin{ex}
	From Lemma \ref{charakteryzacja} it immediately follows that if $(X,d,\mu)$ is totally bounded metric measure space, then $\mu$ is integrable.
	\end{ex}
	
	The following example is taken from \cite{zwartestale}.
	
	\begin{ex}\label{gaussowska} Let $((0,\infty),d,\mu)$, where $d=\left|\cdot - \cdot\right|$ is the Euclidean distance and $\mu$ is defined as $\mbox{d}\mu=e^{\varepsilon x^{\beta}}\mbox{d}x$ for any fixed $\beta \in (1,\infty)$ and $\varepsilon \in \left\{-1,1\right\}$. Then, the measure $\mu$ is integrable.
	\end{ex}

	\subsection{Measure-preserving isometries}
	In this section, we recall the notion of measure-preserving isometries introduced in \cite{Górka}. Namely, let $\textnormal{Iso}(X,d)$ be the group of isometries of the metric space $(X,d)$. Then, by $\textnormal{Iso}_\mu(X,d)$ we denote the subgroup of measure-preserving isometries, i.e.
	\begin{equation*}
	\textnormal{Iso}_\mu(X,d):=\left\{\varphi\in \textnormal{Iso}(X,d): \varphi_\# \mu=\mu \right\},
	\end{equation*}
	where $\varphi_\#\mu$ is the pushforward measure, that is $\varphi_{\#}\mu(A):=\mu(\varphi^{-1}(A))$ for every Borel set $A\subseteq X$.  In further considerations, we shall assume that $H$ is a subset of $\textnormal{Iso}_\mu(X,d)$. We define the orbit of $x\in X$ under the action of $H$ as
	\begin{equation*}
	H(x)=\left\{h(x): h \in H\right\}.
	\end{equation*}
	For given $x\in X$ and $r\in (0,\infty)$ we consider the following quantity
	\begin{equation*}
	M_H(x,r):=\sup\left\{\# A: A\subseteq H(x) \wedge B(\xi,r)\cap B(\eta,r)=\emptyset \textnormal{ for } \xi,\eta \in A \textnormal{ such that } \xi \neq \eta\right\},
	\end{equation*}
	where $\# A$ denotes cardinality of the set $A$. In other words, $M_H(x,r)$ is the lowest upper bound for the number of non-overlapping balls centered in $H(x)$ with radius $r$.

	\subsection{The space of measurable functions}
	
	Let $\left(X,\mu\right)$ be measure space such that $\mu(X)<\infty$. Then, by $L^0(X,\mu)$ we denote the space of all measurable functions on $X$ that are finite $\mu$-almost everywhere. The space $L^0(X,\mu)$ endowed with a metric
	\begin{equation*}
	d(f,g):=\int_X \frac{\left|f(x)-g(x)\right|}{1+\left|f(x)-g(x)\right|}\mbox{d}\mu(x), \textnormal{ where } f,g\in L^0(X,\mu)
	\end{equation*}
	is a complete metric space. It is well known that the convergence in this metric is equivalent to convergence in measure.
	
	In \cite{zwartestale}, the following characterization of totally bounded sets in $L^0(X,\mu)$ was proven.
	
	\begin{tw}\label{krotov}
	Let $(X,d)$ be a separable metric space and let $\mu$ be a Borel measure on $X$ such that $\mu(X)<\infty$. Then, a subset $\mathcal{F} \subseteq L^0(X,\mu)$ is totally bounded if for any $\varepsilon\in (0,\infty)$ there exist $\delta\in (0,\infty)$ and $\lambda\in (0,\infty)$ such that for any $u\in \mathcal{F}$ there is a set $E(u)\subseteq X$, satisfying following conditions:
	\begin{enumerate}
		\item[(i)] $\mu\left(E(u)\right) < \varepsilon$;
		\item[(ii)] $\left|u(x)-u(y)\right| < \varepsilon$ for any $x,y\in X \setminus E(u)$ such that $d(x,y)<\delta$;
		\item[(iii)] $\left|u(x)\right| \leq \lambda$ for $x\in X \setminus E(u)$. 
\end{enumerate}
\end{tw}
	
	\subsection{Variable exponent Lebesgue spaces}
	Now we recall the notion of the variable exponent Lebesgue space. Let $(X,\mu)$ be a measure space. For measurable function $p:X \rightarrow (0,\infty]$ and measurable subset $E\subseteq X$ we use the following notation
	\begin{equation*}
	p_E^+:=\esssup_{x\in E} p(x), \hspace{4mm} p_E^-:=\essinf_{x\in E} p(x).
	\end{equation*}
	If $E=X$, then we abbreviate $p^+=p_X^+$ and $p^-=p_X^-$. Moreover, if $p_1,p_2: X \rightarrow (0,\infty)$ are measurable functions, we write $p_1 \ll p_2$ if and only if $\left(p_2-p_1\right)^- >0$. 
	
	By a variable exponent we mean every measurable function $p:X \rightarrow (0,\infty]$ such that $p^->0$. The class of all variable exponents on $X$ is denoted by $\mathcal{P}(X)$. Additionally, we define the class of bounded exponents as $\mathcal{P}_b(X):=\mathcal{P}(X) \cap L^{\infty}(X,\mu).$
	
	If $p\in \mathcal{P}(X)$, then we say that $u \in L^{p(\cdot)}(X,\mu)$ if and only if $u$ is measurable and there exists $\lambda \in (0,\infty)$ such that following semi-modular is finite
	\begin{equation*}
	\rho_{p(\cdot)}(\lambda u):=\int_X \varphi_{p(x)}\left(\lambda\left|u(x)\right|\right)\mbox{d}\mu(x),
	\end{equation*}
	where
	\begin{equation*}
		\varphi_p(t):= \left\{ \begin{array}{lll} t^p,& \textnormal{ for } p\in (0,\infty),\\ 0,& \textnormal{ for } t\leq 1, p=\infty,\\ \infty,& \textnormal{ for } t>1, p=\infty.\end{array}\right.
	\end{equation*}
	for every $t\geq 0$.
	The space $L^{p(\cdot)}(X,\mu)$ is a quasi-Banach space due to the following Luxemburg quasi-norm
	\begin{equation*}
	\left\|u\right\|_{L^{p(\cdot)}\left(X,\mu\right)}:=\inf\left\{\lambda\in (0,\infty):  \rho_{p(\cdot)}\left(\frac{u}{\lambda}\right)\leq 1 \right\} \textnormal{ for } u\in L^{p(\cdot)}(X,\mu).
	\end{equation*}
	Throughout this article, we denote by $\kappa_{p(\cdot)}\in[1,\infty)$ the constant appearing in the quasi-triangle inequality satisfied by the above quasi-norm. If $p^- \geq 1$, then $L^{p(\cdot)}(X,\mu)$ becomes a Banach space. The variable exponent Lebesgue space is a special case of the Musielak-Orlicz space. If $p$ is constant, then $L^{p(\cdot)}(X,\mu)$ coincides with an ordinary Lebesgue space. More information about variable exponent Lebesgue spaces can be found in books \cite{cruzuribe, diening}
	
	In some statements throughout this paper, additional assumptions on the exponents are necessary. Namely, if $(X,d)$ is a metric space, we say that function $p:X \rightarrow \mathbb{R}$ is \texttt{locally log-H\"older continuous} on $X$, if
	\begin{equation*}
	\scalebox{1.1}{$\exists_{C_{\log}(p)\in (0,\infty)}\hspace{1mm}\forall_{x,y\in X}$} \hspace{1mm} \left|p(x)-p(y)\right|\leq \frac{C_{\log}(p)}{\log\left(\mbox{e}+1/d(x,y)\right)}.
	\end{equation*}
	Moreover, we say that $p$ satisfies \texttt{log-H\"older decay condition at infinity} with the basepoint $x_0 \in X$ if
	\begin{equation*}
	\scalebox{1.1}{$\exists_{p_{\infty}\in \mathbb R} \hspace{1mm} \exists_{C_{\infty}(p)\in (0,\infty)}\hspace{1mm}\forall_{x\in X}$} \hspace{1mm} \left|p(x)-p_{\infty}\right| \leq \frac{C_{\infty}(p)}{\log\left(e+d(x,x_0)\right)}.
	\end{equation*}
	We will say that $p$ is \texttt{globally log-H\"older continuous} on $X$, if it is locally log-H\"older continuous and satisfies log-H\"older decay condition at infinity.

	Subsequently, we define the following class of regular exponents
	\begin{align*}
	\mathcal{P}^{\textnormal{log}}(X)&:=\left\{p\in \mathcal{P}(X): 1/p \textnormal{ is locally log-H\"older continuous} \right\},\\
	\mathcal{P}_{\infty}^{\textnormal{log}}(X)&:=\left\{p\in \mathcal{P}(X): 1/p \textnormal{ is globally log-H\"older continuous} \right\}.
	\end{align*}
	Moreover, we set $\mathcal{P}_b^{\log}(X):=\mathcal{P}_b(X) \cap \mathcal{P}^{\log}(X).$
	
	Furthermore, in the context of a metric measure space $(X,d,\mu)$, we say that the function $p: X \to \mathbb R$ belongs to the class $\mathcal{P}_{\mu}^{\log}(X)$, if $p\in \mathcal{P}_b(X)$ and there exists a constant $C_{\log,\mu}(p)\in (0,\infty)$ such that for every ball $B\subseteq X$
	\begin{equation*}
		\mu(B)^{\frac{1}{p_{B}^+}-\frac{1}{p_B^-}} \leq C_{\log,\mu}(p)
	\end{equation*}
	and there exists $p_{\infty} \in \mathbb (0,\infty]$ such that
	\begin{equation*}
		1\in L^{s(\cdot)}(X,\mu), \textnormal{ where } \frac{1}{s(x)}:=\left|\frac{1}{p(x)}-\frac{1}{p_{\infty}}\right|.
	\end{equation*}
	It is well-known (see \cite{maximal}) that if $p\in \mathcal{P}_{\infty}^{\log}(X)$ and $\mu$ is doubling, then $p\in \mathcal{P}_{\mu}^{\log}(X)$.

	The proof of the following lemma can be found in \cite{zanurzenia}.
	\begin{lem}\label{loglemma}
		Let $(X,d,\mu)$ be a metric measure space and $B:=B(z,r)$, where $z \in X$, $r\in (0,\infty)$. Assume that $t\in \mathcal{P}^{\log}_b(B)$. Then, the following statements are true.
		\begin{enumerate} 			\item[(i)] For any $R\in [2r,\infty)$ and $x \in B$ the following inequality holds
			\begin{equation*}
				e^{-C_{\log}(1/t)}(1/R)^{1/t_B^-} \leq (1/R)^{1/t(x)}\leq e^{C_{\log}(1/t)}(1/R)^{1/t_B^+}.
			\end{equation*}
			\item[(ii)] We have
			\begin{equation*}
				r^{1/t_B^+ - 1/t_B^-}\leq e^{C_{\log}(1/t)} 2^{\frac{1}{t_B^-}-\frac{1}{t_B^+}}.
			\end{equation*}
			\item[(iii)] There exists constant $M(r,t)\geq 1$ depending on $r,t$ and log-H\"older constant of $1/t$ such that for every $x,y\in B$
			\begin{equation*}
				\frac{1}{M(r,t)}d(x,y)^{\frac{1}{t(y)}} \leq d(x,y)^{\frac{1}{t(x)}} \leq M(r,t)d(x,y)^{\frac{1}{t(y)}}.
			\end{equation*}
		\end{enumerate}
	\end{lem}

	In addition, for a given family $H$ of measure-preserving isometries of the space $(X,d,\mu)$ we define the class of $H$-invariant exponents as follows
	\begin{equation*}
	\mathcal{P}_H(X)
	:=\left\{p\in \mathcal{P}(X): \scalebox{1.1}{$\forall_{\varphi \in H}$} \hspace{2mm} p \circ \varphi =p\right\}.
	\end{equation*}
	
	We now state some well-known and useful results concerning the properties of the semi-modular $\rho_{p(\cdot)}$ and quasi-norm $\left\|\cdot\right\|_{L^{p(\cdot)}\left(X,\mu\right)}$. Their proofs can be found in \cite{cruzuribe}. First, we present a version of H\"older's inequality adapted to variable exponent spaces.
	\begin{prop}\label{holder1}
		Let $(X,\mu)$ be a measure space and suppose $p\in \mathcal{P}_b(X)$ with $p^- > 1$. Let $p'\in \mathcal{P}_b(X)$ be conjugate exponent to $p$, 
		i.e. $ \displaystyle\scalebox{1.1}{$\forall_{x\in X}$}\hspace{1mm}\frac{1}{p(x)}+\frac{1}{p'(x)}=1$. Then, for every $f\in L^{p(\cdot)}(X,\mu)$, $g\in L^{p'(\cdot)}(X,\mu)$ the following inequality holds
		\begin{equation*}
		\left\|fg\right\|_{L^1\left(X,\mu\right)}\leq 2\left\|f\right\|_{L^{p(\cdot)}\left(X,\mu\right)}\left\|g\right\|_{L^{p'(\cdot)}\left(X,\mu\right)}.
		\end{equation*}
	\end{prop}
	
	Next, we have a scaling property of the modular.
	
	\begin{lem}\label{skalowanie}
	Let $(X,\mu)$ be a measure space and $p\in \mathcal{P}_b(X), \alpha\in (0,\infty).$ Then, for every non-negative function $u\in L^{\alpha p(\cdot)}(X,\mu)$
	\begin{equation*}
		\left\| u \right\|_{L^{\alpha p(\cdot)}(X,\mu)} = \left\| u^{\alpha} \right\|_{L^{p(\cdot)}(X,\mu)}^{\frac{1}{\alpha}}.
	\end{equation*}
	\end{lem}

	Furthermore, we often need the following estimates relating the quasi-norm and the semi-modular.

	\begin{lem}\label{modularnorma}
	Let $(X,\mu)$ be a measure space and $p\in \mathcal{P}_b(X)$. Then, the following inequality
	\begin{equation*}
	\min\left\{ \rho_{p(\cdot)}\left(u\right)^{\frac{1}{p^-}}, \rho_{p(\cdot)}\left(u\right)^{\frac{1}{p^+}} \right\} \leq \left\|u\right\|_{L^{p(\cdot)}(X,\mu)} \leq \max\left\{ \rho_{p(\cdot)}\left(u\right)^{\frac{1}{p^-}}, \rho_{p(\cdot)}\left(u\right)^{\frac{1}{p^+}} \right\}
	\end{equation*}
	holds for all $u\in L^{p(\cdot)}(X,\mu)$. In particular,
	\begin{equation*}
		\left\|u \right\|_{L^{p(\cdot)}(X,\mu)} \leq 1 \iff \rho_{p(\cdot)}\left(u\right) \leq 1.
	\end{equation*}
\end{lem}	

Using these lemmas, we prove the following proposition, which will be useful in the main part of the paper.

	\begin{prop}\label{jednakowo}
		Let $(X,d,\mu)$ be a metric measure space such that $M:=\sup\left\{\mu\left(B(x,1)\right): x\in X\right\}$ is finite. Suppose that $p_1,p_2\in \mathcal{P}_b(X)$ are such that $p_1\ll p_2$ in $X$. Then, there exists a constant $C=C(p_1,p_2,M)\in (0,\infty)$ such that for every $y\in X$, $r\in (0,1)$ and $u\in L^{p_2(\cdot)}(B(y,r))$ we have that $u\in L^{p_1(\cdot)}(B(y,r))$ and
		\begin{equation*}
		\left\| u\right\|_{L^{p_1(\cdot)}(B(y,r))}\leq C \left\|u \right\|_{L^{p_2(\cdot)}(B(y,r))}.
		\end{equation*}
	\end{prop}
	
	\begin{proof}
		Fix $y\in X$. By the homogeneity of quasi-norm $\left\|\cdot\right\|_{L^{p(\cdot)}}$, without loss of generality we may assume that $\left\|u\right\|_{L^{p_2(\cdot)}(B(y,r))}= 1$. Lemma \ref{modularnorma} then implies that $\rho_{p_2(\cdot)}\left(u\chi_{B(y,r)}\right)\leq 1$. Set $q=p_2/p_1$, then $\rho_{q(\cdot)}\left(\left|u\right|^{p_1}\chi_{B(y,r)}\right)\leq 1.$ Applying Lemma \ref{modularnorma} again we obtain 
		\begin{equation*}\left\| \left|u\right|^{p_1} \right\|_{L^{q(\cdot)}(B(y,r))}\leq 1.
		\end{equation*} Using H\"older inequality for the exponents $q=p_2/p_1$ and $q'=p_2/(p_2-p_1)$ we get
		\begin{align*}
		\rho_{p_1(\cdot)}\left(u \chi_{B(y,r)}\right)=\int_{B(y,r)} \left|u(x)\right|^{p_1(x)}\mbox{d}\mu(x)& \leq 2 \left\| \left|u\right|^{p_1(\cdot)} \right\|_{L^{q(\cdot)}(B(y,r))} \left\|1\right\|_{L^{q'(\cdot)}(B(y,r))} \\ &\leq 2 \max\left\{\mu(B(y,r))^{\frac{1}{(q')_{B(y,r)}^+}},\mu(B(y,r))^{\frac{1}{(q')_{B(y,r)}^-}}\right\} \\ & \leq 2 \max\left\{\left(M+1\right)^{\frac{1}{(q')_{B(y,r)}^+}},\left(M+1\right)^{\frac{1}{(q')_{B(y,r)}^-}}\right\} \\ & \leq 2\left(M+1\right)^{\frac{1}{(q')^-}}:=N.
		\end{align*} From the definition of the semi-modular $\rho_{p_1}$ we get
		\begin{equation*}
		\rho_{p_1(\cdot)}\left(\frac{u}{N^{\frac{1}{p_1^-}}}\chi_{B(y,r)}\right)\leq \frac{1}{N}\rho_{p_1(\cdot)}\left(u\chi_{B(y,r)}\right)\leq 1.
		\end{equation*}
		Lemma \ref{modularnorma} yields $\left\|u\right\|_{L^{p_1(\cdot)}(B(y,r))} \leq N^{\frac{1}{p_1^-}}$. Setting $C:=N^{\frac{1}{p_1^-}}$ completes the proof.
	\end{proof}
	
	Using H\"older inequality, one can prove the following lemma.
	\begin{lem}\label{wlozenielp}\cite{zanurzenia}
		Let $(X,\mu)$ be a measure space with finite measure and fix $p,q\in \mathcal{P}_b(X)$ such that $q \gg p$. Then, $L^{q(\cdot)}(X,\mu) \subseteq L^{p(\cdot)}(X,\mu)$ and for every $u\in L^{q(\cdot)}(X,\mu)$, 
		\begin{equation*}
			\left\| u \right\|_{L^{p(\cdot)}(X,\mu)} \leq 2^{\frac{1}{p^-}}\max\left\{\left\|1 \right\|_{L^{t'(\cdot)}(X,\mu)}^{\frac{1}{p^+}}, \left\|1 \right\|_{L^{t'(\cdot)}(X,\mu)}^{\frac{1}{p^-}}\right\} \left\| u \right\|_{L^{q(\cdot)}(X,\mu)},
		\end{equation*}
		where $\displaystyle t:=q/p$.
	\end{lem}

Let us now recall the following version of Jensen's inequality for variable exponent Lebesgue spaces. The proof of this result can be found in \cite{maximal}.
\begin{prop}\label{jensen}
	Let $(X,d,\mu)$ be a metric measure space and $p \in \mathcal{P}^{\log}_{\mu}(X)$, where $p^- > 1$. Define $q: X \times X \to \mathbb R$ by
	\begin{equation*}
		\frac{1}{q(x,y)}:= \max\left\{\frac{1}{p(x)}-\frac{1}{p(y)},0\right\}.
	\end{equation*}
	Then, for any $\gamma \in (0,1)$ there exists $\beta\in (0,1)$ depending only on $\gamma$ and $C_{\log,\mu}(p)$ such that
	\begin{equation*}
		\left(\beta \fint_{B} \left|f(y)\right| \mbox{d}\mu(y)\right)^{p(x)} \leq \fint_{B} \left|f(y)\right|^{p(y)}\mbox{d}\mu(y)+ \fint_B \gamma^{q(x,y)}\chi_{\left\{y\in X:0<f(y)\leq 1\right\}}(y) \mbox{d}\mu(y)
	\end{equation*}
	for every ball $B \subseteq X$, $x\in B$ and $f\in L^{p(\cdot)}(X,\mu) + L^{\infty}(X,\mu)$ with\footnote{Let us recall that $L^{p(\cdot)}(X,\mu)+L^{\infty}(X,\mu):=\left\{f_0+f_1: f_0 \in L^{p(\cdot)}(X,\mu), f_1\in L^{\infty}(X,\mu)\right\}$. Moreover, we define $\left\|f\right\|_{L^{p(\cdot)}(X,\mu)+L^{\infty}(X,\mu)}:=\inf\left\{\left\|f_0\right\|_{L^{p(\cdot)}(X,\mu)}+\left\|f_1\right\|_{L^{\infty}(X,\mu)}: f=f_0+f_1, f_0 \in L^{p(\cdot)}(X,\mu), f_1\in L^{\infty}(X,\mu)\right\}.$} $\left\|f\right\|_{L^{p(\cdot)}(X,\mu)+L^{\infty}(X,\mu)}\leq 1.$
\end{prop}

	To finish this subsection, let us recall the following Vitali-type compactness theorem from \cite{bandaliyev}.
	
	\begin{tw}\label{hanson}
	Suppose that $(X,\mu)$ is a measure space such that $\mu(X)< \infty$ and $p\in \mathcal{P}_b(X)$. Then, the subset $\mathcal{F} \subseteq L^{p(\cdot)}(X,\mu)$ is totally bounded in $L^{p(\cdot)}(X,\mu)$ if and only if:
	\begin{enumerate}
		\item[(i)] $\mathcal{F}$ is totally bounded in $L^0(X,\mu)$;
		\item[(ii)] the family $\mathcal{F}$ is $p(\cdot)$-equi-integrable, that is,
		\begin{equation*}
		\scalebox{1.1}{$\forall_{\varepsilon\in (0,\infty)} \exists_{\delta\in (0,\infty)} \forall_{A\subseteq X}$} \hspace{2mm} \mu(A)<\delta \implies \sup_{u\in \mathcal{F}} \int_{A} \left|u(x)\right|^{p(x)}\mbox{d}\mu(x) < \varepsilon.
		\end{equation*}
	\end{enumerate}
	\end{tw}
	
	\subsection{Mixed Lebesgue sequence spaces} Now, we recall definition of mixed Lebesgue sequence spaces, which play crucial role in the definition of variable exponent Besov and Triebel-Lizorkin spaces.

	Let $(X,\mu)$ be a measure space and $p\in \mathcal{P}_b(X)$, $q\in \mathcal{P}(X)$. We define the following semi-modular (see \cite{AH})
	\begin{equation*}
		\rho_{\ell^{q(\cdot)}(L^{p(\cdot)}(X,\mu))}\left(\left\{u_k\right\}_{k\in \mathbb Z}\right):=\sum_{k\in \mathbb Z} \inf\left\{\lambda_k\in (0,\infty): \rho_{p(\cdot)}\left(\frac{u_k}{\lambda_k^{\frac{1}{q(\cdot)}}}\right) \leq 1\right\}
	\end{equation*}
	for every sequence $\left\{u_k\right\}_{k\in \mathbb Z}\subseteq L^{p(\cdot)}(X,\mu)$,
	where we apply the convention $\lambda^{1/\infty}=1$.	Using this semi-modular, we define the \texttt{mixed Lebesgue sequence space} $\ell^{q(\cdot)}(L^{p(\cdot)}(X,\mu))$ as follows
	\begin{equation*}
		\ell^{q(\cdot)}(L^{p(\cdot)}(X,\mu)):=\left\{\left\{u_k\right\}_{k\in \mathbb Z}\subseteq L^{p(\cdot)}(X,\mu): \exists_{\mu\in (0,\infty)} \hspace{1mm} \rho_{\ell^{q(\cdot)}(L^{p(\cdot)}(X,\mu))}\left(\frac{\left\{u_k\right\}_{k\in \mathbb Z}}{\mu}\right) <\infty \right\}.
	\end{equation*}
	This is a quasi-normed space with the quasi-norm
	\begin{equation*}
		\left\| \left\{u_k\right\}_{k\in \mathbb Z} \right\|_{\ell^{q(\cdot)}(L^{p(\cdot)}(X,\mu))} := \inf\left\{\mu\in (0,\infty): \rho_{\ell^{q(\cdot)}(L^{p(\cdot)}(X,\mu))}\left(\frac{\left\{ u_k \right\}_{k\in \mathbb Z}}{\mu}\right) \leq 1 \right\}
	\end{equation*}for $\left\{u_k\right\}_{k\in \mathbb Z} \in \ell^{q(\cdot)}(L^{p(\cdot)}(X,\mu))$.
	More details on mixed Lebesgue sequence spaces $\ell^{q(\cdot)}(L^{p(\cdot)})$ can be found in \cite{AH, nasze, ghorba, kempka}.
	
	Moreover, we say that the sequence $\left\{u_k\right\}_{k\in \mathbb Z}$ of measurable functions belongs to the \texttt{mixed Lebesgue sequence space} $L^{p(\cdot)}(\ell^{q(\cdot)}(X,\mu))$ if and only if
	\begin{equation*}
		\left\| \left\{u_k\right\}_{k\in \mathbb Z} \right\|_{\ell^{q(\cdot)}} \in L^{p(\cdot)}(X,\mu),
	\end{equation*}
	where for $x\in X$ we define
	\begin{equation*}
		\left\| \left\{u_k(x)\right\}_{k\in \mathbb Z} \right\|_{\ell^{q(x)}}:= \left\{\begin{array}{ll} \displaystyle\sup_{k\in \mathbb Z} \left|u_k(x)\right|,& \textnormal{ if } q(x)=\infty,\\ \displaystyle \left(\sum_{k\in \mathbb Z} \left|u_k(x)\right|^{q(x)}\right)^{\frac{1}{q(x)}},& \textnormal{ if } q(x)<\infty.\end{array}\right.
	\end{equation*}
	The space $L^{p(\cdot)}(\ell^{q(\cdot)}(X,\mu))$ is also a quasi-normed space with the quasi-norm
	\begin{equation*}
		\left\| \left\{u_k\right\}_{k\in \mathbb Z} \right\|_{L^{p(\cdot)}(\ell^{q(\cdot)}(X,\mu))}:= \left\| \left\| \left\{u_k\right\}_{k\in \mathbb Z} \right\|_{\ell^{q(\cdot)}} \right\|_{L^{p(\cdot)}(X,\mu)}, \textnormal{ where } \left\{u_k\right\}_{k\in \mathbb Z} \in L^{p(\cdot)}(\ell^{q(\cdot)}(X,\mu)).
	\end{equation*}
	
	\begin{lem}\label{mieszane}\cite{AH}
		Let $(X,\mu)$ be a measure space and $p\in \mathcal{P}_b(X)$, $q\in \mathcal{P}(X)$.
		\begin{enumerate}
			\item[(i)] If $q$ is constant, then
			\begin{equation*}
				\left\| \left\{f_k\right\}_{k\in \mathbb Z} \right\|_{\ell^{q}(L^{p(\cdot)}(X,\mu))}= \left\| \left\{ \left\|f_k\right\|_{L^{p(\cdot)}(X,\mu)}\right\}_{k\in \mathbb Z} \right\|_{\ell^q}
			\end{equation*}
			for every $\left\{f_k\right\}_{k\in \mathbb Z} \subseteq L^{p(\cdot)}(X,\mu)$.
			\item[(ii)] If $q^+ <\infty$, then
			\begin{equation*}
				\rho_{\ell^{q(\cdot)}(L^{p(\cdot)}(X,\mu))}\left( \left\{f_k\right\}_{k\in \mathbb Z}\right)= \sum_{k\in \mathbb Z} \left\| \left|f_k\right|^{q(\cdot)} \right\|_{L^{\frac{p(\cdot)}{q(\cdot)}}(X,\mu)}
			\end{equation*}
			for every $\left\{f_k\right\}_{k\in \mathbb Z}\subseteq L^{p(\cdot)}(X,\mu)$. Moreover, $\left\{f_k\right\}_{k\in \mathbb Z}\in \ell^{q(\cdot)}(L^{p(\cdot)}(X,\mu))$ if and only if
			\begin{equation*}
				\sum_{k\in \mathbb Z} \left\| \left|f_k\right|^{q(\cdot)} \right\|_{L^{\frac{p(\cdot)}{q(\cdot)}}(X,\mu)}<\infty.
			\end{equation*}
		\end{enumerate}
	\end{lem}

	\subsection{Variable exponent Sobolev, Triebel-Lizorkin and Besov spaces on metric spaces}

	\indent Let $(X,d,\mu)$ be a metric measure space and let $u: X \to \overline{\mathbb{R}}$ be a measurable function which is finite $\mu$-almost everywhere. Let $s\in \mathcal{P}_b(X)$. We say that a non-negative function $g: X \to \mathbb \overline{\mathbb{R}}$ is a \texttt{scalar $s(\cdot)$-gradient} of $u$ if there is a set $N_u\subseteq X$ such that $\mu(N_u)=0$ and
	\begin{equation}\label{pointwise}
		\vert u(x)-u(y)\vert\leq d(x,y)^{s(x)}g(x)+d(x,y)^{s(y)}g(y)
	\end{equation}
	holds for all $x,y\in X\setminus N_u$. The collection of all scalar $s(\cdot)$-gradients of $u$ we denote by $\mathcal{D}^{s(\cdot)}(u)$.
	
	Moreover, we say that the sequence of non-negative measurable functions $\left\{g_k\right\}_{k\in \mathbb Z}$ is a \texttt{vector $s(\cdot)$-gradient} of $u$ if there exists a set $N_u \subseteq X$ such that $\mu(N_u)=0$ and
	\begin{equation*}
		\left|u(x)-u(y)\right|\leq d(x,y)^{s(x)}g_k(x)+d(x,y)^{s(y)}g_k(y)
	\end{equation*}
	for all $x,y\in X\setminus N_u$ satisfying $2^{-k-1}\leq d(x,y)<2^{-k}$. The collection of all vector $s(\cdot)$-gradients of $u$ we denote by $\mathbb D^{s(\cdot)}(u)$.
	
	For $s,p\in \mathcal{P}_b(X)$ we define the \texttt{homogeneous Haj{\l}asz-Sobolev space} $\dot{M}^{s(\cdot),p(\cdot)}(X,d,\mu)$ as the space of all measurable functions $u:X \to \overline{\mathbb{R}}$ which are finite almost $\mu$-everywhere and
	\begin{equation*}
		\left\|u\right\|_{\dot{M}^{s(\cdot),p(\cdot)}(X,d,\mu)}:=\inf_{g\in \mathcal{D}^{s(\cdot)}(u)}\left\|g\right\|_{L^{p(\cdot)}(X,\mu)}<\infty.
	\end{equation*}
	Moreover, we define the \texttt{non-homogeneous Haj{\l}asz-Sobolev space} $M^{s(\cdot),p(\cdot)}(X,d,\mu)$ as
	\begin{equation*}
		M^{s(\cdot),p(\cdot)}(X,d,\mu)=\dot{M}^{s(\cdot),p(\cdot)}(X,d,\mu) \cap L^{p(\cdot)}(X,\mu).
	\end{equation*}
	
	The space $M^{s(\cdot),p(\cdot)}(X,d,\mu)$ is a quasi-Banach space with the following quasi-norm
	\begin{equation*}
		\left\|u\right\|_{M^{s(\cdot),p(\cdot)}(X,d,\mu)}:=\left\|u\right\|_{L^{p(\cdot)}(X,\mu)}+\left\|u\right\|_{\dot{M}^{s(\cdot),p(\cdot)}(X,d,\mu)}.
	\end{equation*}
	
	Now we shall present the definition of variable exponent Haj{\l}asz-Triebel-Lizorkin and Haj{\l}asz-Besov spaces. 
	
	For $s,p\in \mathcal{P}_b(X)$, $q\in \mathcal{P}(X)$ we define
	\begin{enumerate}
		\item[$(i)$] the \texttt{homogeneous Haj{\l}asz-Triebel-Lizorkin space} $\dot{M}^{s(\cdot)}_{p(\cdot),q(\cdot)}(X,d,\mu)$ as the space of all measurable functions $u:X \to \overline{\mathbb{R}}$  which are finite almost $\mu$-everywhere and
		\begin{equation*}
			\left\|u\right\|_{\dot{M}^{s(\cdot)}_{p(\cdot),q(\cdot)}\left(X,d,\mu\right)}:=\inf_{g\in \mathbb D^{s(\cdot)}(u)} \left\| g \right\|_{L^{p(\cdot)}\left(\ell^{q(\cdot)}\left(X,\mu\right)\right)}<\infty.
		\end{equation*}
		\item[(ii)] the \texttt{homogeneous Haj{\l}asz-Besov space} $\dot{N}^{s(\cdot)}_{p(\cdot),q(\cdot)}(X,d,\mu)$ as the space of all measurable functions $u: X\to \overline{\mathbb{R}}$  which are finite almost $\mu$-everywhere and
		\begin{equation*}
			\left\|u\right\|_{\dot{N}^{s(\cdot)}_{p(\cdot),q(\cdot)}(X,d,\mu)}:=\inf_{g\in \mathbb D^{s(\cdot)}(u)} \left\| g \right\|_{\ell^{q(\cdot)}\left(L^{p(\cdot)}(X,\mu)\right)}<\infty.
		\end{equation*}
		
	\end{enumerate} 
	
	For $s,p\in \mathcal{P}_b(X)$, $q\in \mathcal{P}(X)$ we define \texttt{non-homogeneous Haj{\l}asz-Triebel-Lizorkin space} $M^{s(\cdot)}_{p(\cdot),q(\cdot)}(X,d,\mu)$ and \texttt{non-homogeneous Haj{\l}asz-Besov space} $N^{s(\cdot)}_{p(\cdot),q(\cdot)}(X,d,\mu)$ as
	\begin{align*}
		M^{s(\cdot)}_{p(\cdot),q(\cdot)}(X,d,\mu)&:= \dot{M}^{s(\cdot)}_{p(\cdot),q(\cdot)}(X,d,\mu) \cap L^{p(\cdot)}(X,\mu),\\
		N^{s(\cdot)}_{p(\cdot),q(\cdot)}(X,d,\mu)&:= \dot{N}^{s(\cdot)}_{p(\cdot),q(\cdot)}(X,d,\mu) \cap L^{p(\cdot)}(X,\mu).
	\end{align*}
	
	We endow spaces $M^{s(\cdot)}_{p(\cdot),q(\cdot)}(X,d,\mu)$, $N^{s(\cdot)}_{p(\cdot),q(\cdot)}(X,d,\mu)$ with the following quasi-norms
	\begin{align*}
		\left\|u\right\|_{M^{s(\cdot)}_{p(\cdot),q(\cdot)}(X,d,\mu)}&:=\left\|u\right\|_{L^{p(\cdot)}(X,\mu)}+\left\|u\right\|_{\dot{M}^{s(\cdot)}_{p(\cdot),q(\cdot)}(X,d,\mu)}, \textnormal{ where } u\in M^{s(\cdot)}_{p(\cdot),q(\cdot)}(X,d,\mu)\\
		\left\|u\right\|_{N^{s(\cdot)}_{p(\cdot),q(\cdot)}(X,d,\mu)}&:=\left\|u\right\|_{L^{p(\cdot)}(X,\mu)}+\left\|u\right\|_{\dot{N}^{s(\cdot)}_{p(\cdot),q(\cdot)}(X,d,\mu)}, \textnormal{ where } u\in N^{s(\cdot)}_{p(\cdot),q(\cdot)}(X,d,\mu).
	\end{align*}
	Both Haj{\l}asz-Triebel-Lizorkin and Haj{\l}asz-Besov spaces are quasi-Banach spaces due to their quasi-norms. They were first introduced\footnote{It is worth noting that the authors of \cite{zmiennepraca} used a slightly different definition of $M^{s(\cdot),p(\cdot)}$, $M^{s(\cdot)}_{p(\cdot),q(\cdot)}$, $N^{s(\cdot)}_{p(\cdot),q(\cdot)}$ than we do here. See \cite{zanurzenia} for the proof of the equivalence of these two definitions.} in \cite{zmiennepraca}. Moreover, in \cite{zmiennepraca} it was proven that $M^{s(\cdot)}_{p(\cdot),q(\cdot)}$, $N^{s(\cdot)}_{p(\cdot),q(\cdot)}$ coincide with classical Triebel-Lizorkin and Besov space on Euclidean space. More information about classical variable exponent Besov and Triebel-Lizorkin spaces can be found in \cite{AH, lizorkin, xu}.
	
	Now we introduce invariant Sobolev, Besov and Triebel-Lizorkin spaces. Namely, if $H \subseteq \textnormal{Iso}_\mu(X)$, then by $M^{s(\cdot),p(\cdot)}_H(X,d,\mu)$ we denote the subspace of $M^{s(\cdot),p(\cdot)}(X,d,\mu)$ consisting of $H$-invariant functions, i.e.
	\begin{equation*}
	M^{s(\cdot),p(\cdot)}_H(X,d,\mu):=\left\{u\in M^{s(\cdot),p(\cdot)}(X,d,\mu): \scalebox{1.1}{$\forall_{\varphi \in H}$} \hspace{2mm}u \circ \varphi = u\right\}.
	\end{equation*}
	Similarily, we define
	\begin{align*}
		&M^{s(\cdot),H}_{p(\cdot),q(\cdot)}(X,d,\mu):=\left\{u\in M^{s(\cdot)}_{p(\cdot),q(\cdot)}(X,d,\mu): \scalebox{1.1}{$\forall_{\varphi \in H}$} \hspace{2mm}u \circ \varphi = u\right\},\\
		&N^{s(\cdot),H}_{p(\cdot),q(\cdot)}(X,d,\mu):=\left\{u\in N^{s(\cdot)}_{p(\cdot),q(\cdot)}(X,d,\mu): \scalebox{1.1}{$\forall_{\varphi \in H}$} \hspace{2mm}u \circ \varphi = u\right\}.
	\end{align*} 
	This is easy to notice that $M^{s(\cdot),p(\cdot)}_H(X,d,\mu)$, $M^{s(\cdot),H}_{p(\cdot),q(\cdot)}(X,d,\mu)$ and $N^{s(\cdot),H}_{p(\cdot),q(\cdot)}(X,d,\mu)$ are closed subspaces of, respectively, $M^{s(\cdot),p(\cdot)}(X,d,\mu)$, $M^{s(\cdot)}_{p(\cdot),q(\cdot)}(X,d,\mu)$ and $N^{s(\cdot)}_{p(\cdot),q(\cdot)}(X,d,\mu)$. In particular, these are themselves quasi-Banach spaces.
	
	Below, we shall state some useful properties of Sobolev, Besov and Triebel-Lizorkin spaces. The proof of the following proposition can be found in \cite{zanurzenia}.
	
	\begin{prop}\label{pomiedzy}
		Let $(X,d,\mu)$ be a metric measure space. The following statements hold.
		\begin{enumerate}
			\item[(i)] If $p,s\in \mathcal{P}_b(X)$, $q_1,q_2\in \mathcal{P}(X)$ are such that $q_1\leq q_2$, then
			\begin{align*}
				\dot{M}^{s(\cdot)}_{p(\cdot),q_1(\cdot)}(X,d,\mu) \hookrightarrow \dot{M}^{s(\cdot)}_{p(\cdot),q_2(\cdot)}(X,d,\mu),\hspace{8mm}\dot{N}^{s(\cdot)}_{p(\cdot),q_1(\cdot)}(X,d,\mu) \hookrightarrow     \dot{N}^{s(\cdot)}_{p(\cdot),q_2(\cdot)}(X,d,\mu)
			\end{align*}
			and the operator semi-norms of the above embeddings are less or equal to one.
			\item[(ii)] If $p,s\in \mathcal{P}_b(X)$, then
			\begin{align*}
				\dot{M}^{s(\cdot)}_{p(\cdot),\infty}(X,d,\mu)=\dot{M}^{s(\cdot),p(\cdot)}(X,d,\mu)
			\end{align*}
			with equal quasi-semi-norms.
			\item[(iii)] If $p,s\in \mathcal{P}_b(X)$, it holds that
			\begin{equation*}
				\dot{M}^{s(\cdot)}_{p(\cdot),p(\cdot)}(X,d,\mu)=\dot{N}^{s(\cdot)}_{p(\cdot),p(\cdot)}(X,d,\mu)
			\end{equation*}
			with equal quasi-semi-norms.
			\item[(iv)] Let $s,t,p\in \mathcal{P}_b(X)$, $q_1,q_2\in \mathcal{P}(X)$ are such that $ t \gg s$. Then,
			\begin{align*}
				N^{t(\cdot)}_{p(\cdot),q_1(\cdot)}(X,d,\mu)\hookrightarrow N^{s(\cdot)}_{p(\cdot),q_2(\cdot)}(X,d,\mu), \hspace{8mm}
				M^{t(\cdot)}_{p(\cdot),q_1(\cdot)}(X,d,\mu)\hookrightarrow M^{s(\cdot)}_{p(\cdot),q_2(\cdot)}(X,d,\mu).
			\end{align*}
			\item[(v)] If $p,s\in \mathcal{P}_b(X)$, $q\in \mathcal{P}(X)$, then it holds that
			\begin{align*}
				\dot{M}^{s(\cdot)}_{p(\cdot),q(\cdot)}(X,d,\mu) \hookrightarrow \dot{M}^{s(\cdot),p(\cdot)}(X,d,\mu)
			\end{align*}
			where the operator semi-norm is less or equal to one.
			\item[(vi)] If $p,s\in \mathcal{P}_b(X)$, $q\in \mathcal{P}(X)$ are such that $p\geq q$, then
			\begin{equation*}
				\dot{N}^{s(\cdot)}_{p(\cdot),q(\cdot)}(X,d,\mu) \hookrightarrow \dot{M}^{s(\cdot),p(\cdot)}(X,d,\mu),
			\end{equation*}
			where the operator semi-norm is less or equal to one.
			\item[(vii)] If $p,s\in \mathcal{P}_b(X)$, $q\in \mathcal{P}(X)$, then for every $t \in \mathcal{P}_b(X)$ such that $s \gg t$ it holds that
			\begin{equation*}
				N^{s(\cdot)}_{p(\cdot),q(\cdot)}(X,d,\mu) \hookrightarrow M^{t(\cdot),p(\cdot)}(X,d,\mu).
			\end{equation*}
			\item[(viii)] If $p,s\in \mathcal{P}_b(X)$ and $q\in \mathcal{P}(X)$, then
			\begin{equation*}
				\dot{M}^{s(\cdot)}_{p(\cdot),q(\cdot)}(X,d,\mu) \hookrightarrow \dot{N}^{s(\cdot)}_{p(\cdot),\infty}(X,d,\mu).
			\end{equation*}
			\item[(ix)] If $p,s\in \mathcal{P}_b(X)$ and $\delta \in (0,\infty)$, then for every $t\in \mathcal{P}_b(X)$ such that $s \gg t$ there exists constant $\zeta(p,s,t,\delta)\in (0,\infty)$  such that for every ball $B:=B(x_0,r_0)\subseteq X$ with $r_0\in (0,\delta]$ and every $q\in \mathcal{P}(X)$, it holds that
			\begin{equation*}
				\dot{N}^{s(\cdot)}_{p(\cdot),q(\cdot)}(B,d,\mu) \hookrightarrow \dot{M}^{t(\cdot),p(\cdot)}(B,d,\mu)
			\end{equation*}
			and
			\begin{equation*}
				\left\| u \right\|_{\dot{M}^{t(\cdot),p(\cdot)}(B,d,\mu)} \leq \zeta(p,s,t,\delta)\left\|u \right\|_{\dot{N}^{s(\cdot)}_{p(\cdot),q(\cdot)}(B,d,\mu)},
			\end{equation*}
			for all $u\in \dot{N}^{s(\cdot)}_{p(\cdot),q(\cdot)}(B,d,\mu).$
		\end{enumerate}	
	\end{prop}
	
	Now we recall the Sobolev-Poincar\'{e} inequality from \cite{zanurzenia}, which will be essential in the proof of one of the main results.
	
	\begin{tw}\label{ciagle}
	Let $(X,d,\mu)$ be a metric measure space and assume that there exist $Q\in \mathcal{P}^{\log}_b(X)$, $\delta\in (0,\infty)$ and $b\in (0,\infty)$ such that for every $x\in X$ and $r \in (0,\delta]$ it holds
	\begin{equation*}
		\mu\left(B(x,r)\right)\geq br^{Q(x)}.
	\end{equation*}
	Fix $\sigma \in (1,\infty)$ and let $p,s\in \mathcal{P}^{\log}_b(X)$ satisfy $sp \ll Q$. Then, for every ball $B_0=B(x_0,r_0)\subseteq X$ with $r_0\leq \delta/\sigma$ and $u\in M^{s(\cdot),p(\cdot)}(\sigma B_0)$, $g\in \mathcal{D}^{s(\cdot)}(u)\cap L^{p(\cdot)}(\sigma B_0)$ there holds inequality
	\begin{equation}\label{localembedding}
		\left\|u\right\|_{L^{\gamma(\cdot)}(B_0)} \leq \left(1+\Lambda(B_0)\right)C_S \left(\frac{\mu(B_0)}{r_0^{Q(x_0)}}\right)^{\frac{1}{\gamma_{B_0}^-}}\left\|g \right\|_{L^{p(\cdot)}(\sigma B_0)}+\Lambda(B_0)\left\|u\right\|_{L^{p(\cdot)}(B_0)},
	\end{equation}
	where $\displaystyle \gamma=Qp/(Q-sp)$, $C_S$ is a constant depending only on $b$, $\delta$, $\sigma$, $p$, $s$, $Q$ and
	\begin{equation*}
		\Lambda(B_0):=\kappa_{\gamma(\cdot)}^2\max\left\{2,\left(\frac{2}{\mu(B_0)}\right)^{\frac{1}{\gamma_{B_0}^-}}\right\}\left\|1\right\|_{L^{\gamma(\cdot)}(B_0)}.
	\end{equation*}
	\end{tw}

	We will also need the following Sobolev embedding theorems from \cite{zanurzenia}.
	
	\begin{tw}\label{boundedembedding}
		Let $(X,d,\mu)$ be a bounded metric measure space and assume that there exist $Q\in \mathcal{P}_b^{\log}(X)$ and $b\in (0,1]$ such that for every $x\in X$ and $r\in (0,1]$,
		\begin{equation*}
			\mu\left(B(x,r)\right)\geq br^{Q(x)}.
		\end{equation*}
		Let $p,s\in \mathcal{P}_b^{\log}(X)$. If $sp \ll Q$, then, there exists a positive constant $\hat{C}_{S}$, depending only on $b$, $\diam X$, $\mu (X)$, $p$, $s$, $Q$, such that, for every pair of functions $u\in \dot{M}^{s(\cdot),p(\cdot)}(X,d,\mu)$ and $g\in \mathcal{D}^{s(\cdot)}(u)\cap L^{p(\cdot)}(X,\mu)$, there holds
			\begin{equation*}
				\inf_{c\in \mathbb R}\left\|u-c\right\|_{L^{\gamma(\cdot)}(X,\mu)} \leq \hat{C}_{S} \left\|g\right\|_{L^{p(\cdot)}(X,\mu)},
			\end{equation*}
			where $\displaystyle \gamma:=\frac{Qp}{Q-sp}$. Moreover, for every $u\in M^{s(\cdot),p(\cdot)}(X,d,\mu)$ and $g\in \mathcal{D}^{s(\cdot)}(u) \cap L^{p(\cdot)}(X,\mu)$
			\begin{equation*}
				\left\| u\right\|_{L^{\gamma(\cdot)}(X,\mu)} \leq \hat{C}_S\left( \left\|u \right\|_{L^{p(\cdot)}(X,\mu)} + \left\| g\right\|_{L^{p(\cdot)}(X,\mu)}\right).
			\end{equation*}
	\end{tw}

		\begin{tw}\label{gsob}
			Let $(X,d,\mu)$ be a geometrically doubling metric measure space such that $\sup\left\{\mu(B(x,1)):x\in X\right\} <\infty.$ Suppose that there exist $Q\in \mathcal{P}^{\log}_b(X)$ such that $\mu$ is lower Ahlfors $Q(\cdot)$-regular.
			Moreover, let $p,s\in \mathcal{P}^{\log}_b(X)$ be such that $sp \ll Q.$ Then,
			\begin{enumerate}
				\item[(i)] for every $q\in \mathcal{P}(X)$
				\begin{equation*}
					M^{s(\cdot)}_{p(\cdot),q(\cdot)}(X,d,\mu) \hookrightarrow L^{\gamma(\cdot)}(X,\mu),
				\end{equation*}
				where $\displaystyle \gamma:=Qp/(Q-sp)$,
				\item[(ii)] for every $q\in \mathcal{P}(X)$ such that $q \leq p$
				\begin{equation*}
					N^{s(\cdot)}_{p(\cdot),q(\cdot)}(X,d,\mu) \hookrightarrow L^{\gamma(\cdot)}(X,\mu),
				\end{equation*}
				where $\displaystyle \gamma:=Qp/(Q-sp)$,
				\item[(iii)] for every $q\in \mathcal{P}(X)$ and $t\in \mathcal{P}^{\log}_b(X)$ such that $t \ll s$
				\begin{equation*}
					N^{s(\cdot)}_{p(\cdot),q(\cdot)}(X,d,\mu) \hookrightarrow L^{\sigma(\cdot)}(X,\mu),
				\end{equation*}
				where $\displaystyle \sigma:=Qp/(Q-tp)$.
			\end{enumerate}
		\end{tw}
		
		\begin{tw}\label{ghold}
			Let $(X,d,\mu)$ be a metric measure space. Suppose that there exist $Q\in \mathcal{P}^{\log}_b(X)$ such that $\mu$ is lower Ahlfors $Q(\cdot)$-regular.
			Moreover, let $p,s\in \mathcal{P}^{\log}_b(X)$ be such that $sp \gg Q$. Then,
			\begin{enumerate}
				\item[(i)] for every $q\in \mathcal{P}(X)$
				\begin{equation*}
					M^{s(\cdot)}_{p(\cdot),q(\cdot)}(X,d,\mu) \hookrightarrow C^{0,\alpha(\cdot)}(X,d),
				\end{equation*}
				where $\displaystyle \alpha:=s-Q/p$,
				\item[(ii)] for every $q\in \mathcal{P}(X)$ such that $q \leq p$
				\begin{equation*}
					N^{s(\cdot)}_{p(\cdot),q(\cdot)}(X,d,\mu) \hookrightarrow C^{0,\alpha(\cdot)}(X,d),
				\end{equation*}
				where $\displaystyle \alpha:=s-Q/p$,
				\item[(iii)] for every $q\in \mathcal{P}(X)$ and $t\in \mathcal{P}^{\log}_b(X)$ such that $Q \ll tp \ll sp$
				\begin{equation*}
					N^{s(\cdot)}_{p(\cdot),q(\cdot)}(X,d,\mu) \hookrightarrow C^{0,\beta(\cdot)}(X,d),
				\end{equation*}
				where $\displaystyle \beta:=t-Q/p$.
			\end{enumerate}
		\end{tw}
		
	In some proofs, we also will use the following results concerning Lipschitz cut-off functions.
	
	\begin{tw}\label{lipschitz1}\cite{zanurzenia}
		Let $\left(X,d,\mu \right)$ be a metric measure space and let $B\subseteq X$ be non-empty bounded set. Suppose $u: X \to \mathbb [0,1]$ is a $L$-Lipschitz function which is zero outside of $B$, where $L\in(0,\infty)$. Fix $s, p\in \mathcal{P}_b(X)$, $q\in \mathcal{P}(X)$ and assume that $s^+ < 1$ and $q^- <\infty$ or $s^+ \leq 1$ and $q^-=\infty$. Then, the sequence $\left\{g_k\right\}_{k\in \mathbb Z}$ defined as follows
		\begin{equation*}
			g_k(x):=\left\{\begin{array}{lll} L2^{k\left(s(x)-1\right)}\chi_{B}(x),& \textnormal{ for } k\geq k_L,\\ 2^{\left(k+1\right)s(x)+1}\chi_{B}(x),&\textnormal{ for } k< k_L, \end{array} \right. 
		\end{equation*}
		where $k_L\in \mathbb Z$ is such that $2^{k_L-1}\leq L <2^{k_L}$, is a vector $s(\cdot)$-gradient of the function $u$. Moreover, there exist $C_{\textnormal{lip}} \in (0,\infty)$ depending only on $q^-$, $s^-$, $s^+$ such that
		\begin{align*}
			\left\| \left\{g_k\right\} \right\|_{L^{p(\cdot)}\left(\ell^{q(\cdot)} \left(X,\mu\right) \right)} &\leq C_{\textnormal{lip}} \max\left\{L^{s_B^-},L^{s_B^+}\right\}\left\|\chi_{B}\right\|_{L^{p(\cdot)}\left(X,\mu\right)},\\ \left\| \left\{g_k\right\}\right\|_{\ell^{q(\cdot)}\left(L^{p(\cdot)}\left(X,\mu\right)\right)} &\leq C_{\textnormal{lip}} \max\left\{L^{s_B^-},L^{s_B^+}\right\}\left\|\chi_{B}\right\|_{L^{p(\cdot)}\left(X,\mu\right)}
		\end{align*}
		and in particular $u$ belongs to $\dot{M}^{s(\cdot)}_{p(\cdot),q(\cdot)}(X,d,\mu)$, $\dot{N}^{s(\cdot)}_{p(\cdot),q(\cdot)}(X,d,\mu)$.
	\end{tw}

		\begin{lem}\label{cuttoff}
		Let us consider the function $u: \mathbb R^n \to \mathbb R$ defined as
		\begin{equation*}
			u(x):= \left\{\begin{array}{ll} 1, & \textnormal{ for } x\in B(0,1),\\ 2-\left|x\right|, & \textnormal{ for } x\in B(0,2)\setminus B(0,1),\\ 0, & \textnormal{ for } x\in \mathbb R^n \setminus B(0,2).\end{array}\right.
		\end{equation*}
		Suppose that $s\in \mathcal{P}_b(\mathbb R^n)$ is such that $s^+\leq 1$. For fixed $\tau \in (0,\infty)$ and $\sigma \in \mathbb R$ we define the function $u_{\tau,\sigma}:\mathbb R^n \to \mathbb R$ as $u_{\tau,\sigma}(x):=\tau^{\sigma} u(\tau x)$. Then, the function $g_{\tau,\sigma}(x):=2\tau^{\sigma+s(x)}\chi_{B(0,2/\tau)}(x)$ is a scalar $s(\cdot)$-gradient of $u_{\tau, \sigma}$.
	\end{lem}
	
	\begin{proof}
		It suffices to prove that for every $x,y\in \mathbb R^n$ we have
		\begin{equation}\label{gradient}
			\left|u_{\tau,\sigma}(x)-u_{\tau,\sigma}(y)\right| \leq \left|x-y\right|^{s(x)}g_{\tau,\sigma}(x)+\left|x-y\right|^{s(y)}g_{\tau,\sigma}(y).
		\end{equation}
		If $x,y\notin B(0,2/\tau)$, then there is nothing to prove. Therefore without loss of generality we can assume that $x\in B(0,2/\tau)$. If $\left|x-y\right| \leq 1/\tau$, then using the fact that $u$ is $1$-Lipschitz, we get
		\begin{align*}
			\left|u_{\tau,\sigma}(x)-u_{\tau,\sigma}(y)\right| &\leq \tau^{\sigma}\left|\tau x -\tau y\right|\leq \tau^{\sigma}\left|\tau x -\tau y\right|^{s(x)} = \left|x-y\right|^{s(x)}\tau^{\sigma+s(x)}\chi_{B(0,2/\tau)}(x) \\ & \leq  \left|x-y\right|^{s(x)}g_{\tau,\sigma}(x)+\left|x-y\right|^{s(y)}g_{\tau, \sigma}(y).
		\end{align*}
		On the other hand, if $\left|x-y\right| > 1/\tau$, then
		\begin{align*}
			\left|u_{\tau,\sigma}(x)-u_{\tau,\sigma}(y)\right| &\leq \left|u_{\tau,\sigma}(x)\right|+\left|u_{\tau,\sigma}(y)\right|\leq 2\tau^{\sigma}< 2\tau^{\sigma}\tau^{s(x)}\left|x-y\right|^{s(x)} \\&\leq \left|x-y\right|^{s(x)}g_{\tau,\sigma}(x)+\left|x-y\right|^{s(y)}g_{\tau,\sigma}(y)
		\end{align*}
		and hence \eqref{gradient} is proven. 
	\end{proof}
	
	\subsection{H\"older space with variable exponent}
	In this subsection, we recall the definition of variable exponent H\"older space. Let $(X,d)$ be a metric space. By $C(X,d)$ we denote the space of continuous functions $u:X \to \mathbb{R}$ such that the norm
	\begin{equation*}
		\left\|u\right\|_{C(X,d)}= \sup_{x\in X}\left|u(x)\right|
	\end{equation*}
	is finite. Moreover, for bounded function $\alpha : X \to [0,\infty)$ we denote by $C^{0,\alpha(\cdot)}(X,d)$ the \texttt{H\"older space}, i.e. the space of all functions $u\in C(X,d)$ such that
	\begin{equation*}
		\left[ u \right]_{\alpha(\cdot),X}:= \sup_{\substack{x,y\in X \\ x\neq y}} \frac{\left|u(x)-u(y)\right|}{d(x,y)^{\alpha(x)}} < \infty.
	\end{equation*}
	The H\"older space $C^{0,\alpha(\cdot)}(X,d)$ is a Banach space due to the norm
	\begin{equation*}
		\left\| u  \right\|_{C^{0,\alpha(\cdot)}(X,d)} := \left\| u \right\|_{C(X,d)}  + \left[u\right]_{\alpha(\cdot),X}, \textnormal{ where } u \in C^{0,\alpha(\cdot)}(X,d).
	\end{equation*}
	In the main section of the paper, we shall use the following result from \cite{holdery}.
	\begin{tw}\label{holdcomp}
	Let $(X,d)$ be a totally bounded metric space and $\alpha,\beta \in \mathcal{P}_b(X)$. If $\alpha \gg \beta$, then we have compact embedding
	\begin{equation*}
		C^{0,\alpha(\cdot)}(X,d) \hookrightarrow \hookrightarrow C^{0,\beta(\cdot)}(X,d).
	\end{equation*}
	\end{tw}
	
	\section{Main Results}		\subsection{Compact embeddings on metric spaces with finite or integrable measure}
	In this part, we investigate compact embeddings of $M^{s(\cdot)}_{p(\cdot),q(\cdot)}$ and $N^{s(\cdot)}_{p(\cdot),q(\cdot)}$ into Lebesgue and H\"older spaces on metric spaces equipped with finite or integrable measure. We shall use methods from the paper \cite{zwartestale} adapted to the variable exponent setting.
	\subsubsection{Embeddings into the space of measurable functions}
	We begin with the following theorem concerning compact embeddings into $L^0$.
	\begin{tw}\label{mierzalne}
		Let $(X,d,\mu)$ be a metric measure space such that $\mu(X)<\infty.$ Then, for every $s,p \in \mathcal{P}_b(X)$ and $q\in \mathcal{P}(X)$
		\begin{equation*}
			M^{s(\cdot)}_{p(\cdot),q(\cdot)}(X,d,\mu) \hookrightarrow \hookrightarrow L^0(X,\mu) \hspace{15mm} N^{s(\cdot)}_{p(\cdot),q(\cdot)}(X,d,\mu) \hookrightarrow \hookrightarrow L^0(X,\mu).
		\end{equation*}
	\end{tw}
	
	\begin{proof}
		Since by Proposition \ref{pomiedzy} $(v)$ and $(vii)$
		\begin{equation*}
	M^{s(\cdot)}_{p(\cdot),q(\cdot)}(X,d,\mu) \hookrightarrow M^{s(\cdot),p(\cdot)}(X,d,\mu), \hspace{10mm} N^{s(\cdot)}_{p(\cdot),q(\cdot)}(X,d,\mu) \hookrightarrow M^{s(\cdot)/2,p(\cdot)}(X,d,\mu), 
		\end{equation*}
		it suffices to prove that for every $s,p\in \mathcal{P}_b(X)$ the embedding of $M^{s(\cdot),p(\cdot)}(X,d,\mu)$ into $L^0(X,\mu)$ is compact. Let $\mathcal{F}\subseteq M^{s(\cdot),p(\cdot)}(X,d,\mu)$ be a bounded family. Without loss of generality we can assume that
		\begin{equation}\label{zalozenie}
			0<\sup_{u \in \mathcal{F}} \left\|u \right\|_{M^{s(\cdot),p(\cdot)}(X,d,\mu)}<1.
		\end{equation}
		Our goal is to show that $\mathcal{F}$ is totally bounded in $L^0(X,\mu)$. We will apply Theorem \ref{krotov}\footnote{We assume that measure of every ball is positive and finite, which ensures $X$ is separable, allowing the application of Theorem \ref{krotov}.}. Let us fix $\varepsilon \in (0,1)$ and define
		\begin{equation*}
			\lambda:=\left(\frac{4}{\varepsilon}\right)^{\frac{1}{p^-}}, \hspace{5mm} \delta:=\left(\frac{\varepsilon^{1+\frac{1}{p^-}}}{4^{1+\frac{1}{p^-}}}\right)^{\frac{1}{s^-}}.
		\end{equation*}
		Then, $\lambda\in (1,\infty)$ and $\delta\in (0,1)$.
		For each $u\in \mathcal{F}$ we select $g\in \mathcal{D}^{s(\cdot)}(u) \cap L^{p(\cdot)}(X,\mu)$ such that $\left\|g\right\|_{L^{p(\cdot)}(X,\mu)} \leq 1$. Then, there exists measure zero set $G_u \subseteq X$ such that for every $x,y\in X \setminus G_u$
		\begin{equation*}
			\left|u(x)-u(y)\right| \leq d(x,y)^{s(x)}g(x)+d(x,y)^{s(y)}g(y).
		\end{equation*}
		We consider the following set
		\begin{equation*}
			E(u)=\left\{x\in X: \left|u(x)\right|>\lambda \textnormal{ or } g(x)>\lambda \textnormal{ or } x\in G_u\right\}.
		\end{equation*}
		By Chebyshev's inequality, the fact that $\lambda\in (1,\infty)$, assumption \eqref{zalozenie} and Lemma \ref{modularnorma} we obtain
		\begin{align*}
			\mu\left(E(u)\right)& \leq \mu\left(\left\{x \in X: \left|u(x)\right| > \lambda\right\}\right)+\mu\left(\left\{x\in X: g(x) >\lambda\right\}\right)+\mu\left(G_u\right) \\ &\leq \frac{1}{\lambda^{p^-}}\left(\int_X \left|u(x)\right|^{p(x)}\mbox{d}\mu(x)+\int_X g(x)^{p(x)}\mbox{d}\mu(x)\right)\leq \frac{\varepsilon}{4}\cdot 2 <\varepsilon.
		\end{align*}
		Now, let $x,y\in X \setminus E(u)$ be such that $d(x,y)<\delta$. Thus,
		\begin{align*}
			\left|u(x)-u(y)\right|&\leq d(x,y)^{s(x)}g(x)+d(x,y)^{s(y)}g(y)\leq \delta^{s(x)}\lambda+\delta^{s(y)}\lambda \leq 2\lambda \delta^{s^-} = 2\left(\frac{4}{\varepsilon}\right)^{\frac{1}{p^-}}\frac{\varepsilon^{1+\frac{1}{p^-}}}{4^{1+\frac{1}{p^-}}}=\frac{\varepsilon}{2}<\varepsilon.
		\end{align*}
		Applying Theorem \ref{krotov} we conclude that $\mathcal{F}$ is totally bounded in $L^0(X,\mu)$ and the proof is complete.
	\end{proof}
	
	\subsubsection{Embeddings into variable exponent Lebesgue spaces} In this part, we address sufficient and necessary conditions for the compact embeddings of Haj{\l}asz-Triebel-Lizorkin and Haj{\l}asz-Besov spaces into Lebesgue spaces. Below is the main result of this subsection.
	\begin{tw}\label{zwarteskonczone}
		Suppose that $(X,d,\mu)$ is a metric measure space and $p,s\in \mathcal{P}_b(X)$, $q\in \mathcal{P}(X)$. Assume that at least one of the following conditions hold:\begin{enumerate}
			\item[(i)] $(X,d)$ is totally bounded;
			\item[(ii)] measure $\mu$ is integrable and $p\in \mathcal{P}^{\log}_{\mu}(X)$. 
		\end{enumerate}
		Then, we have compact embeddings
		\begin{equation*}
			M^{s(\cdot)}_{p(\cdot),q(\cdot)}(X,d,\mu) \hookrightarrow \hookrightarrow L^{p(\cdot)}(X,\mu), \hspace{15mm} 	N^{s(\cdot)}_{p(\cdot),q(\cdot)}(X,d,\mu) \hookrightarrow \hookrightarrow L^{p(\cdot)}(X,\mu).
		\end{equation*}
	\end{tw}
	
	\begin{proof} We begin with proving proposition that will be crucial in the proof of Theorem \ref{zwarteskonczone}.
		
		\begin{prop}\label{uniintegrable}
			Let $(X,d,\mu)$ be a metric measure space and $p,s\in \mathcal{P}_b(X).$ Assume that at least one of the following conditions is satisfied:
			\begin{enumerate}
				\item[(i)] $(X,d)$ is totally bounded;
				\item[(ii)] measure $\mu$ is integrable and $p\in \mathcal{P}^{\log}_{\mu}(X)$.
			\end{enumerate}Then, every bounded family $\mathcal{F} \subseteq M^{s(\cdot),p(\cdot)}(X,d,\mu)$ is $p(\cdot)$-equi-integrable. Moreover, for every $x_0 \in X$ we have
			\begin{equation}\label{zanik}
				\lim_{R\to \infty} \sup_{u\in \mathcal{F}} \int_{X \setminus B(x_0,R)} \left|u(x)\right|^{p(x)}\mbox{d}\mu(x)=0.
			\end{equation}
		\end{prop}
		\begin{proof}
			Let $\mathcal{F} \subseteq M^{s(\cdot),p(\cdot)}(X,d,\mu)$ be any bounded family. Without loss of generality assume that
			\begin{equation*}
				\sup_{u\in \mathcal{F}} \left\|u\right\|_{M^{s(\cdot),p(\cdot)}(X)}= \frac{1}{2}.
			\end{equation*}
			Fix arbitrary $u\in \mathcal{F}$ and $q\in (0,p^-)$. By the very definition of quasi-norm $\left\|\cdot \right\|_{M^{s(\cdot),p(\cdot)}(X,d,\mu)}$ we can find $g\in \mathcal{D}^{s(\cdot)}(u) \cap L^{p(\cdot)}(X,\mu)$ such that
			\begin{equation}\label{tralala}
				\left\|u\right\|_{L^{p(\cdot)}(X,\mu)}+\left\|g\right\|_{L^{p(\cdot)}(X,\mu)}\leq 1.
			\end{equation}
			By the definition of Haj{\l}asz $s(\cdot)$-gradient we can find the measure zero set $G_u \subseteq X$ such that for every $x,y\in X\setminus G_u$
			\begin{equation*}
				\left|u(x)-u(y)\right| \leq d(x,y)^{s(x)}g(x)+d(x,y)^{s(y)}g(y).
			\end{equation*} Let $x\in X\setminus G_u$ and $r\in (0,1)$, then for all $y\in B(x,r)\setminus G_u$ we have
			\begin{align*}
				\left|u(x)\right|^q&\leq 2^q\left(\left|u(x)-u(y)\right|^q+\left|u(y)\right|^{q}\right)\leq 2^{2q}\left(d(x,y)^{qs(x)}g(x)^q+d(x,y)^{qs(y)}g(y)^q+\left|u(y)\right|^q\right) \\ & \leq 2^{2q}\left[r^{qs^-}\left(g(x)^q+g(y)^q\right)+\left|u(y)\right|^q\right].
			\end{align*}
			Averaging the above inequality with respect to $y\in B(x,r)$ we obtain
			\begin{equation*}
				\left|u(x)\right|^q \leq 4^{q}\left[ r^{qs^-}\left(g(x)^q+\left(g^q\right)_{B(x,r)}\right)+\left(\left|u\right|^q\right)_{B(x,r)}\right],
			\end{equation*}
			where
			\begin{align*}
				\left(g^q\right)_{B(x,r)}:= \fint_{B(x,r)} g(y)^q \mbox{d}\mu(y),\hspace{15mm}
				\left(\left|u\right|^q\right)_{B(x,r)}:= \fint_{B(x,r)} \left|u(y)\right|^q \mbox{d}\mu(y).
			\end{align*}
			Subsequently, we raise the above inequality to the $p(\cdot)/q$ power and integrate over fixed measurable set $E\subseteq X$ to get
			\begin{align}\label{nier1}
				\int_E \left|u(x)\right|^{p(x)}\mbox{d}\mu \leq 4^{p^+ +\frac{p^+}{q}}\left[r^{p^-s^-}\left(\int_E g(x)^{p(x)}\mbox{d}\mu(x)+\int_E\left(g^q\right)_{B(x,r)}^{\frac{p(x)}{q}}\mbox{d}\mu(x)\right)+\int_E \left(\left|u\right|^q\right)_{B(x,r)}^{\frac{p(x)}{q}}\mbox{d}\mu(x)\right].
			\end{align}
			We shall firstly estimate second integral on the right-hand side of the above inequality. Assume that $(i)$ holds. By Lemma \ref{charakteryzacja} we conclude that $\mu(X)<\infty$ and for every $r\in (0,\infty)$ we have $h(r):=\inf\left\{\mu(B(x,r)): x\in X\right\}>0.$ Then, using \eqref{tralala}, Proposition \ref{holder1}, Lemma \ref{modularnorma}, and Lemma \ref{skalowanie} we obtain
			\begin{align*}
				\begin{split}
					\left(g^q\right)_{B(x,r)}^{\frac{p(x)}{q}} & \leq \left( \frac{2}{\mu(B(x,r))}\left\| g^q \right\|_{L^{\frac{p(\cdot)}{q}}(B(x,r))}\left\|1\right\|_{L^{\frac{p(\cdot)}{p(\cdot)-q}}(B(x,r))}\right)^{\frac{p(x)}{q}
					} \\ & \leq 2^{\frac{p(x)}{q}}\left\|g\right\|^{p(x)}_{L^{p(\cdot)}(B(x,r))} \max\left\{\mu(B(x,r))^{-\frac{p(x)}{p^+}},\mu(B(x,r))^{-\frac{p(x)}{p^-}}\right\} \\ & \leq 2^{\frac{p^+}{q}}\max\left\{h(r)^{-\frac{p(x)}{p^+}},h(r)^{-\frac{p(x)}{p^-}}\right\}\leq 2^{\frac{p^+}{q}}\max\left\{h(r)^{-1},h(r)^{-\frac{p^-}{p^+}},h(r)^{-\frac{p^+}{p^-}}\right\}.
				\end{split}
			\end{align*}
			Hence
			\begin{equation}\label{nier2}
				\int_E	\left(g^q\right)_{B(x,r)}^{\frac{p(x)}{q}}\mbox{d}\mu(x) \leq  2^{\frac{p^+}{q}}\max\left\{h(r)^{-1},h(r)^{-\frac{p^-}{p^+}},h(r)^{-\frac{p^+}{p^-}}\right\}\mu(E). 
			\end{equation}
			Similarily, we obtain
			\begin{equation}\label{nier3}
				\int_E \left(\left|u\right|^q\right)_{B(x,r)}^{\frac{p(x)}{q}}\mbox{d}\mu(x) \leq  2^{\frac{p^+}{q}}\max\left\{h(r)^{-1},h(r)^{-\frac{p^-}{p^+}},h(r)^{-\frac{p^+}{p^-}}\right\}\mu(E).
			\end{equation}
			In concert, \eqref{nier1}, \eqref{nier2} and \eqref{nier3} give
			\begin{equation}\label{ost}
				\int_E \left|u(x)\right|^{p(x)}\mbox{d}\mu(x)\leq 4^{p^+ +\frac{p^+}{q}} \left[r^{p^- s^-} + 2^{\frac{p^+}{q}}\max\left\{h(r)^{-1},h(r)^{-\frac{p^-}{p^+}},h(r)^{-\frac{p^+}{p^-}}\right\}\mu(E)\left(1+r^{p^- s^-}\right)\right].
			\end{equation}
			Now, from \eqref{ost} it immediately follows that $\mathcal{F}$ is $p(\cdot		)$-equi-integrable. In addition, since measure $\mu$ is finite, we also have proved \eqref{zanik}. 
			
			Suppose that $(ii)$ holds. Since $p \in \mathcal{P}_{\mu}^{\log}(X)$, there exists $p_{\infty}\in (0,\infty]$ such that $1\in L^{t(\cdot)}(X,\mu)$, where
			\begin{equation*}
				\frac{1}{t(x)}:=\left|\frac{1}{p(x)}-\frac{1}{p_{\infty}}\right|.
			\end{equation*} 
			Hence, we can find $\lambda \in (0,1)$ such that $\rho_{t(\cdot)}(\lambda)<\infty.$ Then, by the Jensen inequality (Proposition \ref{jensen}) applied for $\gamma:=\lambda^{2q}$ and respectively $f:=g^q$, $f:=\left|u\right|^q$ there exists $\beta \in (0,1)$ depending only on $\lambda$ and $C_{\log,\mu}(p/q)$ such that
			\begin{align}\label{oszac}
			\begin{split}	&\left(\beta \left(g^q\right)_{B(x,r)}\right)^{\frac{p(x)}{q}} \leq \fint_{B(x,r)} g(y)^{p(y)} \mbox{d}\mu(y) +\fint_{B(x,r)} \lambda^{2qw(x,y)}\mbox{d}\mu(y),\\&\left(\beta \left(\left|u\right|^q\right)_{B(x,r)}\right)^{\frac{p(x)}{q}} \leq \fint_{B(x,r)} \left|u(y)\right|^{p(y)} \mbox{d}\mu(y) +\fint_{B(x,r)} \lambda^{2qw(x,y)}\mbox{d}\mu(y),
				\end{split}
			\end{align}
			where
			\begin{equation*}
				\frac{1}{w(x,y)}:=\max\left\{ \frac{q}{p(x)}-\frac{q}{p(y)},0\right\}.
			\end{equation*}
			From the definition of $w$ and $t$ it follows that 
			\begin{equation*}
				0 \leq \frac{1}{w(x,y)} \leq \frac{q}{t(x)}+\frac{q}{t(y)},
			\end{equation*}
			for all $x,y\in X$. This implies that $\displaystyle qw(x,y) \geq \min \left\{t(x)/2,t(y)/2\right\}.$ Thus, since $\lambda \in (0,1)$ we obtain
			\begin{equation}\label{ineqnum}
				\lambda^{2qw(x,y)} \leq \lambda^{t(x)}+\lambda^{t(y)}.
			\end{equation}
			Using inequality \eqref{ineqnum} in \eqref{oszac} we get
			\begin{align*}
			&\left(\beta \left(g^q\right)_{B(x,r)}\right)^{\frac{p(x)}{q}} \leq \frac{1}{\mu(B(x,r))}\left(1+\rho_{t(\cdot)}(\lambda)\right)+\lambda^{t(x)},\\&\left(\beta \left(\left|u\right|^q\right)_{B(x,r)}\right)^{\frac{p(x)}{q}} \leq \frac{1}{\mu(B(x,r))}\left(1+\rho_{t(\cdot)}(\lambda)\right)+\lambda^{t(x)}.
			\end{align*}
			Hence, by inequality \eqref{nier1} we get
			\begin{align*}
				\int_E \left|u(x)\right|^{p(x)}\mbox{d}\mu(x)  \leq 2\cdot 4^{p^+\left(1+\frac{1}{q}\right)}\left[ r^{p^- s^-} +  \frac{1+\rho_{t(\cdot)}(\lambda)}{\beta^{\frac{p^+}{q}}}\left(\int_E \frac{1}{\mu(B(x,r))} \mbox{d}\mu(x)+ \int_E \lambda^{t(x)}\mbox{d}\mu(x)\right)\right].
			\end{align*}
		Since $\mu$ is integrable and $\rho_{t(\cdot)}(\lambda)<\infty$, we conclude that the family $\mathcal{F}$ is $p(\cdot)$-equi-integrable and moreover \eqref{zanik} holds, which completes the proof.
		\end{proof}
		
		Now we are ready to prove Theorem \ref{zwarteskonczone}. Since by Proposition \ref{pomiedzy} $(v)$ and $(vii)$ we have continuous embeddings
		\begin{equation*}
			M^{s(\cdot)}_{p(\cdot),q(\cdot)}(X,d,\mu) \hookrightarrow M^{s(\cdot),p(\cdot)}(X,d,\mu), \hspace{10mm} N^{s(\cdot)}_{p(\cdot),q(\cdot)}(X,d,\mu) \hookrightarrow M^{s(\cdot)/2,p(\cdot)}(X,d,\mu),
		\end{equation*}
		it suffices to prove that embedding of $M^{s(\cdot),p(\cdot)}(X,d,\mu)$ into $L^{p(\cdot)}(X,\mu)$ is compact for every $p,s\in \mathcal{P}_b(X)$.
	 Let $\mathcal{F} \subseteq M^{s(\cdot),p(\cdot)}(X,d,\mu)$ be any bounded set in $M^{s(\cdot),p(\cdot)}(X,d,\mu)$. We shall prove that $\mathcal{F}$ is totally bounded in $L^{p(\cdot)}(X,\mu)$. Let $\varepsilon \in (0,1)$ be fixed. It suffices to find $\varepsilon$-net for $\mathcal{F}$ in $L^{p(\cdot)}(X,\mu)$. By Proposition \ref{uniintegrable} we obtain that we can find $x_0\in X$ and $R>0$ such that 
	 \begin{equation*}
	 	\sup_{u\in \mathcal{F}} \int_{X\setminus B(x_0,R)} \left|u(x)\right|^{p(x)}\mbox{d}\mu(x)< \frac{\varepsilon^{p^+}}{2}.
	 \end{equation*}
	 Now, let us consider measure $\nu:= \mu \measurerestr B(x_0,R)$. Then, obviously $\nu$ is finite and $\nu \leq \mu$. Since $M^{s(\cdot),p(\cdot)}(X,d,\mu) \hookrightarrow M^{s(\cdot),p(\cdot)}(X,d,\nu)$, by Theorem \ref{mierzalne} we know that the family $\mathcal{F}$ is totally bounded in $L^0(X,\nu)$. Moreover, since $\mathcal{F}$ is $p(\cdot)$-equi-integrable with respect to $\mu$, we conclude that it is also $p(\cdot)$-equi-integrable with respect to $\nu$. This allows us to apply Theorem \ref{hanson} and deduce that $\mathcal{F}$ is totally bounded in $L^{p(\cdot)}(X,\nu)$. Hence, there exist $n\in \mathbb N$ and the set $\left\{f_i\right\}_{i=1}^n\subseteq L^{p(\cdot)}(X,\nu)$ such that for every $u\in \mathcal{F}$ we can find $j\in \left\{1,\dots,n\right\}$ such that
	 \begin{equation}\label{nier}
	 	\int_X \left|u(x)-f_j(x)\right|^{p(x)}\mbox{d}\nu(x) < \frac{\varepsilon^{p^+}}{2}.
	 \end{equation}
	 For $i\in\left\{1,\dots,n\right\}$ let $g_i:=f_i \chi_{B(x_0,R)}$. We shall show that $\left\{g_i\right\}_{i=1}^n$ is desired $\varepsilon$-net in $L^{p(\cdot)}(X,\mu)$. It is straightforward that $\left\{g_i\right\}_{i=1}^n \subseteq L^{p(\cdot)}(X,\mu)$. Let us fix $u\in \mathcal{F}$ and take $j\in \left\{1,\dots,n\right\}$ such that \eqref{nier} is satisfied. Then
	 \begin{align*}
	 	\int_X \left|u(x)-g_j(x)\right|^{p(x)}\mbox{d}\mu(x) & = \int_{B(x_0,R)} \left|u(x)-f_j(x)\right|^{p(x)}\mbox{d}\mu(x)+\int_{X \setminus B(x_0,R)} \left|u(x)\right|^{p(x)}\mbox{d}\mu(x) \\ & =\int_X \left|u(x)-f_j(x)\right|^{p(x)}\mbox{d}\nu(x)+\int_{X \setminus B(x_0,R)} \left|u(x)\right|^{p(x)}\mbox{d}\mu(x) \\ & < \frac{\varepsilon^{p^+}}{2}+\frac{\varepsilon^{p^+}}{2}=\varepsilon^{p^+}.
	 \end{align*}
	 Finally, from Proposition \ref{modularnorma}
	 \begin{align*}
	 	\left\| u-g_j\right\|_{L^{p(\cdot)}(X,\mu)} & \leq \max\left\{\rho_{p(\cdot)}\left(u-g_j\right)^{\frac{1}{p^-}}, \rho_{p(\cdot)}\left(u-g_j\right)^{\frac{1}{p^+}}\right\}\leq \varepsilon.
	 \end{align*}
	 Therefore, we have found $\varepsilon$-net for $\mathcal{F}$ in $L^{p(\cdot)}(X,\mu)$, which proves that $\mathcal{F}$ is totally bounded in $L^{p(\cdot)}(X,\mu)$ and the proof is done.
	\end{proof}	
	
	\begin{ex}
	Let $((0,\infty),d,\mu_{\beta})$, where $d=\left|\cdot- \cdot\right|$ is the Euclidean distance and $\mu_{\beta}$ is defined as $\mbox{d}\mu_{\beta}=e^{\varepsilon x^{\beta}}\mbox{d}x$ for any fixed $\beta \in (1,\infty)$ and $\varepsilon \in\left\{-1,1\right\}$. Then, for every exponents $p,s\in \mathcal{P}_b((0,\infty))$ and $q\in \mathcal{P}((0,\infty))$ such that $p\in \mathcal{P}_{\mu_{\beta}}^{\log}((0,\infty))$ we have the compact embeddings
	\begin{equation*}
		M^{s(\cdot)}_{p(\cdot),q(\cdot)}((0,\infty),d,\mu_{\beta}) \hookrightarrow \hookrightarrow L^{p(\cdot)}((0,\infty),\mu_{\beta}), \hspace{10mm} N^{s(\cdot)}_{p(\cdot),q(\cdot)}((0,\infty),d,\mu_{\beta}) \hookrightarrow \hookrightarrow L^{p(\cdot)}((0,\infty),\mu_{\beta}).
	\end{equation*}
	\end{ex}
	
	\begin{proof}
	From Example \ref{gaussowska} we know that $\mu_{\beta}$ is integrable. Therefore, the claim follows from Theorem \ref{zwarteskonczone}.
	\end{proof}
	
	For a converse, we shall prove that in certain class of metric measure spaces totally boundedness of $(X,d)$ is necessary for compact embeddings.
	
	\begin{tw}\label{konieczny1}
	Let $(X,d,\mu)$ be a metric measure space such that $\sup\left\{\mu(B(x,1)):x\in X\right\}<\infty$ and $h(r):=\inf\left\{\mu(B(x,r)): x\in X\right\}>0$ for all $r\in (0,1)$. Suppose that there exist $s,p,\beta\in \mathcal{P}_b(X)$, $q\in \mathcal{P}(X)$ satisfying conditions $s^+ <1$ and $q^-<\infty$ or $s^+ \leq 1$ and $q^-=\infty$ such that at least one of the following embeddings is compact
	\begin{equation*}
		M^{s(\cdot)}_{p(\cdot),q(\cdot)}(X,d,\mu) \hookrightarrow L^{\beta(\cdot)}(X,\mu) \hspace{10mm} 
		N^{s(\cdot)}_{p(\cdot),q(\cdot)}(X,d,\mu) \hookrightarrow L^{\beta(\cdot)}(X,\mu).
	\end{equation*}
	Then, the space $(X,d)$ is totally bounded.
	\end{tw}
	
	\begin{proof}
	Let $A\in \left\{M,N\right\}$ and let us assume that embedding of $A^{s(\cdot)}_{p(\cdot),q(\cdot)}(X,d,\mu)$ into $L^{\beta(\cdot)}(X,\mu)$ is compact. It is enough to show that for every $\delta\in (0,1/2)$ every $4\delta$-separated subset of $X$ is finite. Let us fix $\delta \in (0,1/2)$ and let $S \subseteq X$ be any $4\delta$-separated subset. Since $X$ is separable, we have $S=\left\{x_k\right\}_{k=1}^{\infty}$ for some $x_k \in X$, where $k\in \mathbb N$. For each $k\in \mathbb N$ we define the function $\phi_k : X \to [0,1]$ as
	\begin{equation*}
		\phi_k(x):= \left\{\begin{array}{ll} 1,& \textnormal{ for } x\in B(x_k,\delta),\\ \displaystyle 2-\frac{d(x,x_k)}{\delta},& \textnormal{ for } x\in B(x_k,2\delta) \setminus B(x_k,\delta),\\ 0,& \textnormal{ for } x\in X \setminus B(x_k,2\delta).\end{array}  \right.
	\end{equation*}
	Then, $\phi_k$ is $1/\delta$-Lipschitz and supported in $B(x_k,2\delta)$. Hence, by Theorem \ref{lipschitz1} and Lemma \ref{modularnorma} we get
	\begin{align*}
		\left\|\phi_k \right\|_{A^{s(\cdot)}_{p(\cdot),q(\cdot)}(X,d,\mu)}& = \left\|\phi_k\right\|_{L^{p(\cdot)}(X,\mu)} + \left\| \phi_k \right\|_{\dot{A}^{s(\cdot)}_{p(\cdot),q(\cdot)}(X,d,\mu)} \\ & \leq \left\|\chi_{B(x_k,2\delta)}\right\|_{L^{p(\cdot)}(X,\mu)}+C_{\textnormal{lip}}\max\left\{\delta^{-s_{B(x_k,2\delta)}^+}, \delta^{-s_{B(x_k,2\delta)}^-}\right\}\left\|\chi_{B(x_k,2\delta)}\right\|_{L^{p(\cdot)}(X,\mu)} \\ & \leq \max\left\{\mu(B(x_k,2\delta))^{\frac{1}{p^+}},\mu(B(x_k,2\delta))^{\frac{1}{p^-}}\right\}\left(1+C_{\textnormal{lip}} \max\left\{\delta^{-s^+},\delta^{-s^-}\right\}\right) \\ & \leq (M+1)^{\frac{1}{p^-}}\left(1+C_{\textnormal{lip}} \max\left\{\delta^{-s^+},\delta^{-s^-}\right\}\right),
	\end{align*}
	where $M:= \sup\left\{\mu(B(x,1)): x\in X\right\}<\infty$ and $C_{\textnormal{lip}}\in (0,\infty)$ is a constant from Theorem \ref{lipschitz1}. Therefore, the sequence $\left\{\phi_k\right\}_{k=1}^{\infty}$ is bounded in $A^{s(\cdot)}_{p(\cdot),q(\cdot)}(X,d,\mu)$ and our assumption implies that there exists subsequence $\left\{\phi_{k_m}\right\}_{m=1}^{\infty}$ which converges to some $\phi$ in $L^{\beta(\cdot)}(X,\mu)$. However, since $S$ is $4\delta$-separated, we have $B(x_k,2\delta)\cap B(x_l,2\delta)=\emptyset$ for all $k,l\in \mathbb N$ such that $k\neq l$. Hence for $k,l\in \mathbb N$, $k\neq l$ we have
	\begin{align*}
		\rho_{\beta(\cdot)}\left(\phi_k-\phi_l\right)&\geq \int_{B(x_k,2\delta) \cup B(x_l,2\delta)} \left|\phi_k(x)-\phi_l(x)\right|^{\beta(x)}\mbox{d}\mu(x)\\ &=\int_{B(x_k,2\delta)} \left|\phi_k(x)\right|^{\beta(x)}\mbox{d}\mu(x) + \int_{B(x_l,2\delta)} \left|\phi_l(x)\right|^{\beta(x)}\mbox{d}\mu(x)\\&\geq  \mu(B(x_k,\delta))+\mu(B(x_l,\delta))\geq 2h(\delta)>0,
	\end{align*}
	which implies that $\left\{\phi_k\right\}_{k=1}^{\infty}$ can not have convergent subsequence in $L^{\beta(\cdot)}(X,\mu)$. This contradiction implies that every $4\delta$-separated subset of $X$ is finite and the proof is complete.
	\end{proof}
	
	Now, we immediately obtain the following characterization of compact embeddings.
	
	\begin{tw}
	Let $(X,d,\mu)$ be a metric measure space such that $\sup\left\{\mu(B(x,1)): x\in X\right\}<\infty$ and $h(r):=\inf\left\{\mu(B(x,r)): x\in X\right\}>0$ for all $r\in (0,1)$. Then, the following statements are equivalent:
	\begin{enumerate}
		\item[(i)] $\mu(X)<\infty$;
		\item[(ii)] for every $p,s\in \mathcal{P}_b(X)$ and $q\in \mathcal{P}(X)$ we have the following compact embeddings
		\begin{equation*}
			M^{s(\cdot)}_{p(\cdot),q(\cdot)}(X,d,\mu) \hookrightarrow \hookrightarrow L^{p(\cdot)}(X,\mu), \hspace{10mm} 
			N^{s(\cdot)}_{p(\cdot),q(\cdot)}(X,d,\mu) \hookrightarrow \hookrightarrow L^{p(\cdot)}(X,\mu);
		\end{equation*}
		\item[(iii)] there exist $p,s,\beta\in \mathcal{P}_b(X)$ and $q\in \mathcal{P}(X)$ satisfying conditions $s^+ <1$ and $q^- < \infty$ or $s^+ \leq 1$ and $q^- =\infty$ such that at least one of the following compact embeddings hold
			\begin{equation*}
			M^{s(\cdot)}_{p(\cdot),q(\cdot)}(X,d,\mu) \hookrightarrow \hookrightarrow L^{\beta(\cdot)}(X,\mu), \hspace{10mm} 
			N^{s(\cdot)}_{p(\cdot),q(\cdot)}(X,d,\mu) \hookrightarrow \hookrightarrow L^{\beta(\cdot)}(X,\mu).
		\end{equation*}		
	\end{enumerate}
	\end{tw}
	
	\begin{proof}
	Implication $(ii) \implies (iii)$ is straightforward. If $(i)$ holds, then by Lemma \ref{charakteryzacja} and the fact that function $h$ is increasing we obtain that $(X,d)$ is totally bounded. Hence by Theorem \ref{zwarteskonczone} we get $(ii)$. On the other hand, if $(iii)$ holds, then Theorem \ref{konieczny1} yields that $(X,d)$ is totally bounded. In particular $\mu(X)<\infty$ and we got $(i)$.
	\end{proof}
	
	\subsubsection{Rellich-Kondrachov type compactness theorem}
	Now we shall apply the results from the previous subsections to prove the Rellich-Kondrachov type theorem (firstly for Haj{\l}asz-Sobolev spaces).
	\begin{tw}\label{rellich}
		Let $(X,d,\mu)$ be a totally bounded metric measure space such that $\mu$ is lower Ahlfors $Q(\cdot)$-regular for some $Q\in \mathcal{P}_b^{\log}(X)$. Fix $p,s\in \mathcal{P}_b^{\log}(X)$ such that $sp \ll Q$. Then, we have compact embedding
		\begin{equation*}
			M^{s(\cdot),p(\cdot)}(X,d,\mu) \hookrightarrow \hookrightarrow L^{q(\cdot)}(X,\mu)
		\end{equation*}
		for every $q\in \mathcal{P}_b(X)$ such that $p\leq q \ll \gamma$, where $\displaystyle \gamma:=Qp/(Q-sp)$.
	\end{tw}
	
	\begin{proof}
		At first, assume that $p\ll q\ll \gamma$. Let $\left\{u_n\right\}_{n=1}^{\infty}\subseteq M^{s(\cdot),p(\cdot)}(X,d,\mu)$ be a bounded sequence. By the virtue of Theorem \ref{boundedembedding} $(i)$ and Proposition \ref{pomiedzy} $(ii)$ we conclude that there exists $C>0$ such that
		\begin{equation}\label{ogran}
			\sup_{n\in \mathbb N} \left\|u_n\right\|_{L^{\gamma(\cdot)}(X,\mu)}\leq C.
		\end{equation}
		On the other hand, Theorem \ref{zwarteskonczone} yields that $\left\{u_n\right\}_{n=1}^{\infty}$ has subsequence $\left\{u_{n_k}\right\}_{k=1}^{\infty}$ which converges to some $u$ in $L^{p(\cdot)}(X,\mu)$. Moreover, we may assume that $u_{n_k} \to u$ almost everywhere. From \eqref{ogran} we know that $\left\|u\right\|_{L^{\gamma(\cdot)}(X,\mu)} \leq C$. Let $v_k:=\left(u_{n_k}-u\right)\left(2C\kappa_{\gamma(\cdot)}\right)^{-1}$, where $\kappa_{\gamma(\cdot)}$ is a constant from triangle inequality for the quasi-norm $\left\|\cdot\right\|_{L^{\gamma(\cdot)}}$. Then $v_k \to 0$ in $L^{p(\cdot)}(X,\mu)$ and $\left\|v_k\right\|_{L^{\gamma(\cdot)}(X,\mu)} \leq 1$ for every $k\in \mathbb N$.
		
		Let
		\begin{align*}
			q_1=\frac{p\left(\gamma-q\right)}{\gamma-p}, \hspace{4mm} q_2=\frac{\gamma\left(q-p\right)}{\gamma-p},\hspace{4mm}
			w=\frac{\gamma-p}{\gamma-q}, \hspace{4mm} w'=\frac{\gamma-p}{q-p}.
		\end{align*}
		By H\"older inequality (Proposition \ref{holder1}) applied for the exponents $w$ and $w'$ we get that
		\begin{equation*}
			\int_X \left|v_k(x)\right|^{q(x)}\mbox{d}\mu(x) = \int_X \left|v_k(x)\right|^{q_1(x)}\left|v_k(x)\right|^{q_2(x)}\mbox{d}\mu(x) \leq 2\left\| \left|v_k\right|^{q_1} \right\|_{L^{w(\cdot)}(X,\mu)} \left\| \left|v_k\right|^{q_2} \right\|_{L^{w'(\cdot)}(X,\mu)}
		\end{equation*}
		and in particular $v_k \in L^{q(\cdot)}(X,\mu)$ for each $k\in \mathbb N$.	Let us fix $k\in \mathbb N$. Since $\left\|v_k\right\|_{L^{\gamma(\cdot)}(X,\mu)} \leq 1$, by Lemma \ref{modularnorma} we get that
		\begin{equation*}
			\rho_{w'(\cdot)}\left(\left|v_k\right|^{q_2}\right)=\rho_{\gamma(\cdot)}\left(v_k\right)\leq 1.
		\end{equation*}
		Hence, again using Lemma \ref{modularnorma} we get $\left\| \left|v_k\right|^{q_2} \right\|_{L^{w'(\cdot)}(X,\mu)} \leq 1$ and therefore by Lemma \ref{modularnorma}
		\begin{align*}
			\int_X \left|v_k(x)\right|^{q(x)}\mbox{d}\mu(x)&\leq 2\left\| \left|v_k\right|^{q_1} \right\|_{L^{w(\cdot)}(X,\mu)}\leq 2\max\left\{ \rho_{w(\cdot)}\left(\left|v_k\right|^{q_1}\right)^{\frac{1}{w^-}},  \rho_{w(\cdot)}\left(\left|v_k\right|^{q_1}\right)^{\frac{1}{w^+}}\right\} \\ & =2\max\left\{ \rho_{p(\cdot)}(v_k)^{\frac{1}{w^-}},\rho_{p(\cdot)}(v_k)^{\frac{1}{w^+}}\right\} \stackrel{k\to \infty}{\longrightarrow} 0,
		\end{align*} 
		which implies that $v_k \to 0$ in $L^{q(\cdot)}(X,\mu)$ and the proof is complete in the case $p \ll q \ll \gamma$.
		
		Assume now that $p \leq q \ll \gamma$ and fix $\varepsilon \in (0,\infty)$ such that $t:=q+\varepsilon \ll \gamma$. Then, $p \ll t \ll \gamma$ and by the first part of the proof we have
		\begin{equation*}
			M^{s(\cdot),p(\cdot)}(X,d,\mu) \hookrightarrow \hookrightarrow L^{t(\cdot)}(X,\mu).
		\end{equation*}
		Since by Lemma \ref{wlozenielp} there holds the embedding
		\begin{equation*}
			L^{t(\cdot)}(X,\mu) \hookrightarrow L^{q(\cdot)}(X,\mu)
		\end{equation*}
		we obtain that
		\begin{equation*}
			M^{s(\cdot),p(\cdot)}(X,d,\mu) \hookrightarrow \hookrightarrow L^{q(\cdot)}(X,\mu),
		\end{equation*}
		which completes the proof.
	\end{proof}
	
	The following example shows that we can not expect that the Sobolev embedding into the Lebesgue space with the exponent $q=\gamma$ is compact, although other assumptions of Theorem \ref{rellich} are satisfied. 
	
	\begin{ex}\label{przyklad1}
		Let $n\in \mathbb N$. Consider the metric measure space $\left(B, d, \lambda\right)$, where $B=B(0,3)\subseteq \mathbb R^n$, $d:=\left|\cdot-\cdot\right|$ denotes the Euclidean distance and $\lambda$ is the Lebesgue measure. Suppose that $s,p\in \mathcal{P}_b^{\log}( B)$ are such that $sp \ll n$ and $s^+ \leq 1$. Then, the embedding  \begin{equation*}
			M^{s(\cdot),p(\cdot)}( B,d,\lambda) \hookrightarrow L^{\gamma(\cdot)}(B,\lambda)
		\end{equation*}
		is not compact, where $\displaystyle \gamma:=np/(n-sp)$.
	\end{ex}
	\begin{proof}
		Let us define the function $u: \mathbb R^n \to \mathbb R$ as follows
		\begin{equation*}
			u(x):= \left\{\begin{array}{ll} 1, & \textnormal{ for } x\in B(0,1),\\ 2-\left|x\right|, & \textnormal{ for } x\in B(0,2)\setminus B(0,1),\\ 0, & \textnormal{ for } x\in \mathbb R^n \setminus B(0,2).\end{array}\right.
		\end{equation*}
		Now, for fixed $k\geq 2$ let we define $u_k:=k^{\frac{n}{\gamma(0)}}u(kx)$. By the virtue of Lemma \ref{cuttoff} applied for $\tau:=k$ and $\sigma=n/\gamma(0)$ we obtain that the function $g_k(x):=2k^{\frac{n}{\gamma(0)}+s(x)}\chi_{B(0,2/k)}(x)$ is a scalar $s(\cdot)$-gradient of $u_k$. We shall prove that the sequence $\left\{u_{k}\right\}_{k=2}^{\infty}$ is bounded in $M^{s(\cdot),p(\cdot)}(B,d,\lambda)$. By the log-H\"older continuity of $p$ for all $x\in B(0,2/k)$ we have
		\begin{align*}
			p(x)&\leq p(0)+\frac{C_{\log}(p)}{\log(e+1/\left|x\right|)} \leq p(0)+\frac{C_{\log}(p)}{\log(e+k/2)}.
		\end{align*}
		Moreover, the function 
		\begin{equation*}
			(0,1)\ni t \mapsto \left(\frac{1}{t}\right)^{\left[\log\left(e+(2t)^{-1}\right)\right]^{-1}}
		\end{equation*}
		is bounded from above by some constant $M\in (0,\infty)$.
		Hence, since $\gamma(0) > p(0)$, we estimate
		\begin{align*}
			\rho_{p(\cdot)}\left(u_{k}\right) &\leq \int_{B(0,2/k)} k^{\frac{np(x)}{\gamma(0)}}\mbox{d}x \leq \left(\frac{2}{k}\right)^n \omega_n k^{\frac{n}{\gamma(0)}\left( p(0)+C_{\log}(p)\left[\log\left(e+k/2\right)\right]^{-1}\right)} \\ & \leq 2^n \omega_n k^{n\left(\frac{p(0)}{\gamma(0)}-1\right)}M^{nC_{\log}(p)/\gamma(0)}\leq 2^n \omega_n M^{nC_{\log}(p)/\gamma(0)},
		\end{align*}
		where $\omega_n$ denotes the $n$-dimensional Lebesgue measure of unit ball in $\mathbb R^n$. Moreover, by the definition of $\gamma$ and the log-H\"older continuity of $p$ and $sp$ we obtain
		\begin{align*}
			\rho_{p(\cdot)}\left(g_k\right)& \leq 2^{p^+}\int_{B(0,2/k)}k^{\frac{np(x)}{\gamma(0)}+s(x)p(x)}\mbox{d}x \\ & \leq 2^{p^+} \left(\frac{2}{k}\right)^n \omega_n k^{\frac{n}{\gamma(0)}\left[p(0)+C_{\log}(p)\left(\log\left(e+k/2\right)\right)^{-1}\right]}\cdot k^{s(0)p(0)+C_{\log}(sp)\left[\log\left(e+k/2\right)\right]^{-1}}  \\ & \leq 2^{n+p^+} \omega_n k^{\frac{np(0)}{\gamma(0)}+s(0)p(0)-n} \cdot k^{\left(\frac{nC_{\log}(p)}{\gamma(0)}+C_{\log}(sp)\right)\left[\log\left(e+k/2\right)\right]^{-1}} \\ & \leq 2^{n+p^+} \omega_n M^{\frac{nC_{\log}(p)}{\gamma(0)}+C_{\log}(sp)}.
		\end{align*}
		Hence, the sequence $\left\{u_{k}\right\}_{k=2}^{\infty}$ is bounded in $M^{s(\cdot),p(\cdot)}(B,d,\lambda)$.
		On the other hand,
		\begin{align*}
			\rho_{\gamma(\cdot)}\left(u_{k}\right) & \geq \int_{B(0,1/k)} k^{\frac{n\gamma(x)}{\gamma(0)}}\mbox{d}x.
		\end{align*}
		Let $x\in B(0,1/k)$. Using the log-H\"older continuity of $\gamma$ we get
		\begin{equation*}
			1\leq k^{\left|\gamma(x)-\gamma(0)\right|}\leq\left(e+k\right)^{\left|\gamma(x)-\gamma(0)\right|}\leq e^{C_{\log}(\gamma)}.
		\end{equation*}
		Thus,
		\begin{equation}\label{ostatnie}
			\rho_{\gamma(\cdot)}\left(u_{k}\right) \geq \int_{B(0,1/k)} k^{n}e^{-\frac{n}{\gamma(0)}C_{\log}(\gamma)}\mbox{d}x=\omega_n e^{-\frac{n}{\gamma(0)}C_{\log}(\gamma)}>0.
		\end{equation}
		Notice that $u_{k} \stackrel{k \to \infty}{\longrightarrow} 0$ almost everywhere in $\mathbb R^n$. Therefore, if there existed subsequence $\left\{u_{k_m}\right\}_{m=1}^{\infty}$ convergent in $L^{\gamma(\cdot)}(B,\lambda)$, it would converge to $0$. This is an obvious contradiction with \eqref{ostatnie} and the embedding of $M^{s(\cdot),p(\cdot)}(B,d,\lambda)$ into $L^{\gamma(\cdot)}(B,\lambda)$ cannot be compact.
	\end{proof}

	Using Theorem \ref{rellich} we obtain the following version of the Rellich-Kondrachov theorem for Besov and Triebel-Lizorkin spaces.
	\begin{tw}
		Suppose that $(X,d,\mu)$ is a totally bounded metric measure space and $\mu$ is lower Ahlfors $Q(\cdot)$-regular for some $Q\in \mathcal{P}_b^{\log}(X)$. Fix $q\in \mathcal{P}(X)$ and $p,s\in \mathcal{P}_b^{\log}(X)$ such that $sp \ll Q$.	Then, the following statements are true.
		\begin{enumerate}
			\item[(i)] For every $\beta\in \mathcal{P}_b(X)$ we have compact embedding
			\begin{equation*}
				M^{s(\cdot)}_{p(\cdot),q(\cdot)}(X,d,\mu) \hookrightarrow \hookrightarrow L^{\beta(\cdot)}(X,\mu)
			\end{equation*}
			if $p\leq \beta \ll \gamma$, where $\displaystyle \gamma:=Qp/(Q-sp)$.
			\item[(ii)] For every $\beta\in \mathcal{P}_b(X)$ such that $q\leq p$ we have compact embedding
			\begin{equation*}
				N^{s(\cdot)}_{p(\cdot),q(\cdot)}(X,d,\mu) \hookrightarrow \hookrightarrow L^{\beta(\cdot)}(X,\mu)
			\end{equation*}
			if $p\leq \beta \ll \gamma$, where $\displaystyle \gamma:=Qp/(Q-sp)$.
			\item[(iii)] For every $\beta\in \mathcal{P}_b(X)$ and $t\in \mathcal{P}_b^{\log}(X)$ such that $t\ll s$ we have compact embedding
			\begin{equation*}
				N^{s(\cdot)}_{p(\cdot),q(\cdot)}(X,d,\mu) \hookrightarrow \hookrightarrow L^{\beta(\cdot)}(X,\mu),
			\end{equation*}
			if $p\leq \beta \ll \sigma$, where $\displaystyle \sigma:=Qp/(Q-tp)$.
		\end{enumerate}
	\end{tw}
	
	\begin{proof} 
		The statement $(i)$ follows immediately from Proposition \ref{pomiedzy} $(v)$ and Theorem \ref{rellich}. Then, $(ii)$ follows from Proposition \ref{pomiedzy} $(vi)$ and Theorem \ref{rellich}. Finally, the statement $(iii)$ follows from Proposition \ref{pomiedzy} $(vii)$ and Theorem \ref{rellich}.
	\end{proof} 
	
	We also obtain similar theorem in the case of embedding into H\"older spaces.
	
	\begin{tw}
		Let $(X,d,\mu)$ be a totally bounded metric measure space such that $\mu$ is lower Ahlfors $Q(\cdot)$-regular for some $Q\in \mathcal{P}_b^{\log}(X)$. Moreover, let $s,p\in \mathcal{P}_b^{\log}(X)$ be such that $sp \gg Q$. Then, the following statements are true.
		\begin{enumerate}
			\item[(i)] For every $q\in \mathcal{P}(X)$ we have compact embedding
			\begin{equation*}
				M^{s(\cdot)}_{p(\cdot),q(\cdot)}(X,d,\mu) \hookrightarrow \hookrightarrow C^{0,\lambda(\cdot)}(X,d),
			\end{equation*}
			for every $\lambda \in \mathcal{P}_b(X)$ such that $\lambda \ll \alpha$, where $\alpha=s-Q/p$.
			\item[(ii)] For every $q\in \mathcal{P}(X)$ such that $q\leq p$ we have compact embedding
			\begin{equation*}
				N^{s(\cdot)}_{p(\cdot),q(\cdot)}(X,d,\mu) \hookrightarrow \hookrightarrow C^{0,\lambda(\cdot)}(X,d),
			\end{equation*}
			for every $\lambda \in \mathcal{P}_b(X)$ such that $\lambda \ll \alpha$, where $\alpha=s-Q/p$.
			\item[(iii)] For every $q\in \mathcal{P}(X)$ and $t\in \mathcal{P}_b^{\log}(X)$ such that $Q\ll tp \ll sp$ we have compact embedding
			\begin{equation*}
				M^{s(\cdot)}_{p(\cdot),q(\cdot)}(X,d,\mu) \hookrightarrow \hookrightarrow C^{0,\lambda(\cdot)}(X,d),
			\end{equation*}
			for every $\lambda \in \mathcal{P}_b(X)$ such that $\lambda \ll \beta$, where $\beta=t-Q/p$.
		\end{enumerate}
	\end{tw}
	
	\begin{proof}
		All the statements follow immediately from Theorems \ref{ghold} and \ref{holdcomp}.
	\end{proof}
	
		\subsection{Compact embeddings of group invariant Besov and Triebel-Lizorkin spaces}
	In this subsection we shall investigate the influence of isometry group action on the compact embeddings of invariant Haj{\l}asz-Besov and Haj{\l}asz-Triebel-Lizorkin spaces. In particular, we answer the open question posed in \cite{Górka}. Next, we deal with the necessary condition for compact embeddings of group invariant spaces. At the end of subsection we present a partial characterization of compact embeddings. 
	
	\subsubsection{Berestycki-Lions type theorem} In this subsection, we investigate sufficient conditions for compact embeddings of group invariant Sobolev, Besov and Triebel-Lizorkin spaces. Firstly, we shall start with proving the Berestycki-Lions type theorem  for fractional Haj{\l}asz-Sobolev spaces. The idea of its proof is inspired on the methods used in \cite{Gaczkowski, Górka}.

	\begin{tw}\label{zwartenieogr}
	
	Suppose that $(X,d,\mu)$ is a geometrically doubling metric measure space where $\mu$ satisfies the condition  $\sup\left\{\mu(B(x,1)): x \in X \right\}<\infty$. Moreover, assume that $\mu$ is lower Ahlfors $Q(\cdot)$-regular for some $Q\in \mathcal{P}_b^{\log}(X)$. Suppose that $H\subseteq \textnormal{Iso}_{\mu}(X,d)$ is such that there exist $x_0\in X$ and $r\in (0,\infty)$ satisfying
		\begin{equation*}
			\lim_{R\to \infty} \inf_{x\in X \setminus B(x_0,R)} M_H(x,r)=\infty.
		\end{equation*}
		Then, for all $p\in \mathcal{P}_b^{\log}(X)\cap \mathcal{P}_H(X)$, $s\in \mathcal{P}_b^{\log}(X)$ such that $s^+ \leq 1$ and $sp \ll Q$ we have the compact embedding
		\begin{equation*}
			M^{s(\cdot),p(\cdot)}_H(X,d,\mu) \hookrightarrow \hookrightarrow L^{q(\cdot)}(X,\mu)
		\end{equation*}
		for all $q\in \mathcal{P}_b(X)$ satisfying $p \ll q \ll \gamma$, where $\displaystyle \gamma=Qp/(Q-sp)$.
	\end{tw}
	
	\begin{remark}
		It is worth to notice that the above result is new even for the constant exponents, since we do not need assumption on reflexivity of Haj{\l}asz-Sobolev space, the condition $p^->1$ nor upper Ahlfors regularity of measure as in \cite{Górka}. Moreover, the set $H$ does not necessarily have to be subgroup of $\textnormal{Iso}_{\mu}(X,d)$.
	\end{remark}
	
	 Before we give the proof of Theorem \ref{zwartenieogr}, we need some definitions and lemmas.
	
	\begin{defi}
		Let $(X,d)$ be a metric space. Fix $x_0 \in X$ and $R\in (0,\infty)$. We define the following map
		\begin{equation*}
			f_{x_0,R}(x):=\left\{ \begin{array}{lll} \frac{1}{R}\left(2R-d(x,x_0)\right),& \textnormal{ for } x\in B(x_0,2R)\setminus B(x_0,R),\\ 1, & \textnormal{ for } x\in B(x_0,R),\\ 0, & \textnormal{ for } x\in X \setminus B(x_0,2R).\end{array} \right.
		\end{equation*}
		Now, we define the operator $F_{x_0,R}$ by the formula $F_{x_0,R}(u):=uf_{x_0,R}$ for measurable function $u$.
	\end{defi}
	
	\begin{lem}\label{operator}
		Let $(X,d,\mu)$ be a metric measure space such that $\displaystyle h(r):=\inf_{x\in X}\mu(B(x,r))>0$ for all $r\in (0,1)$. Then, for every $s,p\in \mathcal{P}_b(X)$ with $s^+\leq 1$, every $x_0\in X$ and $R\in (0,\infty)$ the operator 
		\begin{equation*}
			F_{x_0,R}: M^{s(\cdot),p(\cdot)}(X,d,\mu) \rightarrow L^{p(\cdot)}(X,\mu)
		\end{equation*} is compact.
	\end{lem}
	
	\begin{proof}
		Let $\mathcal{F} \subseteq M^{s(\cdot),p(\cdot)}(X,d,\mu)$ be a bounded set. We begin with proving that $F_{x_0,R}(\mathcal{F})$ is bounded in $M^{s(\cdot),p(\cdot)}(X,d,\mu)$. Fix $u \in \mathcal{F}$ and $g\in \mathcal{D}^{s(\cdot)}(u) \cap L^{p(\cdot)}(X,\mu)$. Define
		\begin{equation*}
			\tilde{g}(x):= g(x)\chi_{B(x_0,2R)}(x)+2\left|u(x)\right|R^{-s(x)}\chi_{B(x_0,2R)}(x).
		\end{equation*}
		Obviously, we see that $\tilde{g} \in L^{p(\cdot)}(X,\mu)$. We shall check that $\tilde{g} \in \mathcal{D}^{s(\cdot)}(F_{x_0,R}(u))$. Let $G_u\subseteq X$ be the measure zero set such that
		\begin{align*}
			\left|u(x)-u(y)\right|&\leq d(x,y)^{s(x)}g(x)+d(x,y)^{s(y)}g(y)
		\end{align*}
		for all $x,y\in X \setminus G_u$.
		We prove that
		\begin{equation}\label{hajlasz}
			\left|u(x)f_{x_0,R}(x)-u(y)f_{x_0,R}(y)\right| \leq d(x,y)^{s(x)}\tilde{g}(x)+d(x,y)^{s(y)}\tilde{g}(y)
		\end{equation}
		for every $x,y\in X \setminus G_u$. If $x,y\notin B(x_0,2R)$, then \eqref{hajlasz} is obvious. Therefore, we assume that at least one of $x,y$ belongs to $B(x_0,2R)$. Due to the symmetry of \eqref{hajlasz} we can assume that $y\in B(x_0,2R)$. If $d(x,y)\leq R$, then
		\begin{align*}
 \left|f_{x_0,R}(x)-f_{x_0,R}(y)\right| & \leq  \frac{1}{R}d(x,y)= \frac{1}{R}d(x,y)^{s(y)}d(x,y)^{1-s(y)} \leq \frac{1}{R}d(x,y)^{s(y)} R^{1-s(y)}\\ & \leq \frac{1}{R^{s(y)}}d(x,y)^{s(y)}.
		\end{align*}
		On the other hand, if $d(x,y)>R$, then
		\begin{align*}
		\left|f_{x_0,R}(x)-f_{x_0,R}(y)\right| &\leq 2=2d(x,y)^{s(y)}d(x,y)^{-s(y)}\leq \frac{2}{R^{s(y)}}d(x,y)^{s(y)}.
		\end{align*}Now, if also $x\in B(x_0,2R)$ then we have
		\begin{align*}
			\left|u(x)f_{x_0,R}(x)-u(y)f_{x_0,R}(y)\right|& \leq \underbrace{\left|f_{x_0,R}(x)\right|}_{\leq 1} \left|u(x)-u(y)\right| +\left|u(y)\right| \left|f_{x_0,R}(x)-f_{x_0,R}(y)\right| \\ & \leq  d(x,y)^{s(x)}g(x)+d(x,y)^{s(y)}g(y)+\left|u(y)\right|\frac{2}{R^{s(y)}}d(x,y)^{s(y)}\\ & \leq d(x,y)^{s(x)}\tilde{g}(x)+d(x,y)^{s(y)}\tilde{g}(y).
		\end{align*}
		If $x\notin B(x_0,2R)$, then
		\begin{align*}
			\left|u(x)f_{x_0,R}(x)-u(y)f_{x_0,R}(y)\right|& \leq \underbrace{\left|f_{x_0,R}(x)\right|}_{=0} \left|u(x)-u(y)\right| +\left|u(y)\right| \left|f_{x_0,R}(x)-f_{x_0,R}(y)\right| \\ & \leq  \left|u(y)\right|\frac{2}{R^{s(y)}}d(x,y)^{s(y)}\\ & \leq d(x,y)^{s(x)}\tilde{g}(x)+d(x,y)^{s(y)}\tilde{g}(y).
		\end{align*}Hence $F_{x_0,R}(u) \in M^{s(\cdot),p(\cdot)}(X,d,\mu)$. Moreover, it is straightforward that
		\begin{align*}
			\left\|F_{x_0,R}(u)\right\|_{M^{s(\cdot),p(\cdot)}(X,d,\mu)} \leq \left(1+\kappa_{p(\cdot)}\max\left\{1+\frac{2}{R^{s^+}},1+\frac{2}{R^{s^-}}\right\}\right)\left\|u\right\|_{M^{s(\cdot),p(\cdot)}(X,d,\mu)},
		\end{align*}
		where $\kappa_{p(\cdot)}$ is a constant from triangle inequality for quasi-norm $\left\|\cdot \right\|_{L^{p(\cdot)}}$. Therefore, the family $F_{x_0,R}(\mathcal{F})$ is bounded in the Sobolev space $M^{s(\cdot),p(\cdot)}(X,d,\mu)$, as wanted. 	
		
		By Proposition \ref{uniintegrable} we obtain that $F_{x_0,R}(\mathcal{F})$ is $p(\cdot)$-equi-integrable. Let us fix $\varepsilon\in (0,\infty)$. By Theorem \ref{mierzalne} we get
		\begin{equation*}
			M^{s(\cdot),p(\cdot)}(B(x_0,2R)) \hookrightarrow \hookrightarrow L^0(B(x_0,2R)).
		\end{equation*}
		Thus, by Theorem \ref{hanson} we deduce that $F_{x_0,R}(\mathcal{F})$ it totally bounded in $L^{p(\cdot)}(B(x_0,2R))$. Hence, there exist $n\in \mathbb N$ and the set $\left\{w_i\right\}_{i=1}^n\subseteq L^{p(\cdot)}(B(x_0,2R))$ such that for fixed $u\in \mathcal{F}$ there exists $j\in \left\{1,\dots,n\right\}$ such that we have
		\begin{equation*}
			\int_{B(x_0,2R)} \left|w_j(x)-F_{x_0,R}(u)(x)\right|^{p(x)}\mbox{d}\mu(x) < \varepsilon.
		\end{equation*}
		Let 
		\begin{equation*}
			v_i(x):=\left\{\begin{array}{ll} w_i(x), & \textnormal{ for } x\in B(x_0,2R),\\ 0, & \textnormal{ for } x\in X \setminus B(x_0,2R).\end{array}\right.
		\end{equation*}
		We have that $\left\{v_i\right\}_{i=1}^n\subseteq L^{p(\cdot)}(X,\mu)$ and
		\begin{equation*}
			\int_{X} \left|v_j(x)-F_{x_0,R}(u)(x)\right|^{p(x)}\mbox{d}\mu(x)=\int_{B(x_0,2R)}\left|w_j(x)-F_{x_0,R}(u)(x)\right|^{p(x)}\mbox{d}\mu(x) <\varepsilon.
		\end{equation*}
		Therefore, $F_{x_0,R}(\mathcal{F})$ is totally bounded in $L^{p(\cdot)}(X,\mu)$ and the proof is done.
	\end{proof}
	
	\begin{prop}\label{cauchy}
		Suppose that $(X,d,\mu)$ is a geometrically doubling metric measure space where $\mu$ satisfies the condition  $\sup\left\{\mu(B(x,1)): x \in X \right\}<\infty$. Moreover, assume that $\mu$ is lower Ahlfors $Q(\cdot)$-regular for some $Q\in \mathcal{P}^{\log}_b(X)$.
		Fix $p,s\in \mathcal{P}_b^{\textnormal{log}}(X)$ such that $sp \ll Q$. Suppose that  $\left\{u_n\right\}_{n=1}^{\infty}$ is a bounded sequence in $M^{s(\cdot),p(\cdot)}(X,d,\mu)$ such that there exists $r\in (0,1/2)$ satisfying
		\begin{equation*}
			\scalebox{1.1}{$\forall_{\varepsilon \in (0,\infty)}\hspace{1mm}\exists_{N\in \mathbb R}\hspace{1mm} \forall_{m,n>N}$\hspace{1mm}} \sup_{y\in X} \int_{B(y,r)}\left|u_m(x)-u_n(x)\right|^{p(x)} \mbox{d}\mu(x)<\varepsilon.
		\end{equation*}
		Then, $u_n \to u$ in $L^{q(\cdot)}(X,\mu)$ for some $u\in L^{q(\cdot)}(X,\mu)$, where $p \ll q \ll \gamma$ and $\displaystyle \gamma=Qp/(Q-sp)$.
	\end{prop}
	
	\begin{proof}
		Firstly, note that assumption that $p \ll q$ assures that the quantity $\xi:=(q-p)^-$ is strictly positive. Moreover, since $\gamma \gg q$, we can choose $\eta \in (0,\infty)$ such that $\eta \leq \xi^2/2p^+$ and $q+\eta \ll \gamma$. In addition, we introduce auxiliary exponent $t:=q+\eta$. Uniform continuity of $p$ and relation $p+\xi/2 \ll q$ imply that we can find some $\tilde{r}\in (0,r]$ such that for every $y\in X$ we have $q_{B(y,\tilde{r})}^-\geq p_{B(y,2\tilde{r})}^+ +\xi/2$. Since both $\tilde{r}$ and $r$ can be taken sufficiently small, we can assume that $r=\tilde{r}$.
		
		Now, define exponent
		\begin{equation*}
			w=\frac{t-p}{t-q}.
		\end{equation*}
		Once can easily convince oneself that
		\begin{equation*}
			w'=\frac{t-p}{q-p}
		\end{equation*}
		and for $x\in X$
		\begin{equation*}
			w'(x)=1+\frac{t(x)-q(x)}{q(x)-p(x)}=1+\frac{\eta}{q(x)-p(x)}\leq 1+\frac{\eta}{\xi}\leq 1+\frac{\xi}{2p^+}\leq 1+\frac{\xi}{2p_{B(y,2r)}^+}.
		\end{equation*}
		Hence, for all $y\in X$
		\begin{equation}\label{trololo}
			\frac{q_{B(y,r)}^-}{\left(w'\right)^+p_{B(y,2r)}^+}\geq \frac{p_{B(y,2r)}^++\xi/2}{\left(1+\xi/2p_{B(y,2r)}^+\right)p_{B(y,2r)}^+}=1.
		\end{equation} 
		Now, for $n,m\in \mathbb N$ we shall denote $v_{n,m}:=u_n-u_m$. Applying Theorem \ref{ciagle} for $\sigma:=2$ and $B_0:=B(y,r)$ we get  
		\begin{equation*}
			\left\|v_{n,m}\right\|_{L^{\gamma(\cdot)}(B(y,r))} \leq D(y,r)\left\|g_{n,m} \right\|_{L^{p(\cdot)}(B(y,2r))}+ \Lambda(B(y,r))\left\| v_{n,m} \right\|_{L^{p(\cdot)}(B(y,r))},
		\end{equation*}
		where
		\begin{equation*}
			D(y,r):=\left(1+\Lambda(B(y,r))\right)C_S \left(\frac{\mu(B(y,r))}{r^{Q(y)}}\right)^{\frac{1}{\gamma_{B(y,r)}^-}}
		\end{equation*}
		and $g_{n,m}\in \mathcal{D}^{s(\cdot)}(v_{n,m})\cap L^{p(\cdot)}(B(y,2r))$. Here, $\Lambda(B(y,r))$ is a constant from Theorem \ref{ciagle}, i.e.
		\begin{equation*}
			\Lambda(B(y,r))=\kappa_{\gamma(\cdot)}^2 \max\left\{2, \left(\frac{2}{\mu(B(y,r))}\right)^{\frac{1}{\gamma_{B(y,r)}^-}}\right\} \left\|1\right\|_{L^{\gamma(\cdot)}(B(y,r))},
		\end{equation*}
		where $\kappa_{\gamma(\cdot)}$ is a constant from triangle inequality for $L^{\gamma(\cdot)}$ quasi-norm. Since $\mu$ is lower Ahlfors $Q(\cdot)$-regular, there exists $b\in (0,1)$ such that $\mu(B(y,r))\geq br^{Q(y)}$. Then, using the assumption that $M:=\displaystyle \sup\left\{\mu(B(x,1)): x\in X \right\}<\infty$ we obtain
		\begin{align*}
			\Lambda(B(y,r))&\leq \kappa_{\gamma(\cdot)}^2 \max\left\{2, \left(\frac{2}{br^{Q(y)}}\right)^{\frac{1}{\gamma_{B(y,r)}^-}}\right\}\max\left\{\mu(B(y,r))^{\frac{1}{\gamma_{B(y,r)}^-}}, \mu(B(y,r))^{\frac{1}{\gamma_{B(y,r)}^+}}\right\} \\ & \leq \kappa_{\gamma(\cdot)}^2 \left(\frac{2}{br^{Q^+}}\right)^{\frac{1}{\gamma^-}} (M+1)^{\frac{1}{\gamma^-}}:=\Lambda(r)<\infty.
		\end{align*}
		Therefore
		\begin{align*}
			D(y,r)& \leq \left(1+\Lambda(r)\right) C_S \left(\frac{M+1}{r^{Q^+}}\right)^{\frac{1}{\gamma^-}}:=D(r)
		\end{align*}
		and hence for every $y\in X$, $n,m\in \mathbb N$ it holds
		\begin{equation*}
			\left\|v_{n,m}\right\|_{L^{\gamma(\cdot)}(B(y,r))}\leq E(r)\left\|v_{n,m}\right\|_{M^{s(\cdot),p(\cdot)}(B(y,2r))},
		\end{equation*}
		where $E(r):=D(r)+\Lambda(r)<\infty$. Moreover, since $t\ll \gamma$, Proposition \ref{jednakowo} yields that there exists $C=C(t,\gamma,M)\in (0,\infty)$ such that for every $y\in X$ we have
		\begin{equation}\label{ogr}
			\left\|v_{n,m}\right\|_{L^{t(\cdot)}(B(y,r))}\leq C
			\left\|v_{n,m}\right\|_{L^{\gamma(\cdot)}(B(y,r))}\leq E(r)C\left\|v_{n,m}\right\|_{M^{s(\cdot),p(\cdot)}(B(y,2r))}.
		\end{equation}
		Without loss of generality we can assume that $$\left\|u_n\right\|_{M^{s(\cdot),p(\cdot)}(X,d,\mu)} < \min\left\{1, \left(2\kappa_{s(\cdot),p(\cdot)}E(r)C\right)^{-1}, \left(2\kappa_{s(\cdot),p(\cdot)}\right)^{-1}\right\},$$ where $\kappa_{s(\cdot),p(\cdot)}\in (0,\infty)$ is a constant from the triangle inequality for $\left\|\cdot\right\|_{M^{s(\cdot),p(\cdot)}(X)}$ quasi-norm. Then, from inequality \eqref{ogr} it follows that $\left\|v_{n,m}\right\|_{L^{t(\cdot)}(B(y,r))}\leq 1$.
		Subsequently, applying H\"older inequality (Proposition \ref{holder1}) for the exponents 
		\begin{equation*}
			q_1:=\frac{p(t-q)}{t-p}, \hspace{5mm} q_2:=\frac{t(q-p)}{t-p}.
		\end{equation*}
		we get
		\begin{align}\label{holder}
			\int_{B(y,r)} \left|v_{n,m}(x)\right|^{q(x)}\mbox{d}\mu(x)\leq 2\left\| \left|v_{n,m}\right|^{q_1}\right\|_{L^{w(\cdot)}(B(y,r))} \left\| \left|v_{n,m}\right|^{q_2} \right\|_{L^{w'(\cdot)}(B(y,r))}.
		\end{align}
		Now, denote $$a_{n,m}(y):=\left\| \left|v_{n,m}\right|^{q_1} \right\|_{L^{w(\cdot)}(B(y,r))}.$$ From Lemma \ref{modularnorma} it follows that
		\begin{equation}\label{est}
			a_{n,m}(y) \leq \rho_{w(\cdot)}\left(\left|v_{n,m}\right|^{q_1}\chi_{B(y,r)}\right)^{\frac{1}{w^+}}=\left(\int_{B(y,r)} \left|v_{n,m}(x)\right|^{p(x)}\mbox{d}\mu(x)\right)^{\frac{1}{w^+}} \leq \tau_{n,m}^{\frac{1}{w^+}},
		\end{equation}
		where $\displaystyle \tau_{n,m}:=\sup_{y\in X} \int_{B(y,r)}\left|v_{n,m}(x)\right|^{p(x)}\mbox{d}\mu(x).$ 
		
		Using the above result, we shall prove that $\left\{u_n\right\}_{n=1}^{\infty}$ is a Cauchy sequence in $L^{q(\cdot)}(X,\mu)$. Let us define the following quantity
		\begin{equation*}
			b_{n,m}(y):=\left\| \left|v_{n,m}\right|^{q_2} \right\|_{L^{w'(\cdot)}(B(y,r))}.
		\end{equation*}
		By Lemma \ref{modularnorma}, inequality \eqref{ogr} and the facts that $\left\|v_{n,m}\right\|_{L^{t(\cdot)}(B(y,r))}\leq 1$ and $t_{B(y,r)}^- \geq q_{B(y,r)}^-$, we get
		\begin{align*}
			b_{n,m}(y)^{(w')^+}&\leq \int_{B(y,r)} \left|v_{n,m}(x)\right|^{t(x)}\mbox{d}\mu(x)\leq \left\|v_{n,m}\right\|_{L^{t(\cdot)}(B(y,r))}^{t^-_{B(y,r)}}\leq \left\|v_{n,m}\right\|_{L^{t(\cdot)}(B(y,r))}^{q^-_{B(y,r)}} \\& \leq \left(E(r)C \left\|v_{n,m}\right\|_{M^{s(\cdot),p(\cdot)}(B(y,2r))}\right)^{q^-_{B(y,r)}}.
		\end{align*}
		Since $\left\|u_n\right\|_{M^{s(\cdot),p(\cdot)}(X,d,\mu)} <  (2\kappa_{s(\cdot),p(\cdot)})^{-1}$, then there exists $g_n\in \mathcal{D}^{s(\cdot)}(u_n)$ such that \begin{equation*}\left\|u_n\right\|_{L^{p(\cdot)}(X,\mu)}+\left\|g_n\right\|_{L^{p(\cdot)}(X,\mu)}\leq (2\kappa_{s(\cdot),p(\cdot)})^{-1}.
		\end{equation*}
		Denote $g_{n,m}:=g_n+g_m$. Then, obviously $g_n+g_m \in \mathcal{D}^{s(\cdot)}(v_{n,m})$. Using \eqref{trololo}, Lemma \ref{modularnorma} and the fact that $\left\|v_{n,m}\right\|_{L^{p(\cdot)}(X,\mu)}\leq 1$, $\left\|g_{n,m}\right\|_{L^{p(\cdot)}(X,\mu)}\leq 1$ we get
		\begin{align*}
			b_{n,m}(y)&\leq \max\left\{1,E(r)C\right\}^{q^+/(w')^+} \left(\left\|v_{n,m}\right\|_{L^{p(\cdot)}(B(y,2r))}+\left\|g_{n,m}\right\|_{L^{p(\cdot)}(B(y,2r))}\right)^{q^-_{B(y,r)}/(w')^+}\\ &\leq V_r\left[ \left(\int_{B(y,2r)} \left|v_{n,m}(x)\right|^{p(x)}\right)^{\frac{q^-_{B(y,r)}}{(w')^+p^+_{B(y,2r)}}}\mbox{d}\mu(x) + \left(\int_{B(y,2r)} \left|g_{n,m}(x)\right|^{p(x)}\mbox{d}\mu(x)\right)^{\frac{q^-_{B(y,r)}}{(w')^+p^+_{B(y,2r)}}}\right] \\ & \leq V_r\left(\int_{B(y,2r)} \left|v_{n,m}(x)\right|^{p(x)}\mbox{d}\mu(x) + \int_{B(y,2r)} \left|g_{n,m}(x)\right|^{p(x)}\mbox{d}\mu(x)\right)
		\end{align*}
		where $V_r:=2^{q^+/(w')^+}\max\left\{1,E(r)C\right\}^{q^+/(w')^+}$. 
		Combining estimates for $a_{n,m}$, $b_{n,m}$ and \eqref{holder} we get
		\begin{align*}
			\int_{B(y,r)} \left|v_{n,m}(x)\right|^{q(x)}\mbox{d}\mu(x)& \leq 2 a_{n,m}(y) b_{n,m}(y)\\ &\leq 2V_r \tau_{n,m}^{\frac{1}{w^+}} \left(\int_{B(y,2r)} \left|v_{n,m}(x)\right|^{p(x)}\mbox{d}\mu(x) + \int_{B(y,2r)} \left|g_{n,m}(x)\right|^{p(x)}\mbox{d}\mu(x)\right).
		\end{align*}
		Thus, using Lemma \ref{pokryciowy} for $\delta:=2r$ we find positive constants $A$, $B$ and sequence $\left\{x_i\right\}_{i=1}^{\infty} \subseteq X$ such that 
		\begin{equation*}
			X=\bigcup_{i=1}^{\infty} B(x_i,r)
		\end{equation*}
		and every $x\in X$ belongs to at most $A\cdot2^B$ balls $B(x_i,2r)$. Hence,
		\begin{align*}
			\int_X \left|v_{n,m}(x)\right|^{q(x)}\mbox{d}\mu(x)&\leq \sum_{i=1}^{\infty} \int_{B(x_i,r)} \left|v_{n,m}(x)\right|^{q(x)}\mbox{d}\mu(x) \\ & \leq 2V_r \tau_{n,m}^{\frac{1}{w^+}}\sum_{i=1}^{\infty} \int_{B(x_i,2r)} \left(\left|v_{n,m}(x)\right|^{p(x)}+\left|g_{n,m}\right|^{p(x)}\right)\mbox{d}\mu(x) \\ & \leq 2\tau_{n,m}^{\frac{1}{w^+}}V_rA\cdot2^B\int_X\left(\left|v_{n,m}(x)\right|^{p(x)}+\left|g_{n,m}(x)\right|^{p(x)}\right)\mbox{d}\mu(x)\leq 4\tau_{n,m}^{\frac{1}{w^+}}V_r A\cdot2^B.
		\end{align*}
		Therefore, we obtain that $\left\{u_n\right\}_{n=1}^{\infty}$ satisfies Cauchy condition in $L^{q(\cdot)}(X,\mu)$. By completeness of $L^{q(\cdot)}(X,\mu)$ we deduce existence of $u\in L^{q(\cdot)}(X,\mu)$ such that $u_n \to u$ in $L^{q(\cdot)}(X,\mu)$, which completes the proof.
	\end{proof}
	
	Now we can return to the proof of Theorem \ref{zwartenieogr}.
	
	\begin{proof}[Proof of Theorem \ref{zwartenieogr}]
		Let $\left\{u_n\right\}_{n=1}^{\infty}$ be a bounded sequence in $M^{s(\cdot),p(\cdot)}_H(X,d,\mu)$. We will show that sequence $\left\{u_n\right\}_{n=1}^{\infty}$ satisfies assumptions of Proposition	\ref{cauchy}. Let $v_{n,m}:=u_n-u_m$. We know that the measure $\mu$, functions $u_n$, $p$ are $H$-invariant, so for every $y_1\in X$ and $y_2\in H(y_1)$ there holds equality
		\begin{equation*}
			\int_{B(y_1,r)} \left|v_{n,m}(x)\right|^{p(x)}\mbox{d}\mu(x)=\int_{B(y_2,r)} \left|v_{n,m}(x)\right|^{p(x)}\mbox{d}\mu(x).
		\end{equation*}
		Indeed, if $y_2 \in H(y_1)$, then there exists $f\in H$ such that $y_2=f(y_1)$. Since $f$ is measure-preserving isometry and $v_{n,m}$, $p$ are $H$-invariant, by change of variables formula \cite[Theorem 3.6.1]{Bogachev} we obtain
		\begin{align*}
			\int_{B(y_2,r)} \left|v_{n,m}(x)\right|^{p(x)}\mbox{d}\mu(x)&=\int_{f\left(B(y_1,r)\right)} \left|v_{n,m}(x)\right|^{p(x)}\mbox{d}\mu(x)=\int_{f\left(B(y_1,r)\right)} \left|v_{n,m}(x)\right|^{p(x)}\mbox{d}\left(f_\# \mu\right)(x)\\ &=\int_{B(y_1,r)} \left|\left(v_{n,m}\circ f\right)(x)\right|^{\left(p\circ f\right)(x)}\mbox{d}\mu(x)=\int_{B(y_1,r)} \left|v_{n,m}(x)\right|^{p(x)}\mbox{d}\mu(x),
		\end{align*}
		where $f_\# \mu$ is the pushforward measure.
		Subsequently, for every $y\in X$ we have
		\begin{equation*}
			M_H(y,r)\int_{B(y,r)}\left|v_{n,m}(x)\right|^{p(x)}\mbox{d}\mu(x) \leq \int_X \left|v_{n,m}(x)\right|^{p(x)}\mbox{d}\mu(x)\leq 1.
		\end{equation*}
		Using the above inequality and according to the assumption on $M_H(y,r)$ we obtain that for each $\varepsilon\in (0,\infty)$, there exists some finite $R_{\varepsilon}$ such that for every $n,m$ we have
		\begin{equation*}
			\sup_{y\in X \setminus B(x_0,R_{\varepsilon})} \int_{B(y,r)} \left|v_{n,m}(x)\right|^{p(x)}\mbox{d}\mu(x) \leq \left(\inf_{y\in X \setminus B(x_0,R_{\varepsilon})} M_H(y,r)\right)^{-1} \leq \varepsilon.
		\end{equation*}
		In particular, for every $k\in \mathbb N$ we can find $R_k\in (0,\infty)$ such that
		\begin{equation}\label{zawieranie22}
			\sup_{y\in X \setminus B(x_0,R_{k})} \int_{B(y,r)} \left|v_{n,m}(x)\right|^{p(x)}\mbox{d}\mu(x) \leq \frac{1}{k}.
		\end{equation}
		Additionally, we can take $R_k$ such that the sequence $\left\{R_k\right\}$ is non-decreasing. On the other hand,
		\begin{equation}\label{zawieranie2}
			\sup_{y\in B(x_0,R_{k})} \int_{B(y,r)} \left|v_{n,m}(x)\right|^{p(x)}\mbox{d}\mu(x) \leq \int_{B(x_0,R_{k}+r)} \left|v_{n,m}(x)\right|^{p(x)}\mbox{d}\mu(x).
		\end{equation} By Proposition \ref{operator} we know that the operator $F_{x_0,R_1+r}$ is compact from $M^{s(\cdot),p(\cdot)}(X,d,\mu)$ to $L^{p(\cdot)}(X,\mu)$. Thus, we can extract subsequence $\left\{u_n^1\right\}_{n=1}^{\infty}$ of the sequence $\left\{u_n\right\}_{n=1}^{\infty}$ which satisfies Cauchy condition in $L^{p(\cdot)}(B(x_0,R_1+r))$.
		
		Suppose that for some $k\in \mathbb N$ we have constructed the subsequence $\left\{u^k_n\right\}_{n=1}^{\infty}$ of $\left\{u_n\right\}_{n=1}^{\infty}$ satisfying Cauchy condition in $L^{p(\cdot)}(B(x_0,R_k+r))$. Hence, due to the compactness of $F_{x_0,R_{k+1}+r}$ we can extract subsequence $\left\{u^{k+1}_n\right\}_{n=1}^{\infty}$ of $\left\{u^k_n\right\}_{n=1}^{\infty}$ which is a Cauchy sequence in $L^{p(\cdot)}(B(x_0,R_{k+1}+r))$. Now, let $w_n:=u^n_n$. By the construction $\left\{w_n\right\}_{n=1}^{\infty}$ satisfies Cauchy condition in $L^{p(\cdot)}(B(x_0,R_{k}+r))$ for each $k$. Let us fix $\varepsilon\in (0,\infty)$ and take $k$ such that $1/k < \varepsilon$. Then, there exists $N$ such that for $n,m>N$
		\begin{equation*}
			\int_{B(x_0,R_k+r)} \left|w_n(x)-w_m(x)\right|^{p(x)}\mbox{d}\mu(x)\leq \varepsilon.
		\end{equation*}
		Finally, for $n,m>N$ from \eqref{zawieranie22}, \eqref{zawieranie2} we have
		\begin{align*}
			\sup_{y\in X} \int_{B(y,r)} \left|w_n(x)-w_m(x)\right|^{p(x)}\mbox{d}\mu(x)& \leq \int_{B(x_0,R_k+r)} \left|w_n(x)-w_m(x)\right|^{p(x)}\mbox{d}\mu(x)+\frac{1}{k} \leq 2\varepsilon.
		\end{align*}
		Hence, from Proposition \ref{cauchy} it follows that there exists $u\in L^{q(\cdot)}(X)$ such that $w_n \to u$ in $L^{q(\cdot)}(X)$, which completes the proof.
	\end{proof}
	
	 The following example shows that in general the embedding in Theorem \ref{zwartenieogr} is not compact if $q=p$.
	
	\begin{ex}
		Let $n\in \mathbb N$. Consider the metric measure space $(\mathbb R^n,d,\lambda)$, where $d=\left|\cdot - \cdot\right|$ denotes the Euclidean distance and $\lambda$ is the Lebesgue measure. Suppose that $H$ is any subgroup of the orthogonal group of $\mathbb R^n$. Moreover, suppose that $s\in \mathcal{P}_b(X)$, $p\in (0,\infty)$ are such that $ps^+ < n$ and $s^+ \leq 1$. Then, the embedding
		\begin{equation*}
			M^{s(\cdot),p}_H(\mathbb R^n,d,\lambda) \hookrightarrow L^{p}(\mathbb R^n,\lambda)
		\end{equation*}
		is not compact.
	\end{ex}
	
	\begin{proof}
		Let us define the function $u: \mathbb R^n \to \mathbb R$ as follows
		\begin{equation*}
			u(x):= \left\{\begin{array}{ll} 1, & \textnormal{ for } x\in B(0,1),\\ 2-\left|x\right|, & \textnormal{ for } x\in B(0,2)\setminus B(0,1),\\ 0, & \textnormal{ for } x\in \mathbb R^n \setminus B(0,2).\end{array}\right.
		\end{equation*}
		Now, for fixed $k\geq 1$ let we define $u_k:=k^{-\frac{n}{p}}u(x/k)$. By the virtue of Lemma \ref{cuttoff} applied for $\tau:=1/k$ and $\sigma=n/p$ we obtain that the function $g_k(x):=2k^{-\frac{n}{p}-s(x)}\chi_{B(0,2k)}(x)$ is a scalar $s(\cdot)$-gradient of $u_k$.
		Since obviously for every $k\geq 1$ functions $u_{k}, g_k \in L^{p}(\mathbb R^n,\lambda)$ and $u_k$ is $H$-invariant, we conclude that $u_{k} \in M^{s(\cdot),p}_H(\mathbb R^n ,d,\lambda)$. Moreover, we have 
		\begin{equation*}
			\left\|u_k\right\|_{L^p(\mathbb R^n)} \leq 2^{\frac{n}{p}}\omega_n^{\frac{1}{p}} \textnormal{ and } \left\|g_k\right\|_{L^p(\mathbb R^n)} \leq 2^{\frac{n}{p}+1} \omega_n^{\frac{1}{p}},
		\end{equation*} where $\omega_n$ denotes the Lebesgue measure of unit ball in $\mathbb R^n$. Therefore, $\left\{u_k\right\}_{k=1}^{\infty}$ is bounded in $M^{s(\cdot),p}_H(\mathbb R^n, d,\lambda)$. On the other hand, we shall show that this sequence has not subsequence convergent in $L^p(\mathbb R^n,\lambda)$. Assume the contrary, that there exists a subsequence $\left\{u_{k_m}\right\}_{m=1}^{\infty}$ convergent in $L^p(\mathbb R^n,\lambda)$ to some function $u$. Passing to the subsequence, we can assume that $u_{k_m} \stackrel{m\to \infty}{\longrightarrow} u$ almost everywhere in $\mathbb R^n$. Let us notice that $u_k \stackrel{k\to \infty}{\longrightarrow} 0$ everywhere in $\mathbb R^n$, which implies that $u=0$. However, for every $m\in \mathbb N$ we have
		\begin{equation*}
			\left\| u_{k_m} \right\|_{L^p(\mathbb R^n)}^p \geq \int_{B(0,k_m)} k_m^{-n} \mbox{d}x=\omega_n>0,
		\end{equation*}
		and hence $\left\{u_{k_m}\right\}_{m=1}^{\infty}$ cannot converge to $0$ in $L^p(\mathbb R^n,\lambda)$. Thus we got a contradiction and the embedding of $M^{s(\cdot),p}_H(\mathbb R^n,d,\lambda)$ into $L^p(\mathbb R^n, \lambda)$ is not compact.
	\end{proof}
	
	Applying Proposition \ref{pomiedzy} $(v)$, $(vii)$ we immediately obtain Berestycki-Lions type theorem for Triebel-Lizorkin and Besov spaces.
	
	\begin{tw}\label{zwartenieogrbesov}
		Suppose that $(X,d,\mu)$ is a geometrically doubling metric measure space where $\mu$ satisfies the condition  $\sup\left\{\mu(B(x,1)): x \in X \right\}<\infty$. Moreover, assume that $\mu$ is lower Ahlfors $Q(\cdot)$-regular for some $Q\in \mathcal{P}_b^{\log}(X)$. Suppose that $H\subseteq \textnormal{Iso}_{\mu}(X)$ is such that there exist $x_0\in X$ and $r\in (0,\infty)$ satisfying
		\begin{equation*}
			\lim_{R\to \infty} \inf_{x\in X \setminus B(x_0,R)} M_H(x,r)=\infty.
		\end{equation*}
		Let $p\in \mathcal{P}_b^{\log}(X) \cap P_H(X)$, $q\in \mathcal{P}(X)$ and $s\in \mathcal{P}_b^{\log}(X)$ be such that $s^+ \leq 1$ and $sp \ll Q$.
		\begin{enumerate}
			\item[(i)] We have the compact embedding
			\begin{equation*}
				M^{s(\cdot),H}_{p(\cdot),q(\cdot)}(X,d,\mu) \hookrightarrow \hookrightarrow L^{\beta(\cdot)}(X,\mu)
			\end{equation*}
			for every $\beta\in \mathcal{P}_b(X)$ satisfying $p \ll \beta \ll \gamma$, where $\displaystyle \gamma=Qp/(Q-sp)$.
			\item[(ii)] For all $t\in \mathcal{P}_b^{\log}(X)$ satisfying inequality $t\ll s$ we have the compact embedding
			\begin{equation*}
				N^{s(\cdot),H}_{p(\cdot),q(\cdot)}(X,d,\mu) \hookrightarrow \hookrightarrow L^{\beta(\cdot)}(X,\mu)
			\end{equation*}
			for every $\beta \in \mathcal{P}_b(X)$ satisfying $p \ll \beta \ll \sigma$, where $\displaystyle \sigma=Qp/(Q-tp)$.
		\end{enumerate}
	\end{tw}
	
	\subsubsection{Necessary conditions for compact embeddings of group invariant spaces} Now we investigate the necessary conditions for compact embeddings of $M^{s(\cdot),H}_{p(\cdot),q(\cdot)}$, $N^{s(\cdot),H}_{p(\cdot),q(\cdot)}$. The following theorem is main result of this part of our paper.
	
	\begin{tw}\label{konieczny}Suppose that $(X,d,\mu)$ is unbounded metric measure space where $\mu$ satisfies the condition  $\sup\left\{\mu(B(x,1)): x \in X \right\}<\infty$. Assume that $g(r):=\inf\left\{\mu(B(x,r)): x\in X\right\}>0$ for all $r\in (0,1)$. Let $H$ be a compact subgroup of $\textnormal{Iso}_{\mu}(X,d)$. Let $s,p,\beta\in \mathcal{P}_b(X)$, $q\in \mathcal{P}(X)$ be such that $s^+ < 1$ and $q^- <\infty$ or $s^+ \leq 1$ and $q^- =\infty$. Suppose that at least one of the following compact embeddings hold
		\begin{equation*}
			M^{s(\cdot),H}_{p(\cdot),q(\cdot)}(X,d,\mu) \hookrightarrow \hookrightarrow L^{\beta(\cdot)}(X,\mu), \hspace{15mm} N^{s(\cdot),H}_{p(\cdot),q(\cdot)}(X,d,\mu) \hookrightarrow \hookrightarrow L^{\beta(\cdot)}(X,\mu).
		\end{equation*}
		Then, for every $x_0\in X$ and $r\in (0,1/3)$
		\begin{equation*}
			\lim_{R \to \infty} \inf_{x\in X \setminus B(x_0,R)} M_H(x,r)=\infty.
		\end{equation*}
	\end{tw}
	
	\begin{proof}
		Assume the contrary, i.e. that there exist $x_0\in X$, $r\in (0,1/3)$ and sequence $\left\{x_k\right\}_{k=1}^{\infty} \subseteq X$ such that $d(x_k,x_0) \to \infty$ as $k\to \infty$ and
		\begin{equation*}
			0<\lambda:=\sup_{k\in \mathbb N} M_H(x_k,r)<\infty.
		\end{equation*}
		The above inequality implies that for every $k\in \mathbb N$ cardinality of any $2r$-separated subset of $H(x_k)$ is not greater than $\lambda$.
		
		For $k\in \mathbb N$ let $A_k \subseteq H(x_k)$ be the maximal $2r$-separated subset of $H(x_k)$. Then
		\begin{equation*}
			A_k:=\left\{y_1^k,\dots,y_{n_k}^k\right\},
		\end{equation*}
		for some $\left\{y_i^k\right\}_{i=1}^{n_k}\subseteq H(x_k)$ where $1\leq n_k \leq \lambda$ and
		\begin{equation}\label{zawieranie}
			H(x_k) \subseteq \bigcup_{i=1}^{n_k} B\left(y_i^k,2r\right).
		\end{equation}
		For $k\in \mathbb N$ we define the function $f_k: X \to \mathbb R$ as
		\begin{equation*}
			f_k(y):=\left\{ \begin{array}{lll} 1,& \textnormal{ for } y\in B\left(x_k,r/4\right),\\ \displaystyle 2-\frac{4}{r}d(y,x_k),& \textnormal{ for } y\in B\left(x_k,r/2\right) \setminus B\left(x_k,r/4\right),\\ 0,& \textnormal{ for } y\in X \setminus B\left(x_k,r/2\right).  \end{array}\right.
		\end{equation*}
		Obviously $f_k$ is $4/r$-Lipschitz and supported in $B\left(x_k,r/2\right).$
		Moreover, since group $H$ is compact, by \cite[Theorem 9.2.2 and Proposition 9.3.3]{Cohn} it carries left-invariant Haar measure $\nu_H$ such that $\nu_H(H)=1.$ Thus we define $\varphi_k: X \to \mathbb R$
		\begin{equation*}
			\varphi_k(y):=\int_H f_k\left(h^{-1}y\right) \mbox{d}\nu_H(h),
		\end{equation*}
		for $k\in \mathbb N$ and $y\in X.$ It is easy to notice that $\varphi_k$ is $H$-invariant and $4/r$-Lipschitz. In addition, for every $k\in \mathbb N$ function $\varphi_k$ is zero outside the set
		\begin{equation*}
			S_k:=\bigcup_{a\in H(x_k)} B\left(a,r/2\right).
		\end{equation*}
		Indeed, if some $y_0\in X$ is such that $d(y_0,a)\geq r/2$ for all $a\in H(x_k)$, then $d(h^{-1}y_0,x_k)\geq r/2$ for all $h\in H$. Hence $f_k(h^{-1}y_0)=0$ and $\varphi_k(y_0)=0.$ Hence by Theorem \ref{lipschitz1} we have $\varphi_k \in \dot{A}^{s(\cdot),H}_{p(\cdot),q(\cdot)}(X,d,\mu)$, where $A=M$ or $A=N$. Moreover, using \eqref{zawieranie} we get
		\begin{equation*}
			S_k \subseteq \bigcup_{i=1}^{n_k} B\left(y_i^k,3r\right).
		\end{equation*} Then, since $r\in (0,1/3)$, we have for every $k\in \mathbb N$
		\begin{equation}\label{miara}
			\mu(S_k) \leq \sum_{i=1}^{n_k} \mu(B(y_i^k,3r))\leq n_k\sup_{x\in X}\mu(B(x,1))\leq \lambda \xi<\infty,
		\end{equation}
		where
		\begin{equation*}
		\xi:=\sup_{x\in X} \mu(B(x,1)).
		\end{equation*} We show that $\left\{\varphi_k\right\}_{k=1}^{\infty}$ is bounded in $A^{s(\cdot),H}_{p(\cdot),q(\cdot)}(X,d,\mu),$ where $A\in \left\{M,N\right\}$. By Theorem \ref{lipschitz1} we obtain
		\begin{equation}\label{hajlasze}
			\left\| \varphi_k \right\|_{\dot{A}^{s(\cdot),H}_{p(\cdot),q(\cdot)}(X,d,\mu)} \leq C_{\textnormal{lip}} \left(\frac{4}{r}\right)^{s^+} \left\|\chi_{S_k}\right\|_{L^{p(\cdot)}(X,\mu)}\leq C_{\textnormal{lip}} \left(\frac{4}{r}\right)^{s^+} \max\left\{(\lambda \xi)^{\frac{1}{p^+}},(\lambda \xi)^{\frac{1}{p^-}}\right\},
		\end{equation}
		where $C_{\textnormal{lip}}$ is a constant from Theorem \ref{lipschitz1}.
		Moreover, applying \eqref{miara} and using the facts that $\varphi_k$ is zero outside of $S_k$ and $\left|\varphi_k(y)\right| \leq 1$ for all $y\in X$, we obtain
		\begin{align}\label{lp}
		\int_X \left|\varphi_k(y)\right|^{p(y)}\mbox{d}\mu(y)= \int_{S_k} \left|\varphi_k(y)\right|^{p(y)}\mbox{d}\mu(y)\leq \mu(S_k) \leq \lambda \xi<\infty 
		\end{align}
		and hence $\left\{\varphi_k\right\}_{k=1}^{\infty}$ is bounded in $L^{p(\cdot)}(X,\mu)$. Therefore, by \eqref{hajlasze} and \eqref{lp} we conclude that $\left\{\varphi_k\right\}_{k=1}^{\infty}$ is bounded in $A^{s(\cdot),H}_{p(\cdot),q(\cdot)}(X,d,\mu)$.
		
		 On the other hand, we will show that sequence $\left\{\varphi_k\right\}_{k=1}^{\infty}$ has not subsequence convergent in $L^{\beta(\cdot)}(X,\mu)$. Since $f_k \stackrel{k\to \infty}{\longrightarrow} 0$ almost everywhere in $X$, by dominated convergence theorem we also know that $\varphi_k \stackrel{k\to \infty}{\longrightarrow} 0$ almost everywhere in $X$. Therefore, it suffices to prove that there exists constant $c\in (0,\infty)$ such that $\rho_{\beta(\cdot)}\left(\varphi_k\right)\geq c>0.$ Notice that since $0\leq \varphi_k \leq 1$, we have
		 \begin{equation}\label{sigma}
		 \rho_{\beta(\cdot)}\left(\varphi_k\right)=\int_X \varphi_k(y)^{\beta(y)}\mbox{d}\mu(y) \geq \int_X \varphi_k(y)^{\sigma(y)}\mbox{d}\mu(y)=\rho_{\sigma(\cdot)}\left(\varphi_k\right),
		 \end{equation}
		 where $\sigma:=\max\left\{2,\beta\right\}$. By the fact that $\sigma^- >1$, H\"older inequality and the Fubini theorem we get
		\begin{align}\label{holderrr}
		\begin{split}	\left\| \varphi_k \right\|_{L^{\sigma(\cdot)}(X,\mu)}= \left\| \varphi_k \right\|_{L^{\sigma(\cdot)}(S_k)} &\geq \frac{1}{2\left\|1\right\|_{L^{\sigma'(\cdot)}(S_k)}} \int_{S_k} \varphi_k(y)\mbox{d}\mu(y)\\&= \frac{1}{2\left\|1\right\|_{L^{\sigma'(\cdot)}(S_k)}}\int_H \int_{S_k} f_k(h^{-1}y)\mbox{d}\mu(y)\mbox{d}\nu_H(h).
			\end{split}
		\end{align}
		Now, notice that for every $h
		\in H$ we have
		\begin{equation}\label{punktowe}
		 \int_{S_k} f_k\left(h^{-1}y\right)\mbox{d}\mu(y) \geq \int_{B\left(h(x_k),r/4\right)} f_k\left(h^{-1}y\right)\mbox{d}\mu(y)=\mu\left(B\left(h(x_k),r/4\right)\right) \geq g(r/4).
		\end{equation}
		Hence, using \eqref{punktowe} in \eqref{holderrr} and applying Lemma \ref{modularnorma} and \eqref{miara} we get
		\begin{align}\label{ostatnie2}
			\begin{split}\left\|\varphi_k \right\|_{L^{\sigma(\cdot)}(X,\mu)} &\geq \frac{1}{2\left\|1\right\|_{L^{\sigma'(\cdot)}(S_k)}} g(r/4)\geq \frac{1}{2\max\left\{ \mu(S_k)^{\frac{1}{(\sigma')_{S_k}^-}},\mu(S_k)^{\frac{1}{(\sigma')_{S_k}^+}} \right\}} g\left(r/4\right) \\&\geq \frac{1}{2 \max\left\{ \lambda \xi,1 \right\}^{\frac{1}{(\sigma')^-}}} g\left(r/4\right) =:d>0,
			\end{split}
		\end{align}
		since from assumptions $g(r/4)=\inf\left\{\mu\left(B(x,r/4)\right): x\in X\right\}>0$. Now, using \eqref{ostatnie2} in inequality \eqref{sigma} and Lemma \ref{modularnorma} we get
		\begin{equation*}
			\rho_{\beta(\cdot)}\left(\varphi_k\right) \geq \rho_{\sigma(\cdot)}\left(\varphi_k\right) \geq \min \left\{\left\| \varphi_k\right\|_{L^{\sigma(\cdot)}(X,\mu)}^{\sigma^+}, \left\|\varphi_k\right\|_{L^{\sigma(\cdot)}(X,\mu)}^{\sigma^-}\right\} \geq \min \left\{d^{\sigma^+}, d^{\sigma^-}\right\}:=c>0.
		\end{equation*} Therefore, the sequence $\left\{\varphi_k\right\}_{k=1}^{\infty}$ has not subsequence convergent in $L^{\beta(\cdot)}(X,\mu).$
	\end{proof}
	
	\subsubsection{Characterization of compact embeddings of group invariant spaces}
	Combining Theorem \ref{zwartenieogrbesov} with Theorem \ref{konieczny} we can obtain the following characterization of compact embeddings of group invariant Besov and Triebel-Lizorkin spaces.
	
	\begin{tw}\label{podsumowanie}
	Suppose that $(X,d,\mu)$ is unbounded geometrically doubling metric measure spaces where $\mu$ satisfies the condition $\sup \left\{ \mu(B(x,1)): x\in X\right\}<\infty$. Moreover, assume that $\mu$ is lower Ahlfors $Q(\cdot)$-regular for some $Q\in \mathcal{P}_b^{\log}(X)$. Let $H$ be the compact subgroup of $\textnormal{Iso}_{\mu}(X,d)$. Then, the following statements are equivalent.
	\begin{enumerate}
		\item[(i)] For every $x_0\in X$ and $r\in \left(0,1/3\right)$ we have
		\begin{equation*}
			\lim_{R\to \infty} \inf_{x\in X \setminus B(x_0,R)} M_H(x,r)=\infty.
		\end{equation*}
		\item[(ii)] For all $p\in \mathcal{P}_b^{\log}(X)\cap \mathcal{P}_H(X)$, $q\in \mathcal{P}(X)$, $s\in \mathcal{P}_b^{\log}(X)$ such that $s^+\leq 1$ and $sp \ll Q$ we have the compact embedding
		\begin{equation*}
			M^{s(\cdot),H}_{p(\cdot),q(\cdot)}(X,d,\mu) \hookrightarrow \hookrightarrow L^{\beta(\cdot)}(X,\mu)
		\end{equation*}
		for every $\beta\in \mathcal{P}_b(X)$ such that $p \ll \beta \ll \gamma$, where $\gamma=Qp/(Q-sp)$.
		\item[(iii)] There exist $s,p,\beta \in \mathcal{P}_b(X)$, $q\in \mathcal{P}(X)$ satisfying $s^+ <1 $ and $q^-<\infty$ or $s^+\leq 1$ and $q^-=\infty$ such that we have the compact embedding
		\begin{equation*}
			M^{s(\cdot),H}_{p(\cdot),q(\cdot)}(X,d,\mu) \hookrightarrow \hookrightarrow L^{\beta(\cdot)}(X,\mu).
		\end{equation*}
		\item[(iv)] For all $p\in \mathcal{P}_b^{\log}(X)\cap \mathcal{P}_H(X)$, $q\in \mathcal{P}(X)$, $s\in \mathcal{P}_b^{\log}(X)$ such that $s^+\leq 1$ and $sp \ll Q$ and for every $t\in \mathcal{P}_b^{\log}(X)$ such that $t \ll s$ we have the compact embedding
		\begin{equation*}
			N^{s(\cdot),H}_{p(\cdot),q(\cdot)}(X,d,\mu) \hookrightarrow \hookrightarrow L^{\beta(\cdot)}(X,\mu)
		\end{equation*}
		for every $\beta\in \mathcal{P}_b(X)$ such that $p \ll \beta \ll \sigma$, where $\sigma=Qp/(Q-tp)$.
		\item[(v)] There exist $s,p,\beta \in \mathcal{P}_b(X)$, $q\in \mathcal{P}(X)$ satisfying $s^+ <1 $ and $q^-<\infty$ or $s^+\leq 1$ and $q^-=\infty$ such that we have the compact embedding
		\begin{equation*}
			N^{s(\cdot),H}_{p(\cdot),q(\cdot)}(X,d,\mu) \hookrightarrow \hookrightarrow L^{\beta(\cdot)}(X,\mu).
		\end{equation*}
	\end{enumerate}
	\end{tw}
	
	We will finish the whole section with the following example, which illustrates the main results.
	
	\begin{ex}
	Consider metric measure space $(\mathbb C, d, \lambda)$, where $d=\left|\cdot - \cdot\right|$ denotes the Euclidean distance and $\lambda$ is the Lebesgue measure. Suppose that $H$ is a cyclic group generated by $T_{\theta}$, where $\theta \in \mathbb R$ and $T_{\theta}: \mathbb C \to \mathbb C$ is a transformation defined for $z\in \mathbb C$ as $T_{\theta}(z)=e^{2\pi \theta i}z.$ Moreover, assume that $q\in \mathcal{P}(\mathbb C)$, $p\in \mathcal{P}_b^{\log}(\mathbb C)\cap \mathcal{P}_H(\mathbb C)$, $t,s\in \mathcal{P}_b^{\log}(\mathbb C)$ are such that $sp \ll 2$, $t\ll s$ and $s^+ \leq 1$. Let us fix any $\beta, \tau \in \mathcal{P}_b(\mathbb C)$ such that $p \ll \beta \ll \gamma$ and $p \ll \tau \ll \sigma$, where $\gamma=Qp/(Q-sp)$, $\sigma=Qp/(Q-tp)$. Then, the embeddings
		\begin{equation}\label{zanurzeniaa}
			M^{s(\cdot),H}_{p(\cdot),q(\cdot)}(\mathbb C,d,\lambda) \hookrightarrow L^{\beta(\cdot)}(\mathbb C,\lambda), \hspace{10mm} N^{s(\cdot),H}_{p(\cdot),q(\cdot)}(\mathbb C,d,\lambda) \hookrightarrow L^{\tau(\cdot)}(\mathbb C,\lambda)
		\end{equation}
		are compact if and only if $\theta \in \mathbb R \setminus \mathbb Q$.
	\end{ex}
	
	\subsection{Compact embeddings between Besov and Triebel-Lizorkin spaces}
	
	In the last section we present some results concerning compact embeddings between spaces $A^{s(\cdot)}_{p(\cdot),q(\cdot)}$ and $\tilde{A}^{t(\cdot)}_{p(\cdot),r(\cdot)}$, where $A, \tilde{A} \in \left\{M,N\right\}$. Below we formulate two main results of this subsection.
	
	\begin{tw}\label{pomiedzyM}
		Let $(X,d,\mu)$ be a metric measure space. Let $p,s\in \mathcal{P}_b(X)$ and $q\in \mathcal{P}(X)$. Then, the following statements are equivalent.
		\begin{enumerate}
			\item[(i)] We have compact embedding
			\begin{equation*}
				M^{s(\cdot)}_{p(\cdot),q(\cdot)}(X,d,\mu) \hookrightarrow \hookrightarrow L^{p(\cdot)}(X,\mu).
			\end{equation*}
			\item[(ii)] For every $r\in \mathcal{P}(X)$ and $t\in \mathcal{P}_b(X)$ such that $t \ll s$ we have compact embeddings
			\begin{equation*}
				M^{s(\cdot)}_{p(\cdot),q(\cdot)}(X,d,\mu) \hookrightarrow \hookrightarrow M^{t(\cdot)}_{p(\cdot),r(\cdot)}(X,d,\mu), \hspace{10mm} 
				M^{s(\cdot)}_{p(\cdot),q(\cdot)}(X,d,\mu) \hookrightarrow \hookrightarrow N^{t(\cdot)}_{p(\cdot),r(\cdot)}(X,d,\mu).
			\end{equation*}
			\item[(iii)] There exist $r\in \mathcal{P}(X)$ and $t\in \mathcal{P}_b(X)$ such that $t \ll s$ and we have at least one of the following compact embeddings
			\begin{equation*}
				M^{s(\cdot)}_{p(\cdot),q(\cdot)}(X,d,\mu) \hookrightarrow \hookrightarrow M^{t(\cdot)}_{p(\cdot),r(\cdot)}(X,d,\mu), \hspace{10mm} 
				M^{s(\cdot)}_{p(\cdot),q(\cdot)}(X,d,\mu) \hookrightarrow \hookrightarrow N^{t(\cdot)}_{p(\cdot),r(\cdot)}(X,d,\mu).
			\end{equation*}
		\end{enumerate}
	\end{tw}
	
	\begin{tw}\label{pomiedzyN}
		Let $(X,d,\mu)$ be a metric measure space. Let $p,s\in \mathcal{P}_b(X)$ and $q\in \mathcal{P}(X)$. Then, the following statements are equivalent.
		\begin{enumerate}
			\item[(i)] We have compact embedding
			\begin{equation*}
				N^{s(\cdot)}_{p(\cdot),q(\cdot)}(X,d,\mu) \hookrightarrow \hookrightarrow L^{p(\cdot)}(X,\mu).
			\end{equation*}
			\item[(ii)] For every $r\in \mathcal{P}(X)$ and $t\in \mathcal{P}_b(X)$ such that $t \ll s$ we have compact embeddings
			\begin{equation*}
				N^{s(\cdot)}_{p(\cdot),q(\cdot)}(X,d,\mu) \hookrightarrow \hookrightarrow M^{t(\cdot)}_{p(\cdot),r(\cdot)}(X,d,\mu), \hspace{10mm} 
				N^{s(\cdot)}_{p(\cdot),q(\cdot)}(X,d,\mu) \hookrightarrow \hookrightarrow N^{t(\cdot)}_{p(\cdot),r(\cdot)}(X,d,\mu).
			\end{equation*}
			\item[(iii)] There exist $r\in \mathcal{P}(X)$ and $t\in \mathcal{P}_b(X)$ such that $t \ll s$ and we have at least one of the following compact embeddings
			\begin{equation*}
				N^{s(\cdot)}_{p(\cdot),q(\cdot)}(X,d,\mu) \hookrightarrow \hookrightarrow M^{t(\cdot)}_{p(\cdot),r(\cdot)}(X,d,\mu), \hspace{10mm} 
				N^{s(\cdot)}_{p(\cdot),q(\cdot)}(X,d,\mu) \hookrightarrow \hookrightarrow N^{t(\cdot)}_{p(\cdot),r(\cdot)}(X,d,\mu).
			\end{equation*}
		\end{enumerate}
	\end{tw}
	
	\begin{proof}[Proof of Theorems \ref{pomiedzyM} and \ref{pomiedzyN}]
	
	It suffices to prove the implication $(i) \implies (ii)$, as implications $(ii) \implies (iii)$ and $(iii) \implies (i)$ are straightforward. To prove it, we need the following proposition. 
	
	\begin{prop}\label{rownowaznosc}
	Let $(X,d,\mu)$ be a metric measure space and $p,s\in \mathcal{P}_b(X)$, $q,r\in \mathcal{P}(X)$. Let $\left\{u_n\right\}_{n=1}^{\infty}$ be a bounded sequence in $M^{s(\cdot)}_{p(\cdot),q(\cdot)}(X,d,\mu)$ or $N^{s(\cdot)}_{p(\cdot),q(\cdot)}(X,d,\mu)$. Assume that the sequence $\left\{u_n\right\}_{n=1}^{\infty}$ converges in $L^{p(\cdot)}(X,\mu)$, Then, it converges both in $M^{t(\cdot)}_{p(\cdot),r(\cdot)}(X,d,\mu)$ and in $N^{t(\cdot)}_{p(\cdot),r(\cdot)}(X,d,\mu)$ for every $t\in \mathcal{P}_b(X)$ such that $t \ll s$.	
	\end{prop}

	\begin{proof}
	Let $A\in \left\{M,N\right\}$ and assume that the sequence $\left\{u_n\right\}_{n=1}^{\infty}$ is bounded in $A^{s(\cdot)}_{p(\cdot),q(\cdot)}(X,d,\mu)$. Moreover, let us fix $t\in \mathcal{P}_b(X)$ such that $t \ll s$. Denote $\alpha:=3s/4+t/4$ and $\beta:=(s+t)/2$. Notice that it suffices to prove that $\left\{u_n\right\}_{n=1}^{\infty}$ converges in $A^{\beta(\cdot)}_{p(\cdot),q^-}(X,d,\mu)$. Indeed, by Proposition \ref{pomiedzy} $(viii)$, $(iv)$, $(iii)$ we get the following strings of embeddings	
	\begin{align*}
		M^{\beta(\cdot)}_{p(\cdot),q^-}(X,d,\mu) &\stackrel{(viii)}{\hookrightarrow} N^{\beta(\cdot)}_{p(\cdot),\infty}(X,d,\mu) \stackrel{(iv)}{\hookrightarrow} N^{t(\cdot)}_{p(\cdot),r(\cdot)}(X,d,\mu),\\
		N^{\beta(\cdot)}_{p(\cdot),q^-}(X,d,\mu) &\stackrel{(iv)}{\hookrightarrow} N^{s(\cdot)/4+3t(\cdot)/4}_{p(\cdot),p(\cdot)}(X,d,\mu)\stackrel{(iii)}{=}M^{s(\cdot)/4+3t(\cdot)/4}_{p(\cdot),p(\cdot)}(X,d,\mu) \stackrel{(iv)}{\hookrightarrow} M^{t(\cdot)}_{p(\cdot),r(\cdot)}(X,d,\mu)
	\end{align*}
	and
	\begin{align*}
	M^{\beta(\cdot)}_{p(\cdot),q^-}(X,d,\mu) \stackrel{(iv)}{\hookrightarrow} M^{t(\cdot)}_{p(\cdot),r(\cdot)}(X,d,\mu),\hspace{15mm}
	N^{\beta(\cdot)}_{p(\cdot),q^-}(X,d,\mu)  \stackrel{(iv)}{\hookrightarrow} N^{t(\cdot)}_{p(\cdot),r(\cdot)}(X,d,\mu).
	\end{align*}  Due to Proposition \ref{pomiedzy} $(iv)$ we have
	\begin{equation*}
	A^{s(\cdot)}_{p(\cdot),q(\cdot)}(X,d,\mu) \hookrightarrow A^{\alpha(\cdot)}_{p(\cdot),q^-}(X,d,\mu)\hookrightarrow A^{\beta(\cdot)}_{p(\cdot),q^-}(X,d,\mu).
	\end{equation*}
	Therefore, the sequence $\left\{u_n\right\}_{n=1}^{\infty}$ is bounded in $A^{\alpha(\cdot)}_{p(\cdot),q^-}(X,d,\mu)$ and it suffices to prove that it converges in $A^{\beta(\cdot)}_{p(\cdot),q^-}(X,d,\mu)$. Denote $\displaystyle C:=\sup_{n\in \mathbb N} \left\|u_n \right\|_{A^{\alpha(\cdot)}_{p(\cdot),q^-}(X,d,\mu)}$.
	
	 By the very definition of the $A^{\alpha(\cdot)}_{p(\cdot),q^-}$-quasi-norm, for each $n\in \mathbb N$ we can find $g_n\in \mathbb{D}^{\alpha(\cdot)}(u_n)$ such that
	\begin{equation*}
	\left\| g_n \right\| \leq \left\| u_n \right\|_{A^{\alpha(\cdot)}_{p(\cdot),q^-}(X,d,\mu)} \leq C,
	\end{equation*}
	where $\left\| \cdot\right\|$ denotes the $L^{p(\cdot)}(\ell^{q^-})$-quasi-norm if $A=M$ and $\ell^{q^-}(L^{p(\cdot)})$-quasi-norm if $A=N$.
	Now, notice that for every $n,m\in \mathbb N$ we have that $g_n+g_m \in \mathbb{D}^{\alpha(\cdot)}(u_n-u_m)$. From now, we shall use notation $g_{n,m}:=g_n+g_m$ and $u_{n,m}:=u_n-u_m$. Let us take arbitrary $\varepsilon \in (0,1)$ and let $E\subseteq X$ be such that $\mu(E)=0$ and 
	\begin{equation*}
	\left| u_{n,m}(x)-u_{n,m}(y)\right| \leq  d(x,y)^{\alpha(x)}g_{n,m,k}(x)+d(x,y)^{\alpha(y)}g_{n,m,k}(y) \textnormal{ for all } x,y\in X \setminus E
	\end{equation*}
	satisfying $2^{-k-1} \leq d(x,y) <2^{-k}$, where $k\in \mathbb Z$. Let $l \in \mathbb Z$ be such integer that $2^{-l} \leq \varepsilon < 2^{-l+1}$. Then, since $\varepsilon \in (0,1)$ we have $l\in (0,\infty)$. Now, if $k\in \mathbb Z$ is such that $k\geq l$, then for all $x,y\in X \setminus E$ with $2^{-k-1} \leq d(x,y) <2^{-k}$ we have
	\begin{align*}
	\left|u_{n,m}(x)-u_{n,m}(y)\right|& \leq d(x,y)^{\beta(x)}\varepsilon^{\alpha(x)-\beta(x)} g_{n,m,k}(x)+d(x,y)^{\beta(y)} \varepsilon^{\alpha(y)-\beta(y)} g_{n,m,k}(y)\\ & \leq d(x,y)^{\beta(x)}\varepsilon^{\left(\alpha-\beta\right)^-}g_{n,m,k}(x)+d(x,y)^{\beta(y)} \varepsilon^{\left(\alpha-\beta\right)^-} g_{n,m,k}(y).
	\end{align*}
	On the other hand, if $k<l$, then
	\begin{align*}
	\left| u_{n,m}(x) - u_{n,m}(y) \right| & \leq \left|u_{n,m}(x)\right| + \left|u_{n,m}(y)\right| \\ & \leq d(x,y)^{\beta(x)} 2^{(k+1)\beta(x)} \left|u_{n,m}(x)\right|+d(x,y)^{\beta(y)}2^{(k+1)\beta(y)}\left|u_{n,m}(y)\right|.
	\end{align*}
	Therefore, the sequence $h_{n,m}$ defined as
	\begin{equation*}
	h_{n,m,k}(x):= \left\{\begin{array}{ll} \varepsilon^{\left(\alpha-\beta\right)^-}g_{n,m,k}(x), & \textnormal{ if } k \geq l,\\ 2^{(k+1)\beta(x)}\left|u_{n,m}(x)\right|, & \textnormal{ if } k<l, \end{array} \right.
	\end{equation*}
	is a vector $\beta(\cdot)$-gradient of $u_{n,m}$. 
	
	Now we shall estimate quasi-norm of $h_{n,m}$ in the case when $A=M$. Let us fix $x\in X\setminus E$. If $q^-=\infty$, then
	\begin{align*}
		\left\| h_{n,m}(x) \right\|_{\ell^{q^-}} \leq \varepsilon^{(\alpha-\beta)^-}\left\| g_{n,m}(x)\right\|_{\ell^{\infty}}+\left(\frac{2}{\varepsilon}\right)^{\beta^+}\left| u_{n,m}(x) \right|
	\end{align*}
	If $q^-< \infty$, then we estimate
	\begin{align*}
	\left\|h_{n,m}(x)\right\|_{\ell^{q^-}}^{q^-} &= \sum_{k=-\infty}^{l-1} 2^{(k+1)\beta(x)q^-}\left|u_{n,m}(x)\right|^{q^-}+\sum_{k=l}^{\infty} \varepsilon^{(\alpha-\beta)^-q^-}g_{n,m,k}(x)^{q^-}\\ & =  \left|u_{n,m}(x)\right|^{q^-}\frac{2^{l\beta(x)q^-}}{1-2^{-\beta(x)q^-}} +\varepsilon^{(\alpha-\beta)^-q^-}\sum_{k=l}^{\infty} g_{n,m,k}(x)^{q^-} \\ & \leq \left|u_{n,m}(x)\right|^{q^-}Q(\beta,q)\left(\frac{2}{\varepsilon}\right)^{\beta^+ q^-}+\varepsilon^{(\alpha-\beta)^-q^-}\left\| g_{n,m}(x)\right\|_{\ell^{q^-}}^{q^-},
	\end{align*}
	where
	\begin{equation*}
	Q(\beta,q):=\frac{1}{1-2^{-\beta^-q^-}}
	\end{equation*} Taking in both cases the $L^{p(\cdot)}$-quasi-norm we obtain
	\begin{align}\label{cauchym}
	\left\| u_n-u_m \right\|_{\dot{M}^{\beta(\cdot)}_{p(\cdot),q^-}(X,d,\mu)} \leq \kappa_{p(\cdot)}\left(2^{\frac{1}{q^-}}Q(\beta,q)^{\frac{1}{q^-}} \left(\frac{2}{\varepsilon}\right)^{\beta^+} \left\|u_n-u_m\right\|_{L^{p(\cdot)}(X,\mu)}+2^{\frac{1}{q^-}+1}\varepsilon^{(\alpha-\beta)^-} C\right).
	\end{align}
	Here, in the case $q^-=\infty$, we apply conventions $1/\infty=0$ and $2^{-\infty}=0$.
	
	Now, we are going to find similar estimate for $\left\|h_{n,m}\right\|$ if $A=N$. For every $k\in \mathbb Z$ and $n,m\in \mathbb N$ we have
	\begin{equation*}
			\left\|h_{n,m,k}\right\|_{L^{p(\cdot)}(X,\mu)}:= \left\{\begin{array}{ll} \varepsilon^{\left(\alpha-\beta\right)^-}\left\|g_{n,m,k}\right\|_{L^{p(\cdot)}(X,\mu)}, & \textnormal{ if } k \geq l,\\ \left\|2^{(k+1)\beta(x)}u_{n,m}\right\|_{L^{p(\cdot)}(X,\mu)}, & \textnormal{ if } k<l. \end{array} \right.
	\end{equation*}
	Hence, if $q^-=\infty$, then
	\begin{equation*}
	\left\| h_{n,m} \right\|_{\ell^{q^-}\left(L^{p(\cdot)}(X,\mu)\right)} \leq \left(\frac{2}{\varepsilon}\right)^{\beta^+} \left\|u_{n,m}\right\|_{L^{p(\cdot)}(X,\mu)} + \varepsilon^{\left(\alpha-\beta\right)^-} \left\| g_{n,m} \right\|_{\ell^{q^-}\left(L^{p(\cdot)}(X,\mu)\right)}.
	\end{equation*}
	On the other hand, if $q^-< \infty$, then
	\begin{align*}
	&\left\| h_{n,m} \right\|_{\ell^{q^-}(L^{p(\cdot)}(X,\mu))}^{q^-}  \leq \sum_{k=-\infty}^{l-1} \left\|2^{(k+1)\beta(x)}u_{n,m}\right\|_{L^{p(\cdot)}(X,\mu)}^{q^-}+ \varepsilon^{(\alpha-\beta)^- q^-} \sum_{k=l}^{\infty} \left\|g_{n,m,k} \right\|_{L^{p(\cdot)}(X,\mu)}^{q^-} \\ & \hspace{8mm} \leq \sum_{k=-\infty}^{l-1} \max\left\{2^{(k+1)q^-\beta^+},2^{(k+1)q^-\beta^-}\right\}\left\|u_{n,m}\right\|_{L^{p(\cdot)}(X,\mu)}^{q^-}+ \varepsilon^{(\alpha-\beta)^- q^-} \sum_{k=l}^{\infty} \left\|g_{n,m,k} \right\|_{L^{p(\cdot)}(X,\mu)}^{q^-} \\ & \hspace{8mm}\leq \left\|u_{n,m}\right\|_{L^{p(\cdot)}(X,\mu)}^{q^-}\left[\sum_{k=-\infty}^{l-1} 2^{(k+1)q^-\beta^+} +\sum_{k=-\infty}^{l-1} 2^{(k+1)q^-\beta^-}\right]+ \varepsilon^{(\alpha-\beta)^- q^-} \sum_{k=l}^{\infty} \left\|g_{n,m,k} \right\|_{L^{p(\cdot)}(X,\mu)}^{q^-} \\ & \hspace{8mm}\leq 2Q(\beta,q) \left(\frac{2}{\varepsilon}\right)^{\beta^+ q^-} \left\|u_{n,m} \right\|_{L^{p(\cdot)}(X,\mu)}^{q^-} +	\varepsilon^{(\alpha-\beta)^- q^-} \left\| g_{n,m}\right\|_{\ell^{q^-}(L^{p(\cdot)}(X,\mu))}^{q^-}.
	\end{align*}
	Therefore,
	\begin{equation}\label{cauchyn}
	\left\|u_n -u_m\right\|_{\dot{N}^{\beta(\cdot)}_{p(\cdot),q^-}(X,d,\mu)} \leq \kappa_{p(\cdot)}\left(2^{\frac{1}{q^-}}2Q(\beta,q)^{\frac{1}{q^-}}\left(\frac{2}{\varepsilon}\right)^{\beta^+} \left\|u_n -u_m \right\|_{L^{p(\cdot)}(X,\mu)} +2^{\frac{1}{q^-}+1} \varepsilon^{(\alpha-\beta)^-}C\right),
	\end{equation}
	where we applied once more the conventions $1/\infty=0$ and $2^{-\infty}=0$. 
	
	Since $\left\{u_n\right\}$ is a Cauchy sequence in $L^{p(\cdot)}(X,\mu)$ and $\varepsilon \in (0,1)$ is arbitrary, by \eqref{cauchym}, \eqref{cauchyn} we deduce that $\left\{u_n\right\}$ is also a Cauchy sequence in $A^{\beta(\cdot)}_{p(\cdot),q^-}(X,d,\mu)$ and the proof is finished.
	\end{proof}
	
	We can get back to the proof of Theorems \ref{pomiedzyM}, \ref{pomiedzyN}. Assume that $(i)$ holds. Let $\left\{u_n\right\}_{n=1}^{\infty}$ be a sequence bounded in $A^{s(\cdot)}_{p(\cdot),q(\cdot)}(X,d,\mu)$, where $A \in \left\{M,N\right\}$. By assumption $(i)$ there exists subsequence $\left\{u_{n_k}\right\}_{k=1}^{\infty}$ which converges in $L^{p(\cdot)}(X,\mu)$. Now, thanks to Proposition \ref{rownowaznosc} we deduce that it also converges in $\tilde{A}^{t(\cdot)}_{p(\cdot),r(\cdot)}(X,d,\mu)$, where $\tilde{A} \in \left\{M,N\right\}$ and the proof is done.
	\end{proof}
	
	Combining together Theorems \ref{zwarteskonczone}, \ref{pomiedzyM} and \ref{pomiedzyN} we obtain the following result.
	
	\begin{corollary}
	Let $(X,d,\mu)$ be a metric measure space. Let $p,s\in \mathcal{P}_b(X)$ and $q\in \mathcal{P}(X)$. Assume that at least one of the following conditions hold:
	\begin{enumerate}
		\item[(i)] $(X,d)$ is totally bounded;
		\item[(ii)] measure $\mu$ is integrable, finite and $p\in \mathcal{P}_{\mu}^{\log}(X)$.
	\end{enumerate}
	Then, for every $r\in \mathcal{P}(X)$ and $t\in \mathcal{P}_b(X)$ such that $t \ll s$ we have compact embeddings
	\begin{align*}
	 M^{s(\cdot)}_{p(\cdot),q(\cdot)}(X,d,\mu) \hookrightarrow \hookrightarrow M^{t(\cdot)}_{p(\cdot),r(\cdot)}(X,d,\mu),&\hspace{10mm} M^{s(\cdot)}_{p(\cdot),q(\cdot)}(X,d,\mu) \hookrightarrow \hookrightarrow N^{t(\cdot)}_{p(\cdot),r(\cdot)}(X,d,\mu),\\
	 N^{s(\cdot)}_{p(\cdot),q(\cdot)}(X,d,\mu) \hookrightarrow \hookrightarrow M^{t(\cdot)}_{p(\cdot),r(\cdot)}(X,d,\mu),& \hspace{10mm} N^{s(\cdot)}_{p(\cdot),q(\cdot)}(X,d,\mu) \hookrightarrow \hookrightarrow N^{t(\cdot)}_{p(\cdot),r(\cdot)}(X,d,\mu).
	\end{align*}
	\end{corollary}
	
	\textbf{Acknowledgements:} The author is grateful to P. Górka for careful reading of the manuscript, valuable remarks, and essential support during the preparation of this paper.
	
	\smallskip
	{\small Micha{\l} Dymek}\\
	\small{Department of Mathematics and Information Sciences,}\\
	\small{Warsaw University of Technology,}\\
	\small{Pl. Politechniki 1, 00-661 Warsaw, Poland} \\
	{\tt michal.dymek.dokt@pw.edu.pl}\\
	

\begin{thebibliography}{99}
	\bibitem{maximal} {\sc T. Adamowicz, P. Harjulehto, P. H\"{a}st\"{o}}, Maximal operator in variable exponent Lebesgue spaces on unbounded quasimetric measure spaces, Math. Scand. (2015) 116: 5-22.
	
	\bibitem{Adams} {\sc R. A. Adams}, Sobolev Spaces, Pure and Applied Mathematics 65, Academic Press, New York-London, 1975.
	
	\bibitem{AH} {\sc A. Almeida, P. H\"ast\"o}, Besov spaces with variable smoothness and integrability, J. Funct. Anal. 258 (2010), 1628-1655.
	
	\bibitem{samko} {\sc A. Almeida, S. Samko}, Embeddings of variable Haj{\l}asz-Sobolev spaces into Hölder spaces of variable order,
	J. Math. Anal. Appl. 353 (2) (2009) 489–496.
		
	\bibitem{zanurzenia} {\sc R. Alvarado, M. Dymek, P. Górka, N. Karak}, Embeddings of variable Sobolev, Besov, and Triebel-Lizorkin spaces on metric measure spaces, Preprint (2026), arXiv:2603.19111.
	
	\bibitem{zwartestale} {\sc R. Alvarado, P. Górka, A. S{\l}abuszewski}, Compact embeddings of Sobolev, Besov, and Triebel-Lizorkin
	spaces, Journal of Differential Equations, Volume 446, 2025.
	
	\bibitem{Ambrosio} {\sc L. Ambrosio, P. Tilli}, Topics on analysis in metric spaces, Oxford Lecture Ser. Math. Appl. 25, Oxford Univ. Press, Oxford, 2004.
	
	\bibitem{baloch} {\sc Z. M. Balogh, A. Kristály}, Lions-type compactness and Rubik actions on the Heisenberg group, Calc. Var. Partial Differential Equations 48, 2013, 89–109.
	
	\bibitem{bandaliyev} {\sc R. Bandaliyev, P. Górka}, Relatively compact sets in variable-exponent Lebesgue spaces, Banach
	J. Math. Anal. 12 (2018), 331-346.
	
	\bibitem{Berestycki} {\sc H. Berestycki, P. L. Lions}, Existence of a ground state in nonlinear equations of the Klein-Gordon type. Variational inequalities and complementarity problems, In: Proc. Internat. School, Erice, 1978, Wiley, Chichester, 1980, 35–51.
	
	\bibitem{Bogachev} {\sc V. Bogachev}, Measure Theory, Springer Verlag, Berlin (2007).
	
	\bibitem{Cohn} {\sc  D. L. Cohn}, Measure Theory, Birkh\"auser, New York, 2013.
	
	\bibitem{Coifman} {\sc R. R. Coifman, G. Weiss}, Analyse harmonique non-commutative sur certains espaces homog`enes, Lecture Notes in Math., vol. 242, Springer-Verlag, Berlin, 1971. MR 58:17690.
	
	\bibitem{cruzuribe} {\sc D. Cruz-Uribe, A. Fiorenza}, Variable Lebesgue spaces: foundations and harmonic analysis, Springer Science Business Media, 2013.
	
	\bibitem{diening} {\sc L. Diening, P. Harjulehto, P. H\"ast\"o, M. R\r{u}\^{z}i\^{c}ka}, Lebesgue and Sobolev spaces with variable exponents, Lecture Notes in Math., Springer,
	Heidelberg, 2011.
	
	\bibitem{lizorkin} {\sc L. Diening, P. H\"at\"o, S. Roudenko}, Function spaces of variable smoothness and integrability, J. Funct. Anal. 256 (6)
	(2009) 1731–1768.
	
	\bibitem{nasze} {\sc M. Dymek, P. G\'orka}, Compactness in the spaces of variable integrability and summability, Math. Nachr. 296 (2023), 4317–4334.
	
	\bibitem{fan} {\sc X. Fan, Y. Zhao, D. Zhao}. Compact imbedding theorems with symmetry of Strauss Lions type for the space $W^{1,p(x)}(\Omega)$, J. Math. Anal. Appl. 255:1, 2001, 333–348.

	\bibitem{farkas} {\sc C. Farkas, A. Kristály, Á. Mester}, Compact Sobolev embeddings on non-compact manifolds via orbit expansions of isometry groups, Calc. Var. Partial Differ. Equ, (2021) 60:128.
	
	\bibitem{gacz} {\sc M. Gaczkowski, P. Górka}. Variable Haj{\l}asz-Sobolev spaces on compact metric spaces. Math. Slov. 67, 2017, 199–208.
	
	\bibitem{Gaczkowski}  {\sc M. Gaczkowski, P. G\'{o}rka, D. Pons}, Sobolev spaces with variable exponents on complete manifolds, J. Funct. Anal. 270, 2016, 1379--1415.
	
	\bibitem{krytyczny} {\sc M. Gaczkowski, P. G\'{o}rka, D. Pons}, Symmetry and compact embeddings for critical 
	exponents in metric-measure space, J. Differential Equations 269 (2020), 9819–-9837.
	
	\bibitem{ghorba} {\sc A. Ghorbanalizadeh, P. G\'orka}, Completeness and separability of the spaces of variable integrability and summability, Proc. Amer. Math.
	Soc. 149 (2021), 3873–3879.
	
	\bibitem{Górka} {\sc P. G\'orka}, Looking For Compactness In Sobolev Spaces On Noncompact etric Spaces, Ann. Acad. Sci. Fenn., Vol 43, 2018, 531-540.
	
	\bibitem{borel} {\sc P. G\'orka}, Separability of a metric space is equivalent to the existence of a borel measure, Am. Math.
	Mon. 128(1), 2020, 84–86.
	
	\bibitem{kostrzewa} {\sc P. G\'orka, T. Kostrzewa, E. G. Reyes}, Sobolev spaces on locally compact abelian	groups: compact embeddings and local spaces.- J. Funct. Spaces 2014, 2014, ID 404738.
	
	\bibitem{macios} {\sc P. Górka, A. Macios}. Almost everything you need to know about relatively compact	sets in variable Lebesgue spaces. J. Funct. Anal. 269:7, 2015, 1925–1949.
	
	\bibitem{nonlinear} {\sc P. G\'orka, A. S{\l}abuszewski}, Embeddings of the fractional Sobolev spaces on metric-measure spaces, Nonlinear Anal. 221 (2022), Paper No. 112867, 23 pp.
	
	\bibitem{holdery} {\sc P. G\'orka, C. Ruiz, M. Sanchiz},  Compact inclusions between variable H\"older spaces, J. Math. Anal. Appl. 538 (2024) 128328.
	
	\bibitem{hajlaszpierwszy} {\sc P. Hajlasz}, Sobolev spaces on an arbitrary metric space, Potential Anal. 5 (1996), no. 4, 403–415.

	
	\bibitem{hajlaszkoskela} {\sc P. Haj{\l}asz, P. Koskela}, Sobolev met Poincar´e. Mem. Amer. Math. Soc. 145 (2000), no. 688
	
	\bibitem{zmienneq} {\sc P. Harjulehto, P. H\"ast\"o, V. Latvala} Sobolev embeddings in metric measure spaces with variable dimension, Math. Z. 254, 591–609, (2006).
	
	\bibitem{hassel} {\sc B. Hasselblatt, A. Katok}, A First Course in Dynamics, Cambridge University Press (2003).
	
	
	\bibitem{zmiennepraca} {\sc Y. He, Q. Sun, C. Zhuo}, Pointwise characterizations of variable Besov and Triebel-Lizorkin spaces via Haj{\l}asz gradients. Fract. Calc. Appl. Anal. 27, 944–969 (2024).
	
	\bibitem{hebey} {\sc E. Hebey, M. Vaugon}, Sobolev spaces in the presence of symmetries.- J. Math. Pures Appl. 76, 1997, 859–881.
	
	\bibitem{kalamajska} {\sc A. Ka{\l}amajska}, On compactness of embedding for Sobolev spaces defined on metric spaces, Ann. Acad. Sci. Ser.	A. I Math. 24 (1999), 123–132.
	
	\bibitem{kempka} {\sc H. Kempka, J. Vybíral}, A note on the spaces of variable integrability and summability of Almeida and Hästö, Proc. Amer. Math. Soc. 141
	(2013), 3207–3212.
	
	\bibitem{koskela} {\sc P. Koskela, D. Yang, Y. Zhou}, Pointwise characterizations of Besov and Triebel-Lizorkin spaces and quasiconformal mappings, Adv. Math. 226 (2011), 3579-3621.
	
	\bibitem{Lions} {\sc P. L. Lions}, Symétrie et compacité dans les espaces de Sobolev, J. Funct. Anal. 49, 1982,	315–334.
	
	\bibitem{lukkainen} {\sc J. Luukkainen, E. Saksman}, Every complete doubling metric space carries a doubling measure, Proc. of the Amer. Math. Soc.,	Vol. 126, No. 2, 1998, pp. 531-534.
	
	\bibitem{skrzypczak2} {\sc L. Skrzypczak}, Rotation invariant subspaces of Besov and Triebel-Lizorkin space: compactness of embeddings, smoothness and decay properties, Revista Mat. Iberoamer. 18 (2002), 267-299.
	
	\bibitem{skrzypczak} {\sc L. Skrzypczak, C. Tintarev}, A geometric criterion for compactness of invariant sub
	spaces, Arch. Math. 101, 2013, 259–268.
	
	\bibitem{xu} {\sc J.-S. Xu}, Variable Besov and Triebel-Lizorkin spaces, Ann. Acad. Sci. Fenn. Math. 33 (2008) 511–522
	
	\bibitem{yang} {\sc D. Yang}, New characterizations of Hajlasz-Sobolev spaces on metric spaces, Sci. China Ser. A 46 (2003), 675-689
	\end{thebibliography}
\end{document}